\documentclass[10pt]{amsart}

\usepackage{amssymb,amsfonts,color,epsf}
\usepackage [cmtip,arrow]{xy}
\xyoption{all}
\usepackage {pb-diagram,pb-xy}
\usepackage{enumerate}
\usepackage{accents}
\usepackage[noautoscale]{youngtab}
\usepackage{subfigure}
\usepackage{lineno}
\ifx\pdfpageheight\undefined
\usepackage[dvips,colorlinks=true,linkcolor=blue,citecolor=red,%
urlcolor=green]{hyperref}
\usepackage[dvips]{graphicx}
\makeatletter
\edef\Gin@extensions{\Gin@extensions,.mps}
\makeatother
\else
\usepackage[pdftex]{graphicx}
\usepackage[bookmarksopen=false,pdftex=true,breaklinks=true,%
backref=page,pagebackref=true,plainpages=false,%
hyperindex=true,pdfstartview=FitH,colorlinks=true,%
pdfpagelabels=true,colorlinks=true,linkcolor=blue,%
citecolor=red, urlcolor=darkgray, hypertexnames=false%
]%
{hyperref}
\fi
\usepackage{bm}

\usepackage{amssymb,color,epsf}
\usepackage{mathtools}
\usepackage{enumerate}
\usepackage{float}

\usepackage{amsmath}
\usepackage{tabularx}
\usepackage{color, colortbl}
\usepackage{url}

\DeclareMathAlphabet{\mathpzc}{OT1}{pzc}{m}{it}
\usepackage {pb-diagram,pb-xy}
\usepackage[mathscr]{eucal}
\usepackage[utf8]{inputenc}
\usepackage[T1]{fontenc}
\usepackage[english]{babel}
\usepackage{bm}
\usepackage{tabularx}
\usepackage{ltablex}

\usepackage{bm}

\usepackage{amsmath,amssymb,amsfonts}
\newtheorem{theorem}{Theorem}
\newtheorem{lemma}{Lemma}[section]
\newtheorem{corollary}{Corollary}
\newtheorem{proposition}{Proposition}[section]

\newtheorem*{claim*}{Claim}

\newtheorem*{theorem*}{Theorem}
\newtheorem*{corollary*}{Corollary}
\theoremstyle{definition}
\newtheorem{definition}{Definition}[section]

\newtheorem{example}{Example}[section]

\newtheorem{notation}{Notation}[section]

\theoremstyle{remark}
\newtheorem{remark}{Remark}

\newtheorem{thmx}{Theorem}
\definecolor{DarkBlue}{rgb}{0,0.1,0.55}

\numberwithin{equation}{section}

\newcommand {\hide}[1]{}
\newcommand {\sign} {\mathrm{sign}}

\newcommand {\junk}[1]{}

\newcommand {\R} {\mathrm{R}}

\newcommand {\D}     {\mbox{\rm D}}

\newcommand {\C}     {\mathrm{C}}

\newcommand {\Z}  {\mathbb{Z}}

\newcommand {\ZZ} {{\rm Z}}
\newcommand {\RR} {{\mathcal R}}

\newcommand {\la}   {{\langle}}
\newcommand {\ra}   {{\rangle}}
\newcommand {\eps} {{\varepsilon}}
\newcommand {\E} {{\rm Ext}}

\newcommand {\Id} {\mbox{\rm Id}}

\newcommand {\PP}     {\mathbb{P}} %projective space

\newcommand{\card}{\mathrm{card}}
\newcommand{\rank}{\mathrm{rank}}

\def\addots{\mathinner{\mkern1mu
		\raise1pt\vbox{\kern7pt\hbox{.}}
		\mkern2mu\raise4pt\hbox{.}\mkern2mu
		\raise7pt\hbox{.}\mkern1mu}}

\newcommand{\HH}  {\mbox{\rm H}}

\newcommand{\x}{\mathbf{x}}

\newcommand{\size}{\mathrm{size}}

\newcommand{\Cc}{\mathrm{Cc}}

\usepackage{algorithm}% http://ctan.org/pkg/algorithms
\usepackage{algpseudocode}
\algnewcommand\algorithmicinput{\textbf{Input:}}
\algnewcommand\INPUT{\item[\algorithmicinput]}
\algnewcommand\algorithmicoutput{\textbf{Output:}}
\algnewcommand\OUTPUT{\item[\algorithmicoutput]}

\algnewcommand\algorithmicproc{\textbf{Procedure:}}
\algnewcommand\PROCEDURE{\item[\algorithmicproc]}
\algnewcommand\algorithmiccomplexity{\textbf{Complexity:}}
\algnewcommand\COMPLEXITY{\item[\algorithmiccomplexity]}
\usepackage{multicol}
\usepackage{tikz-cd}
\usepackage{tabu}
\usepackage{caption}
\usepackage{booktabs}
\usepackage{makecell}
\newcolumntype{P}[1]{>{\centering\arraybackslash}p{#1}}

\newcommand{\End}{\mathrm{End}}

\newcommand{\sgn}{\mathrm{sgn}}

\newcommand{\Sing}{\mathrm{Sing}}

\newcommand{\mult}{\mathrm{mult}}

\newcommand{\mmrank}{\mathrm{mmrank}}

\newcommand{\Gr}{\mathrm{Gr}}

\newcommand{\Rz}{\R'\langle\zeta\rangle}
\newcommand{\Rt}{\widetilde{\R}}
\newcommand{\bW}{\mathbf{W}}
\newcommand{\bi}{\mathbf{i}}
\makeatletter % for internal macros with @
\renewcommand\p@enumii{}
\makeatother
\title
[Bounds on zero-nonzero patterns and sign conditions]
{
Bounds on the realizations of zero-nonzero patterns and 
sign conditions of polynomials restricted to varieties
and applications 
}

\author{Saugata Basu}
\address{Department of Mathematics,
Purdue University, West Lafayette, IN 47906, U.S.A.}
\email{sbasu@math.purdue.edu}
\urladdr{www.math.purdue.edu/~sbasu}
\author{Laxmi Parida}
\address{
IBM T.J. Watson Research Center,
Yorktown Heights, NY 10598.
}
\email{parida@us.ibm.com}
\urladdr{research.ibm.com/people/laxmi-parida}
\begin{document}
%%\linenumbers

\begin{abstract}
We obtain upper bounds, independent of the ambient dimension, for the number of realizable
zero-nonzero patterns and (over ordered fields) sign conditions of a finite family of polynomials
$\mathcal P$ restricted to an algebraic subset $V$ of affine or projective space.
The bounds depend only on $\card(\mathcal P)$ and the degrees of the polynomials in $\mathcal P$,
together with $\deg(V)$ and $\dim(V)$, and not on the dimension of the space in which $V$ is embedded.
This feature is particularly useful when $V$ has small intrinsic dimension but is presented in a very
high-dimensional ambient space.

We describe several applications.  First, we extend existing results on bounding the
$\varepsilon$-entropy of real algebraic varieties.  Second, we derive lower bounds (in terms of the number
of connected components) for membership testing in semi-algebraic sets in the algebraic computation tree
model.  Finally, motivated by quantum complexity theory, we introduce additive and multiplicative notions
of \emph{relative rank} in finite-dimensional vector spaces and algebras with respect to a fixed algebraic
subset, generalizing the classical notion of tensor rank.  We prove a general lower bound on the maximum
relative rank of finite subsets with respect to algebraic sets of bounded degree and dimension that is
again independent of the ambient dimension.  
As an illustration, we obtain a quantum analog of Shannon's
classical lower bound.
We prove that
if \(\Delta_n\) is an algebraic family of allowed gates with
$$ \dim \Delta_n=\operatorname{poly}(n), \qquad \log\deg\Delta_n=\operatorname{poly}(n), $$
then some Boolean functions require \(2^n/\operatorname{poly}(n)\) gates from \(\Delta_n\), even when individual gates in \(\Delta_n\) may be highly nonlocal.
\end{abstract}

\maketitle
\tableofcontents

\section{Introduction}
Let $k$ be a field. For $x \in k$, we will denote 
\[
\sgn(x) = \begin{cases}
                0 \text{ if } x=0, \\
                1 \text{ else},
        \end{cases}
\]
and for any set of polynomials $\mathcal{P} \subset k[X_1,\ldots,X_N]$
we will call an element of $\{0,1\}^{\mathcal{P}}$ to be a 
\emph{zero-nonzero pattern} on $\mathcal{P}$.

If $k$ is an ordered field we will denote for $x \in k$, 
\[
\sign(x) = \begin{cases}
                0 \text{ if } x=0, \\
                1 \text{ if } x>0, \\
                -1 \text{ if } x < 0,
        \end{cases}
\]
and for any set of polynomials $\mathcal{P} \subset k[X_1,\ldots,X_N]$
we will call an element of $\{0,1,-1\}^{\mathcal{P}}$ to be a 
\emph{sign condition} on $\mathcal{P}$.

For any finite set $\mathcal{P}$ of polynomials in $k[X_1,\ldots,X_N]$, $V\subset k^N$, and
$\pi\in \{0,1\}^{\mathcal{P}}$, we denote by 
\[
\RR(\pi,V) = \{x \in V \mid \sgn
(P(x)) = \pi(P), P \in \mathcal{P}\}.
\]

In the case $k$ is an ordered field,
$S \subset k^N$, and
$\sigma \in \{0,1,-1\}^{\mathcal{P}}$, we denote by 
\[
\RR(\sigma,S) = \{x \in S  \mid \sign(P(x)) = \sigma(P), P \in \mathcal{P}\}.
\]

\subsection{Background and prior work}
Bounds on the number (or other topological invariants of the realizations like the number of connected components etc.) 
of \emph{realizable} zero-nonzero patterns or sign conditions of finite sets of polynomials have had many important applications in mathematics and theoretical computer science. 
To pick a few notable ones: 
they play a key role in bounding the number of point configurations  as well as combinatorial types of polytopes  in $\mathbb{R}^d$ \cite{GP1986}, in bounding the number of geometric permutations induced by transversals to convex subsets of $\mathbb{R}^d$ \cite{GPW1996}, bounding the number (also called the speed) of algebraically defined classes of graphs \cite{Alon-Pach-et-al,Sauermann}, in recent applications of the polynomial method
to incidence combinatorics \cite{Basu-Sombra, Walsh1, Walsh2, Tao-survey, Sheffer-book, Guth-book},
in proving lower bounds for testing  membership in semi-algebraic sets using algebraic computation trees \cite{Ben-Or,Yao,GV2017},
in proving lower bounds on projective dimensions of graphs and span programs
\cite{RBG01},
bounding the $\eps$-entropy of and volumes of tubes around algebraic varieties
\cite{Basu-Lerario2023,Zhang-Kileel2023},
and a host of other applications (including some new ones introduced in this paper).
The following two results may be considered as starting points.

\begin{thmx}\cite{RBG01,Jeronimo-Sabia}
\label{thmx:RBG}
    Let $k$ be a field and $\mathcal{P} \subset k[X_1,\ldots,X_N]_{\leq d}$ be a finite set of polynomials of degree at most $d$. Then, the number of realizable zero-nonzero patterns on $\mathcal{P}$
    (i.e. elements of $\{0,1\}^{\mathcal{P}}$) is bounded by 
    $(O(s d))^N$, where $s = \card(\mathcal{P})$. (We say a zero-nonzero pattern $\pi \in \{0,1\}^{\mathcal{P}}$ is realizable if there exists $x \in k^N$, such that $\sgn(P(x)) = \pi(P)$ for all 
    $P \in \mathcal{P}$.)
\end{thmx}

\begin{thmx}\cite{Warren}
\label{thmx:Warren}
    Let $k$ be an ordered field and $\mathcal{P} \subset k[X_1,\ldots,X_N]_{\leq d}$ be a finite set of polynomials of degree at most $d$. Then, the number realizable sign conditions
    on $\mathcal{P}$
    is bounded by 
    $(O(s d))^N$, where $s = \card(\mathcal{P})$.
    (We say a sign condition $\sigma \in \{0,1,-1\}^{\mathcal{P}}$ is realizable if there exists $x \in k^N$, such that $\sign(P(x))  = \sigma(P)$ for all  $P \in \mathcal{P}$.)
\end{thmx}

The above theorems have been generalized in many different directions.
For example, if $k = \R$ is a real closed field (for example the field $\mathbb{R}$ of real numbers), it is possible to obtain bounds on the higher Betti numbers of the realizations of sign conditions (see Theorem~\ref{thmx:BPR} below). 
Motivated by applications in discrete geometry (bounding the number of geometric permutations, as well as incidences), there was a need for more precise bounds on the number of realizable sign conditions: for example, upper bounds on the number of realizable sign conditions or even on the number of different connected components of the realizations restricted to an algebraic subset of the ambient space which may be of much smaller dimension (see for example \cite{GPW1996} for one such application where the dimension of the algebraic subset plays a key role).

Before proceeding further we recall some definitions from algebraic geometry 
and introduce some useful notations.
\subsection{
Basic definitions and notation
}
Let $k$ denote an algebraically closed field. 
An \emph{algebraic subset}  $V$ of $k^N$ 
\footnote{
In the scheme-theoretic language, $k^N$ is identified with the set of closed $k$-points of the $n$-dimensional affine scheme $\mathbb{A}_k^N$ \cite{Hartshorne}. It comes equipped with the Zariski topology. 
However, we will avoid schematic language in this paper and use $k^n$ and $\mathbb{A}_k^N$ interchangeably.  
Similar, remark applies to the projective scheme $\PP_k^N$ as well. 
}
is the set of common zeros of a set of polynomials in $k[X_1,\ldots,X_N]$.

An algebraic subset admits a decomposition (in the Zariski topology) into a finite union of irreducible ones which are called \emph{varieties}. 
The \emph{dimension} of a variety $V$, denoted $\dim V$,  
is the Krull dimension of its
coordinate ring $k[V] = k[X_1,\ldots,X_n]/I(V)$,
where $I(V)$ is the defining ideal of $V$ i.e. $I(V) =  \{f \in k[X_1,\ldots,X_N] \;\mid\; f(x) = 0, x \in V \}$.
The \emph{degree of $V$}, denoted $\deg V$, 
is  equal to the number of points in the intersection $V \cap L$ for a generically chosen affine subspace $L$ of dimension  $N-\dim V$ in $k^N$.

Finally, we define the degree, $\deg(V)$, 
of an algebraic set $V$ to be  the sum of the degrees of its irreducible components, and 
the dimension, $\dim(V)$, of $V$ to be the maximum of the dimensions of its irreducible components.

A subset of $S \subset k^N$ is called a \emph{constructible subset}
if it is defined by a Boolean formula with atoms of the form
$P = 0$ for $P \in k[X_1,\ldots,X_N]$. 
A subset of $S \subset \PP_k^N$ is called a \emph{constructible subset}
if it is defined by a Boolean formula with atoms of the form
$P = 0$ for $P \in k[X_0,X_1,\ldots,X_N]$, with $P$ homogeneous. 

If $\R$ is a real closed field (for example, $\mathbb{R}$),
a subset of $S \subset \R^N$ is called a \emph{semi-algebraic subset}
if it is defined by a Boolean formula with atoms of the form
$P = 0, P > 0$ for $P \in \R[X_1,\ldots,X_N]$. 
We call a subset of $\mathbb{A}_\R^N$ or $\PP_\R^N$ to be $\mathcal{P}$-semi-algebraic,
for a finite set of polynomials $\mathcal{P}$, if every $P$ in the 
Boolean formula defining the set belongs to $\mathcal{P}$.

If $\C = \R[i]$ 
(which is the algebraic closure of $\R$), then a constructible subset of 
$\C^N$ is also semi-algebraic subset of $\R^{2N} = \C^N$.

\begin{notation}
In case, $k = \C = \R[i]$, where $\R$ is a real closed field, 
for any locally closed semi-algebraic set $S$ we denote by 
    \[
    b_i(S) = \dim_{\mathbb{Z}/2\mathbb{Z}} \HH_i(S,\mathbb{Z}/2\mathbb{Z}),
    \]
    where $\HH_i(S,\mathbb{Z}/2\mathbb{Z})$ denotes the $i$-th homology group of $S$ with coefficients in $\mathbb{Z}/2\mathbb{Z}$ (see for example, \cite[Chapter 6, Section 3]{BPRbook2}) and call $b_i(S)$
    the $i$-th Betti number of $S$ (with coefficients in $\mathbb{Z}/2\mathbb{Z}$). 
    
    Note that $b_0(S)$ is just the number of semi-algebraically connected components of $S$.
    \footnote{We say ``semi-algebraically connected'' rather than connected since $\R$ is assumed to be an arbitrary real closed field, possibly non-archimedean. ``Semi-algebraically connected''  means ``semi-algebraically path connected'' (see for example \cite[Chapter 3, Section 2]{BPRbook2}), i.e. every two points in the set can be joined by a semi-algebraic path whose image is contained in the set. If $\R = \mathbb{R}$, then being semi-algebraically connected and being connected are equivalent properties of semi-algebraic sets.}
\end{notation}

We can now state several generalizations of the two basic theorems mentioned before.
The following theorem
extends Theorem~\ref{thmx:Warren} in 
several directions.

\begin{thmx}\cite{BPR02}
\label{thmx:BPR}
Let $\R$ be a real closed field and $V \subset \R^N$ be an algebraic subset defined by polynomials of degrees bounded by $d$ with $\dim_\R V = p$, and $\mathcal{P} \subset \R[X_1,\ldots,X_N]_{\leq d}$. Then, for each $i, 0 \leq i \leq p$, the sum of the $i$-th Betti numbers of the realizations in $V$ of all sign conditions on $\mathcal{P}$ is bounded by 
\[
s^{p-i} \cdot O(d)^N,
\]
where $s = \card(\mathcal{P})$.
\end{thmx}

However, the above result is still insufficient for certain modern applications involving applications of the polynomial partitioning technique, where one needs more refined results where the dependence
of the bound on the degree of variety $V$ and degrees of the polynomials
in $\mathcal{P}$ are separated out. 

Suppose that we have a sequence of algebraic sets $V_0,\ldots,V_\ell \subset \R^N$
cut out by a sequence of real polynomials of degrees
$d_1,d_2, \ldots,d_\ell$, with real dimensions
$p_1 \geq p_2 \geq \cdots \geq p_\ell$,
and the degrees of the polynomials in $\mathcal{P}$ are bounded by $d \geq d_\ell$.

The following result in this direction is proved in \cite{Barone-Basu2},
and has been applied in the context of incidence geometry \cite{Basu-Sombra, Walsh1, Walsh2}.

\begin{thmx} \cite[Theorem 4]{Barone-Basu2}
\label{thmx:BB}
Suppose that
\[ 2 \leq d_1 \leq d_2 \leq \frac{1}{N + 1} d_3 \leq \frac{1}{(N + 1)^2} d_4
   \leq \cdots \leq \frac{1}{(N + 1)^{\ell - 2}} d_{\ell},\]
   and $d_\ell \leq \frac{1}{N+1} d$.
   Then, the number of semi-algebraically connected components of the realizations of all sign conditions on $\mathcal{P}$ restricted to $V = V_\ell$ is bounded by
 \[
  O (N)^{2 N} (s d)^{p_{\ell}}  \left( \prod_{1 \leq j \leq \ell}
    d_j^{p_{j - 1} - p_j} \right) .
\]
\end{thmx}

The following related result which has applications in incidence geometry via the polynomial partitioning technique (see for example \cite{Guth-book,Sheffer-book} was obtained by M. Walsh \cite{Walsh2}. We denote by $\ZZ(P,V)$ the set of zeros of a polynomial restricted to a
set $V$. 
If $V = \R^N,\C^N$, we will sometimes omit $V$ and denote $\ZZ(P) = \ZZ(P,V)$.

\begin{thmx} \cite[Theorem 2.12]{Walsh2}
\label{thmx:Walsh}
Let $\R$ be a real closed field, $\C = \R[i]$ the algebraic closure of $\R$, and $V \subset \C^N$ be an algebraic variety with $\dim V = p$, and 
$P \in \R[X_1,\ldots,X_n]$.
Then the number of semi-algebraically connected components $C$  of 
$\R^N - \ZZ(P,\R^N))$ such that $C \cap V(\R) \neq \emptyset$ is bounded by
$ C(N) \cdot \deg(V) \cdot \deg(P)^p$ where the constant $C(N)$ depends on $N$
(in an unspecified way), and $V(\R)$ denotes the $\R$-points of $V$.
\end{thmx}

\begin{remark}
    Note that the quantity being bounded in Theorem~\ref{thmx:Walsh} is 
    a priori smaller than the number of semi-algebraically connected components of $V(\R) - \ZZ(P,\R^N)$. In particular, one cannot derive any useful bound on the number of semi-algebraically connected components of $V(\R)$ itself from Theorem~\ref{thmx:Walsh} (by letting $P$  to be a non-zero constant polynomial).
\end{remark}

\begin{remark}
Notice also that 
while Theorem~\ref{thmx:BPR} has better dependence in the combinatorial part (the part depending on $s$) on the dimension of the variety $V$, the algebraic part (the part depending on the degree) is oblivious 
of the dimension of the variety and depends exponentially on $N$ the dimension of the ambient space. The bounds in Theorems~\ref{thmx:BB} and \ref{thmx:Walsh} are better but they still depend on $N$ (through the quantity $C = C(N)$, presumably exponential,  in the case of Theorem~\ref{thmx:Walsh}). 
\end{remark}

In many applications, the dimension of $V$ may be much smaller than $N$.
For instance, consider the Pl\"{u}cker embedding of the Grassmannian variety
$\mathrm{Gr}(m,M)$ of $m$-dimensional subspaces in $k^{M}$ in the 
$\binom{M}{m} -1$ dimensional projective space. It is well known that 
$\dim \mathrm{Gr}(m,M) = m(M-m)$. Thus, if $m = \lceil M/2 \rceil$,
while the dimension of the variety $V = \mathrm{Gr}(m,M)$ is polynomial in $M$, the dimension $N$ of the ambient projective space is exponentially large in $M$. In such a situation, the bounds in Theorems~\ref{thmx:BB}
and \ref{thmx:Walsh} which have exponential dependence on $N$, would produce a bound doubly exponential in $M$.

Laszlo and Viterbo \cite{Laszlo-Viterbo} proved the following:

\begin{thmx} \cite[Theorem 1.2]{Laszlo-Viterbo}
\label{thmx:LV}
    The sum of all $\Z/2\Z$-Betti numbers of $V(\R)$, where $V$ is a non-singular real variety of degree $D$ and dimension $p$, is bounded by 
    \[
    2^{p^2+2} D^{p+1}.
    \]
\end{thmx}

\begin{remark}
\label{rem:LV0}
    The proof of Theorem~\ref{thmx:LV} relies on sophisticated calculation of characteristic classes of the underlying smooth manifold of $V$ and hence cannot be applied to singular varieties.
\end{remark}

\begin{remark}
\label{rem:LV}
    The bound in Theorem~\ref{thmx:LV} in particular gives a bound on the number of semi-algebraically connected components of $V$ which just depends on the degree of $V$ and independent of the dimension of ambient space and is closest in spirit to the results in this paper. 
    This bound was further improved to $D^{p+1}$ in \cite{Kharlamov}. However, these bounds only apply to non-singular real projective varieties -- while in the results about sign conditions proved in this paper (see Theorem~\ref{thm:main:ordered:a} and Remark~\ref{rem:thm:main:ordered:a} below) 
    the hypothesis on $V$ is much weaker -- namely that $V$ only needs to be a union of  
    affine varieties $V_i$, such that each $V_i$ has the same real dimension at all points of $V_i(\R)$. Many naturally occurring real varieties, such as Schubert subvarieties of flag varieties,
    rank subvarieties in spaces of matrices, as well as other determinantal varieties (see \cite{Manivel} for definitions), are often singular but satisfies the real equi-dimensionality condition mentioned above. 
\end{remark}

\subsection{Summary of main contributions}
\label{subsec:contrib}
In this paper we consider varieties and polynomials over real and complex numbers (in fact more generally over arbitrary real closed and algebraically closed fields). We prove bounds on the number of realizable zero-nonzero patterns and sign conditions  
(and also on the number of connected components of their realizations which is a priori bigger number)
which are \emph{independent of $N$} and depend only on the number and degrees of the polynomials, 
and the degree (of the embedding of $V$ in $\mathbb{A}_k^N$ or in $\PP_k^N$) and the dimension of $V$. 
These appear as Theorems~\ref{thm:main:algebraic}, 
\ref{thm:main:ordered:a}, 
\ref{thm:main:ordered:a:projective}, \ref{thm:main:ordered:c} and
\ref{thm:main:ordered:c'} in Section~\ref{sec:main} below.

The independence of these bounds from the dimension $N$ of the ambient space is a new feature of these bounds (see Remark~\ref{rem:minimal-degree} about the existence of such bounds).
This feature is potentially useful in any application where one is interested in the number of zero-nonzero patterns (resp. sign conditions) restricted to algebraic subsets of dimension much smaller than that of the ambient dimension. 
We discuss examples of such applications below in Section~\ref{subsubsec:applications}.

\subsubsection{Methods}
Our techniques for dealing with zero-nonzero patterns and sign conditions are different. To prove the upper bound on the number of zero-nonzero patterns of a family of polynomials restricted to an algebraic subset $V$ of 
$\mathbb{A}_k^N$, we closely follow the proof given in \cite{Jeronimo-Sabia,BPR-tight}, 
with the modification that all algebraic subsets that are considered 
in the proof are contained in $V$ and we can bound their degrees in terms of $\deg V$ and the degrees of the given polynomials. 
We prove an upper bound on the sum of the degrees of the Zariski closure
of the realizations of all zero-nonzero patterns with non-empty realizations on $V$, which in turn gives an upper bound on the number
of realized zero-nonzero patterns.
An important role in the argument is played by the generalized Bezout inequality (Theorem~\ref{thmx:bezout}). 

For the problem of bounding realizable sign conditions over ordered fields we
assume without loss of generality that the field  $k = \R$ is a real closed field. The main technique that is used for proving bounds on the number of realizable sign conditions (or more generally bounding the Betti numbers of these realizations) is via a technique of infinitesimal perturbations
(see for example, \cite{BPR02,BPRbook2}). The idea is to perturb the given set of polynomials using infinitesimal elements in some real closed extension of $\R$, so that:
\begin{enumerate}[(a)]
\item
The polynomials in the new set have the same degrees as ones in the given set, but in addition have better algebraic properties -- like being in general position. These algebraic properties help control the combinatorial part of the bound.
\item 
The realizations of the various realizable sign conditions are related to those of the given set in some way, so that bounding the number (or more generally the Betti numbers)  of 
realizations of the sign conditions of the new set of polynomials, suffices to bound the corresponding quantity for the original set. 
\item Finally, the problem of bounding the number of realizable sign conditions (or more generally the Betti numbers of the realizations) can be reduced to bounding the corresponding quantity for all the non-empty algebraic subsets that one can define using the perturbed set of polynomials.  
\end{enumerate}
Our technique for proving upper bounds on the number of realizable 
sign conditions follow a similar paradigm but meets two important obstacles.

\begin{enumerate}[(a)]
    \item Since we are not given the polynomials that define $V$, we are not able to deform $V$ infinitesimally maintaining its degree;
    \item as a result while we are able to reduce using usual techniques to the case bounding the number of semi-algebraically connected components of certain algebraic subsets of $V(\R)$, all existing bounds for bounding the number of semi-algebraically connected components of possibly singular real algebraic sets depend exponentially on the dimension of the ambient space.
\end{enumerate}

We overcome these problems by first reducing the problem to the problem of bounding the number of semi-algebraically connected components of an appropriate number of real algebraic subsets $X(\R) \subset V(\R)$,
of co-dimension one in $V$.
then after taking the image under a generic projection $\pi$ to $\R^p$ (where 
$p= \dim V$), to the problem of bounding the number of semi-algebraically connected components of the semi-algebraic set consisting of the $\R$-points of $\pi(X) - \Sing(\pi(X))$. We can use the classical bound due to 
Ole{\u\i}nik and Petrovski{\u\i} \cite{OP} to bound the latter quantity. This bound depends exponentially on the dimension of the ambient space, but since  this is $p$ in this situation we can tolerate the exponential dependence
on $p$. The technique described above for obtaining bounds on the number of semi-algebraically connected components of a real variety only in terms of its degree is new and in our opinion the most important technical contribution of the paper.

\subsubsection{Tightness}
\label{subsubsec:tightness}
Our bound (on the sum of the degrees of the realizations) in the algebraically closed case is tight  (see Remark~\ref{rem:tight:algebraic}).
In the real closed case, our main result applies to a certain slightly restricted class of algebraic subsets of $\R^N$  
that includes most real algebraic sets that occur in practice. 
The bound is tight in this case as well up to a factor of $O(1)^{\dim V}$
(see Remark~\ref{rem:tight:ordered}).

\subsubsection{Applications}
\label{subsubsec:applications}
We give several applications of our results. Some of these results 
improve or generalize prior results leveraging our new bounds.
These include upper bounds on 
the $\eps$-entropy  of real varieties (Theorem~\ref{thm:entropy}),
and lower bounds on the depths of algebraic computations trees 
testing membership in a semi-algebraic set (Theorem~\ref{thm:lower-bound}). In each of these cases the new bounds depend only on the degree and dimension of the relevant algebraic subset and independent of
(or in the $\eps$-entropy case very weakly dependent on) the ambient dimension.

We also introduce a notion of additive (resp. multiplicative) rank of elements of a vector space (resp. algebra) relative to some algebraic 
subset $\Delta$ (Definitions~\ref{def:rank:vector} and \ref{def:rank:algebra}). 
This is motivated from quantum complexity theory, where 
$\Delta$ can be thought of as an 
algebraic family of allowed gates
(i.e. a set of unitary operators operating on the Hilbert space associated to a 
$n$-qubit quantum register each having unit cost). We extend the definition of ranks to subsets of the vector space (resp. algebra), and prove lower bounds on the rank of subsets which have some structure
(Theorems~\ref{thm:rank:vector:algebraic}, \ref{thm:rank:vector:ordered},  \ref{thm:rank:algebra:algebraic} and \ref{thm:rank:algebra:ordered}).
As a consequence, finite subsets which are ``grid''-like (i.e. has some algebraic structure) must contain elements of relatively large rank. 
As an application of the rank lower bound theorems  we show how to recover Shannon's lower bound theorem for classical circuits in the case of quantum circuits 
with algebraic families of allowed gates
(Theorems~\ref{thm:quantum1} and \ref{thm:quantum2}).

The rest of the paper is organized as follows. We state our mathematical results in the following Section~\ref{sec:main}. In Section~\ref{sec:applications}
we describe some of the geometric applications of theorems stated in Section~\ref{sec:main}.
These include:
\begin{enumerate}[(a)]
\item bounding the $\eps$-entropy of real varieties (Section~\ref{subsec:entropy}); 
\item an improvement of known lower bound results for algebraic computation trees (Section~\ref{subsec:tree}).
\end{enumerate}
Section~\ref{sec:proof} is devoted to the proofs of all the results mentioned above.

In Section~\ref{subsec:rank}, we prove 
lower bounds on the relative ranks, additive as well as multiplicative 
in vector spaces and algebras respectively, 
and over algebraically closed as well as real closed fields;
and show the lower bound on ranks of subsets in an algebra imply
lower bounds on quantum circuit complexity with 
algebraic families of allowed gates
(Section~\ref{subsec:quantum}).

Finally, in Section~\ref{sec:conclusion} we state some open problems for future research.

\section{Upper bounds on zero-nonzero patterns and sign conditions}
\label{sec:main}

\subsection{Bound on the number of realizable  zero-nonzero patterns}
We will denote
\[
\Pi(\mathcal{P},V) = \{\pi \in \{0,1\}^{\mathcal{P}} \;\mid\; \RR(\pi,V) \neq \emptyset \}
\]
and call $\Pi(\mathcal{P},V)$ the set of \emph{realizable zero-nonzero patterns of $\mathcal{P}$ on $V$}.

We are now in a position to state our first result.

\begin{theorem}
\label{thm:main:algebraic}
    Let $V \subset k^N$ (resp. $V \subset \PP_k^n$) be an algebraic subset 
    with $\deg V = D$, $\dim V = p$, and 
    $\mathcal{P} \subset k[X_1,\ldots,X_N]_{\leq d}$ (resp. $\mathcal{P} \subset k[X_0,X_1,\ldots,X_N]_{\leq d}$)   be a finite set of polynomials (resp. homogeneous polynomials), with $\card(\mathcal{P}) = s$. Then,
    
    \[
    \card(\Pi(\mathcal{P},V)) \leq
        \sum_{\pi \in \{0,1\}^{\mathcal{P}}} \deg \overline{(\RR(\pi,V))}         \leq 
        D \cdot \sum_{i=0}^p \binom{s}{i} \cdot  d^{i} 
    \]
    (where $\overline{(\RR(\pi,V))}$ denotes the Zariski-closure of
    $\RR(\pi,V)$).
\end{theorem}

\begin{remark} [Tightness]
\label{rem:tight:algebraic}
The bound in (the second inequality of)
Theorem~\ref{thm:main:algebraic} is tight. To see this let $N > 2p$, and $V$ be a union of 
$D$ disjoint affine subspaces of dimension $p$ in $k^N$. Let 
$\mathcal{P} = \{P_1,\ldots,P_s\} \subset k[X_1,\ldots,X_N]$, where each $P_i = \prod_{j=1}^{d} L_{i,j}$, with  $L_{i,j} \in k[X_1,\ldots,X_N], \deg(L_{i,j}) =1$. 
Also suppose that the polynomials 
$L_{i,j}$'s are generic (of degree $1$). It is easy to see that with these choices,

\[
\sum_{\pi \in \{0,1\}^{\mathcal{P}}} \deg \overline{(\RR(\pi,V))}         = 
D \cdot \sum_{i=0}^p \binom{s}{i} \cdot  d^{i}.
\]
\end{remark}

\subsection{Bound on the number of realizable sign conditions}
Now suppose $\R$ is a real closed field and let $\C = \R[i]$.
For an algebraic subset $V \subset \C^N$, we denote by $V(\R) = V \cap \R^N$
the set of its real points. We say that $V$ is a \emph{real algebraic set} (or a \emph{real variety} if $V$ is irreducible),
if $V$ is closed under conjugation taking $z$ to $\bar{z}$. It is easy to see that a real algebraic set is the vanishing locus of a finite number of polynomials in $\R[X_1,\ldots,X_N]$.

For $x \in V(\R)$, we will denote by $\dim_x V(\R)$ to be the local 
real dimension (see for example \cite[page 191]{BPRbook2})
of $V(\R)$ at $x$  (as a semi-algebraic set). If $V$ is a non-singular variety of dimension $p$, then $\dim_x V(\R) = p$ for every $x \in V(\R)$. 

We will denote
\[
\Sigma(\mathcal{P},S) = \{\sigma \in \{0,1,-1\}^{\mathcal{P}} \;\mid\; \RR(\sigma,S) \neq \emptyset \},
\]
and call $\Sigma(\mathcal{P},S)$ the set of \emph{realizable sign conditions of $\mathcal{P}$ on $S$}.

\begin{notation}
    For $S$ a semi-algebraic set, we will denote by $\Cc(S)$ the set
    of semi-algebraically connected components of $S$. 
\end{notation}

\begin{remark}
    Note that 
    $\card(\Cc(S)) = b_0(S)$.
\end{remark}

We can now state our main theorems on bounding the number of semi-algebraically connected components of the realizations of realizable sign conditions of a finite set of polynomials restricted to the real points of a given algebraic set $V$.

\begin{theorem}
\label{thm:main:ordered:a}
Let $V=V_1\cup\cdots\cup V_m\subset \C^N$ be a union of real varieties.  Set
\[
\deg(V)=D,\qquad \dim(V)=p,
\]
and assume that
\[
\dim_x V_i(\R)=\dim V_i\qquad \text{for all }i\in[1,m]\text{ and all }x\in V_i(\R).
\]
Let $\mathcal P\subset \R[X_1,\ldots,X_N]_{\le d}$ be a finite set of polynomials, with $\card(\mathcal P)=s$.
Then
\begin{equation}\label{eq:thm-main-ordered-a}
\card(\Sigma(\mathcal P,V(\R)))
\;\le\;
\sum_{\sigma\in\{0,1,-1\}^{\mathcal P}} b_0\bigl(\RR(\sigma,V(\R))\bigr)
\;\le\;
\mathsf{B}(D,d,s,p),
\end{equation}
where
\[
\mathsf{B}(D,d,s,p)
:=
28\,Dd\,(8Dd-1)^{p-1}\!\left(2\sum_{j=1}^{p-1}4^j\binom{s}{j}+ 4^p\binom{s}{p}\right)
\;+\;
14\,D\,(4D-1)^p .
\]
In particular,
\[
\sum_{\sigma\in\{0,1,-1\}^{\mathcal P}} b_0\bigl(\RR(\sigma,V(\R))\bigr)
\;\le\;
\bigl(O(1)\bigr)^p \cdot D^p \cdot  \bigl((sd)^p + D\bigr)
\]
(where the implied constant in the $O(\cdot)$ notation is an absolute constant).
\end{theorem}

\begin{remark}
\label{rem:thm:main:ordered:a}
    \label{rem:thm:main:ordered:a:2}
    The hypothesis on $V$ is clearly satisfied if each $V_i$ is non-singular of dimension $p$ (see also Remark~\ref{rem:LV}).
\end{remark}

\begin{remark}[Tightness]
\label{rem:tight:ordered}
We give examples showing that the dependence on $D$ and $sd$ in
Theorem~\ref{thm:main:ordered:a} is essentially sharp when $p=1$, and discuss the
situation for $p>1$.

\begin{figure}[t]
\centering
\begin{tikzpicture}[scale=0.9]
% --- parameters (edit if desired) ---
\def\D{5}                  % D x D ovals
\def\ovalw{0.45}           % x-radius of each oval (keep < 0.5)
\def\ovalh{0.30}           % y-radius of each oval (keep < 0.5)

% --- draw D^2 disjoint ovals centered at (i,j), i,j=0,...,D-1 ---
\foreach \i in {0,...,4} {
  \foreach \j in {0,...,4} {
    \draw (\i,\j) ellipse [x radius=\ovalw, y radius=\ovalh];
  }
}

% --- draw spaced vertical lines x = delta_{i,j} within the i=0 column ---
% (here: 8 lines, e.g. s d = 8; adjust the list as needed)
\foreach \x in {0.05,0.10,0.15,0.20,0.25,0.30,0.35,0.40} {
  \draw (\x,-0.7) -- (\x,4.7);
}

% --- labels ---
%%\node[anchor=west] at (0,4.85) {column $i=0$: ovals $V_{0,j}$};
%%\node[anchor=west] at (0.05,-0.95) {vertical lines $X_1=\delta_{i,j}$ (spaced out)};

\end{tikzpicture}
\caption{Schematic for the example with $p=1$ and $N=2$: $V(\R)$ consists of $D^2$ disjoint ovals $V_{i,j}$, and
$\ZZ(\mathcal P,\R^2)$ is a collection of vertical lines $X_1=\delta_{i,j}$ intersecting each oval $V_{0,j}$ in $2sd$ points.}
\label{fig:example-p1-N2}
\end{figure}

\smallskip
\noindent\textbf{The case $p=1$, $N=2$.}
Let $V\subset \C^2$ be the hypersurface defined by $Q=0$, where
\[
Q \;=\; \sum_{i=1}^{2}\ \prod_{j=0}^{D-1}(X_i-j)^2\;-\;\varepsilon,
\qquad \varepsilon>0 .
\]
Then $\dim V=1$ and $\deg(V)=2D$. For all sufficiently small real $\varepsilon>0$,
the variety $V$ is non-singular and $V(\R)$ consists of $D^2$ pairwise disjoint ovals
\[
V_{i,j},\qquad 0\le i,j\le D-1,
\]
with $(i,j)$ contained in the bounded component of $\R^2\setminus V_{i,j}$.

Let $\mathcal P=\{P_1,\dots,P_s\}$, where
\[
P_i \;=\; \prod_{j=1}^{d}(X_1-\delta_{i,j}).
\]
Choose parameters satisfying
\[
0<\delta_{1,1}<\delta_{1,2}<\cdots<\delta_{1,d}<\delta_{2,1}<\cdots<\delta_{s,d}
\ \ll\ \varepsilon\ \ll\ 1,
\]
(where $\ll$ means ``sufficiently small'').  Then for each $i$,
$\ZZ(P_i,\R^2)$ is the union of $d$ vertical lines, and each such line meets every oval
$V_{0,0},\dots,V_{0,D-1}$ in exactly two distinct points
(see Figure~\ref{fig:example-p1-N2}).  Consequently, the arrangement
$\ZZ(\mathcal P,\R^2)$ meets each oval $V_{0,j}$ in $2sd$ distinct points, and hence
cuts $V_{0,j}$ into $2sd$ disjoint arcs.  Summing over $j=0,\dots,D-1$ and accounting
for the remaining ovals, one obtains
\begin{align*}
\sum_{\sigma\in\{0,1,-1\}^{\mathcal P}} b_0\bigl(\RR(\sigma,V(\R))\bigr)
&= 4sd\,D + D(D-1) \\
&= D(4sd + D - 1) \\
&= \Omega(1)\,D\,(sd + D),
\end{align*}
which matches the upper bound in Theorem~\ref{thm:main:ordered:a} in the case $p=1$.

\smallskip
\noindent\textbf{The case $p>1$.}
For $p>1$ the situation is more subtle.  Let $N=p+1$ and take $V\subset \C^N$ to be the
union of
\begin{itemize}
\item $D$ real affine hyperplanes in $\C^N$, together with
\item the hypersurface in $\C^N$ defined by
\[
Q \;=\; \sum_{i=1}^{N}\ \prod_{j=0}^{D-1}(X_i-j)^2\;-\;\varepsilon,
\qquad 0<\varepsilon\ll 1.
\]
\end{itemize}
With this choice, $V(\R)$ is the union of  $D^{p+1}$ bounded semi-algebraically connected
components coming from zeros of $Q$, as well as $D$ hyperplanes. Moreover,
$\dim V = p$, and $\deg(V) = D + 2 D = 3 D$. 

Let $\mathcal P=\mathcal P_1\cup\mathcal P_2\subset \R[X_1,\dots,X_N]$, where
$\mathcal P_1$ consists of $s$ polynomials, each of which is a product of $d$
distinct linear forms, and the resulting collection of $sd$ linear factors is chosen
generically.  Let $\mathcal P_2=\{P_1,\dots,P_s\}$ with
\[
P_i \;=\; \prod_{j=1}^{d}(X_1-\delta_{i,j}),
\]
for parameters satisfying
\[
0<\delta_{1,1}<\delta_{1,2}<\cdots<\delta_{1,d}<\delta_{2,1}<\cdots<\delta_{s,d}
\ \ll\ \varepsilon\ \ll\ 1.
\]
In this regime, one checks that the contribution from the hyperplane part of $V(\R)$
and from the hypersurface part combine to give the lower bound
\[
\sum_{\sigma\in\{0,1,-1\}^{\mathcal P}} b_0\bigl(\RR(\sigma,V(\R))\bigr)
\;\ge\;
\Omega(1)^p\Bigl((sd)^p D \;+\; (sd)\,D^p \;+\; D^{p+1}\Bigr).
\]

At present we do not know an example that matches the upper bound in
Theorem~\ref{thm:main:ordered:a} for $p>1$ even up to a factor $\Omega(1)^p$.
Closing this gap remains an interesting open problem.

\end{remark}

Theorem~\ref{thm:main:ordered:a} may be interesting even in the case
$\mathcal{P} = \emptyset$. In this case we just obtain a bound on the number of semi-algebraically connected components of $V$ having the top real dimension.
In order to state this formally we introduce a new notation.
\begin{notation}
\label{not:S_p}
For any semi-algebraic set $S$ and $p \geq 0$, we denote 
\[
    S_p =  \{x \in S \mid \dim_x S \geq p\}.
\]
\end{notation}

We have the following slight variant of Theorem~\ref{thm:main:ordered:a}
in the case of algebraic sets that will be useful later.

\begin{theorem}
\label{thm:main:ordered:a'}
    Let $V \subset \C^N$ be a real algebraic set, with $\deg(V) =D$, and
    $\dim V = p$.
    Then,
    \begin{equation}
    \label{eqn:thm:main:ordered:a'}   
     b_0(V(\R)_p) \leq  14 \cdot D \cdot (4D-1)^{p} \leq 2^{2p+4} \cdot D^{p+1}.
    \end{equation}
    In particular,
    if every semi-algebraically connected component of $V(\R)$ has (real) dimension equal to $p$, 
    then
    \[
    b_0(V(\R)) \leq
    14 \cdot D \cdot (4D-1)^{p} \leq 2^{2p+4} \cdot D^{p+1}.
    \]
\end{theorem}

\begin{remark}
\label{rem:minimal-degree}
Note that if $V \subset \PP_\C^N$ is a variety with $\dim V = p$, such that 
$V$ is not contained in any hyperplane of $\PP_\C^N$. Then
$N \leq \deg(V) + p - 1$ (see for example \cite[Proposition 4.5.6]{Beltrametti-et-al}) 
. 
\footnote{We thank Greg Blekherman and Rainer Sinn for pointing out this connection with the theory of minimal varieties.} 
Now suppose $V$ is a real variety. Then $V$ can be cut out by real  polynomials
of degree at most $D$. 
Now 
using the basic inequality on the Betti numbers of real algebraic sets due to Ole{\u\i}nik and Petrovski{\u\i} \cite{OP}
    (see for example \cite[Proposition 7.28]{BPRbook2}),
we would obtain  the bound:
\begin{eqnarray}
\nonumber
b_0(V(\R))
&\leq&
O(D)^N\\
\label{eqn:rem:minimal-degree}
&\leq&
O(D)^{D+p-1}.
\end{eqnarray}

Then for $D$ large enough,
the bound in inequality \eqref{eqn:rem:minimal-degree} is worse than the bound in Theorem~\ref{thm:main:ordered:a'}.
Similar remarks apply to Theorems~\ref{thm:main:ordered:a}, \ref{thm:main:ordered:a:projective} and \ref{thm:BPR+} as well.  
\end{remark}

\begin{remark}
As remarked earlier, many naturally occurring real varieties, such as Schubert subvarieties of flag varieties,
    rank subvarieties in spaces of matrices, as well as other determinantal varieties, are not non-singular but satisfies the hypothesis of Theorem~\ref{thm:main:ordered:a}. These varieties in 
    general will not satisfy the hypothesis of Theorem~\ref{thmx:LV} because of the presence of singularities.

However, there exist real varieties for which the hypothesis of
Theorem~\ref{thm:main:ordered:a} is not satisfied.

    \begin{enumerate}[(a)]
        \item The singular curve $V \subset \C^2$ in the plane defined by 
        $Y^2 = X^2(X-1)$. In this case the $\dim V = 1$, but $V(\R)$ contains the isolated point $x = (0,0)$, and $\dim_x V(\R) = 0 < 1$. Thus, $V$ does not satisfy the hypothesis of  Theorem~\ref{thm:main:ordered:a}.
        
        \item 
        Another interesting example
        is the well-known Whitney umbrella, $V \subset \C^3$,  defined by the equation $X^2 - Z Y^2 = 0$. In this case $V(\R)$ has only one connected component but every point $x \in V(\R)$ on the ``stick'' of the umbrella (namely the points satisfying  $Z < 0$) has local dimension $\dim_x V(\R) = 1 < 2 = \dim V$. 
        The real variety $V$ does not satisfy the 
        hypothesis of Theorem~\ref{thm:main:ordered:a} (because of the lower dimensional stick of the umbrella).
    \end{enumerate}
\end{remark}

We also prove a projective version of Theorem~\ref{thm:main:ordered:a}.
It is no longer meaningful to talk about the sign of a (homogeneous)
polynomial in a real projective space. However, one still has the notion of zero-nonzero patterns. We bound the number of semi-algebraically connected components of the realizations on 
$V(\R)$ of all zero-nonzero patterns  of a finite set of homogeneous polynomials, where $V$ is a non-singular real projective variety.

\begin{theorem}[Projective case] 
 \label{thm:main:ordered:a:projective}
 Let $V = V_1 \cup \cdots \cup V_m \subset \PP_\C^N$ be a union of real varieties, with $\deg(V) = D$, $\dim V = p$, and such that for each $i, 1 \leq i \leq m$, and $x \in V_i(\R)$, 
    $\dim_x V_i(\R) = \dim V_i $,  and let 
    $\mathcal{P} \subset \R[X_0,\ldots,X_N]_{\leq d}$ be a finite set of homogeneous polynomials, with $\card(\mathcal{P}) = s$. Then,
    \[ \card(\Pi(\mathcal{P},V(\R))) \leq
        \sum_{\pi \in \{0,1\}^{\mathcal{P}}} b_0(\RR(\pi,V(\R)))
        \leq 
        O(1)^p \cdot D^p \cdot ((s d)^p + D).
    \]   
\end{theorem}

\begin{remark}
    \label{rem:Grassmannian}
    Aside from the other applications that we discuss in the next section
    Theorem~\ref{thm:main:ordered:a:projective} can be used to simplify 
    proofs of certain quantitative results in discrete geometry. For example in \cite{GPW1996}, the authors prove a bound on the number of geometric permutations induced by $k$-dimensional transversals to a finite family of convex sets in $\R^d$. A key role is played by a bound on the number of realizable sign conditions of a certain finite set of polynomials restricted to the Grassmannian $\Gr(k,d)$ (more accurately, the affine Grassmannian but this is unimportant). In order to get a tight bound on this quantity that does not depend on the dimension of the Pl\"{u}cker emdedding the  
    authors utilize a fact from real algebraic geometry -- that the Grassmannian varieties, $\Gr(m,N)$ can be embedded semi-algebraically (not scheme theoretically) as an algebraic subset of $\R^{N^2}$ \cite[Theorem 3.4.4]{BCR}. This last fact is somewhat remarkable as there is no such affine embedding of the complex Grassmannian (indeed the complex Grassmannian is compact as a topological space, while every complex affine 
    algebraic set of positive dimension is necessarily non-compact). However, using Theorem~\ref{thm:main:ordered:a:projective} (noting that the Grassmannians are all non-singular projective varieties) 
    and the known expressions for the degree and the dimension of $\Gr(m,N)$
    (see \eqref{eqn:deg:grassmannian}) one can obtain a similar bound
    as in \cite[Theorem 2]{GPW1996} without resorting to the clever non-standard embedding of the real Grassmannian, and the new bound is a slight improvement over the original result when $k$ is close to $d$.
\end{remark}

The bound in Theorem~\ref{thm:main:ordered:a} is worse than in Theorem~\ref{thm:main:algebraic} (in the dependence on $D$ and $d$). We can prove a bound with a better dependence on $D$ 
but with a more stringent assumption on $V$.

\begin{theorem}
\label{thm:main:ordered:c}
Let $\overline{V} \subset \PP_\C^N$ be a non-singular complete intersection real variety 
with $\deg \overline{V} = D$ and $\dim \overline{V} = p$. Let $V = \overline{V} \cap \C^N$ denote the corresponding affine variety.

  Let
    $\mathcal{P} \subset \R[X_1,\ldots,X_N]_{\leq d}$ be a finite set of polynomials with $d \geq D$ and  $\card(\mathcal{P}) = s$. 
    
    Then, for each $i, 0 \leq i \leq p$,
    \[
        \sum_{\sigma \in \{0,1,-1\}^{\mathcal{P}}} b_i(\RR(\sigma,V(\R)))
        \leq 
        O(1)^p \cdot s^{p-i} \cdot D^2 \cdot d^{p}.
    \]  
    In particular,
    \[
        \sum_{\sigma \in \{0,1,-1\}^{\mathcal{P}}} b_0(\RR(\sigma,V(\R)))
        \leq 
        O(1)^p \cdot D^2 \cdot (s d)^{p}.
    \]  
    (The constant in the $O(1)$ is absolute.)
\end{theorem}

\begin{remark}
Note that in Theorem~\ref{thm:main:ordered:c} we can assume that
$V$ is contained in an affine subspace in $\C^N$ of dimension $\leq p + \log_2 D$
(see Lemma~\ref{lem:CI}),
and this removes the dependence on $N$ from the final bound.
\end{remark}

We are able to remove the non-singularity restriction 
in Theorem~\ref{thm:main:ordered:c}.
However, the bound that we prove is on a priori smaller quantity which nevertheless arises in applications
(see for example Theorem~\ref{thmx:Walsh} which plays an important role in applications of polynomial partitioning technique in problems of incidence geometry).

\begin{notation}
    For $S, V \subset \R^N$ semi-algebraic subsets, we will denote
    \[
    \Cc(S,V) = \{C \in \Cc(S) \mid C \cap V \neq \emptyset\}.
    \]
    (Note that $\Cc(S) = \Cc(S,\R^N)$.)
\end{notation}

\begin{theorem}
\label{thm:main:ordered:c'}
Let $V \subset \C^N$ be a real complete intersection variety (not necessarily non-singular)
with $\deg(V) = D$ and $\dim(V) = p$. 
  Let
    $\mathcal{P} \subset \R[X_1\ldots,X_N]_{\leq d}$ be a finite set of polynomials, with $d \geq D$ and $\card(\mathcal{P}) = s$. 
    
    Then,
    \[
        \sum_{\sigma \in \{0,1,-1\}^{\mathcal{P}}} 
        \card(\Cc(\RR(\sigma,\R^{N}),V(\R)))
        \leq 
        O(1)^p \cdot D^4 \cdot (s d)^{2p}.
    \]  
\end{theorem}

\begin{remark}
The exponent of $D$ in the above bound comes from our proof technique which involves treating
$V$ as real algebraic subset of $\R^{2N}$ by taking real and imaginary parts i.e. considering $V$ as the real part of a complex complete intersection of degree $D^2$. 

We then use Theorem~\ref{thm:main:ordered:c} which leads to an additional squaring. Improving this exponent will require some new idea. 
\end{remark}

\subsection{Sum of Betti numbers of realizations of sign conditions}
Using the result of Laszlo and Viterbo (Theorem~\ref{thmx:LV}) we prove the following theorem which extends Theorem~\ref{thmx:BPR} to the case when the ambient variety is more general than $\R^N$.

\begin{theorem}
\label{thm:BPR+}
Let $\overline{V} \subset \PP_\C^N$ be a non-singular real variety 
with $\deg V = D$ and $\dim V = p$. Let $V = \overline{V} \cap \C^N$ denote the corresponding affine variety. Let $\mathcal{P} \subset \R[X_1,\ldots,X_N]_{\leq d}$ be a finite set of polynomials. Then,
for each $i, 0 \leq i \leq p$,
\[
\sum_{\sigma \in \{0,1,-1\}^{\mathcal{P}}} b_i(\RR(\sigma,V(\R))) 
\leq
O(1)^p \cdot s^{p-i} \cdot D^p \cdot (D + d^p). 
\]
\end{theorem}

\begin{remark}
    \label{rem:BPR+} Notice that the hypothesis on the variety $V$ is much stronger than in Theorem~\ref{thmx:BPR} or Theorem~\ref{thm:main:ordered:a}. 
    However, the bound  unlike Theorem~\ref{thmx:BPR} is independent of the ambient dimension.
\end{remark}

\section{Applications}
\label{sec:applications}

\subsection{Bound on the $\eps$-entropy of real varieties}
\label{subsec:entropy}
In this section we assume that $\R = \mathbb{R}$.

A quantitative result from real (or more accurately \emph{metric}) algebraic geometry that is now of great interest in applications (in problems originating in machine learning) is a tight upper bound on the volume of an $\eps$-tube around a real algebraic variety $V$ intersected with a ball of a given radius. A recent result in this direction is \cite[Theorem 1.1]{Basu-Lerario2023}  which gives 
a tight upper bound on this volume depending on the degree $d$ of the polynomials defining $V$ and the codimension of $V$.

More recently, using very different techniques (originally due to Vitushkin \cite{Vitushkin1957}), Zhang and Kileel \cite{Zhang-Kileel2023} proved another bound on the volume of such tubes, which matches the bound in \cite{Basu-Lerario2023} in the regime large $N$ and
fixed $d$. 
They applied their results to prove new bounds in approximation properties of low rank tensors, on the required dimension for randomized sketching algorithms, and 
generalization error bounds for deep neural networks
with rational or ReLU activation functions.
Their technique depends crucially on a bound on the number of connected components of real algebraic varieties and any improvement
on this bound will automatically improve the bound in their paper.
In particular, inequality \eqref{eqn:thm:main:ordered:a'} is applicable in their context.

We follow the same notation as in \cite{Zhang-Kileel2023}. Let $V(\R)  \subset \mathbb{R}^N$ be a real variety, and for $\eps >0$ denote by $\mathcal{N}(V(\R),\eps)$ the minimum number of $\eps$-balls required to cover $V(\R) \cap B_N(0,1)$. The quantity $\log \mathcal{N}(V(\R),\eps)$ is referred to as the $\eps$-entropy of  $V(\R) \cap B_N(0,1)$ and an upper bound on $\mathcal{N}(V(\R),\eps)$ implies an upper bound on the volume of the 
$\eps$-tube around $V(\R) \cap B_N(0,1)$.

The following result appears in \cite[page 4, Table 1, entry on real varieties]{Zhang-Kileel2023}.

\begin{thmx}
\label{thmx:ZK}
Let $V(\R) \subset B_N(0,1)$ be a real variety defined by polynomials of degrees at most $d$ and $\dim V = p$. Then,

\begin{equation}
 \log \mathcal{N}(V(\R),\eps) \leq p \cdot \log(1/\eps) + O(N \cdot\log d + p\cdot\log N).
\end{equation}
\end{thmx}

We prove the following theorem.
\begin{theorem}
\label{thm:entropy}
Let $V \subset \C^N$ be a real variety with $\deg(V) = D, \dim V = p$, such that $V(\R) \subset B_N(0,1)$. 
Then,
\[
   \log \mathcal{N}(V(\mathbb{R})_p,\eps) \leq p \cdot \log(1/\eps) + O(p \cdot  (\log D + \log N))
  \]
(see Notation~\ref{not:S_p}).

In particular, if every point of $V(\R)$ has local real dimension equal to $p$, then
\[
\log \mathcal{N}(V(\R),\eps) \leq p \cdot \log(1/\eps) + O(p \cdot  (\log D + \log N)).
\]
\end{theorem}

\begin{remark}
\label{rem:entropy1}
   Note that right-hand side of the inequality in Theorem~\ref{thm:entropy} depends very weakly on the embedding dimension $N$ ($\log N$ instead of $N$) compared to 
the bound in Theorem~\ref{thmx:ZK}. Notice also that while the bound in
Theorem~\ref{thmx:ZK} is in terms of the degrees of the real 
polynomials cutting out 
$V$, the bound in Theorem~\ref{thm:entropy} is in terms of the degree of $V$ (as a complex variety) which is a more intrinsic invariant. It is well known that any affine variety of degree $D$ in $\C^N$ can be cut out by at most 
$N+1$ polynomials of degree at most $D$. However, if we use this fact in conjunction with Theorem~\ref{thmx:ZK}, we will end up with a much weaker bound than the one in Theorem~\ref{thm:entropy} ($N$ instead of $\log N$ in the expression on the right-hand side). 
\end{remark}

\begin{remark}
    \label{rem:entropy2}
    Using Remark~\ref{rem:minimal-degree} the right hand side of both inequalities in Theorem~\ref{thm:entropy} can be replaced by 
    \[
    p \cdot \log(1/\eps) + O(p \cdot  (\log D + \log (D + p)).
    \]
\end{remark}

\subsection{Lower bounds for algebraic computation trees}
\label{subsec:tree}
We briefly recall the definition of an algebraic computation tree (see for example \cite{GV2017} from which the following description is taken).
\begin{definition}
\label{def:tree}
An algebraic computation tree $T$ with input variables 
$X_1,\ldots,X_N$ is a tree having three kinds of vertices: 
\begin{enumerate}[(a)]
    \item computation vertices having outdegree $1$;
     \item branching vertices with  outdegree $3$, and the three outgoing edges labelled by $0$, $1$ and $-1$;
     \item leaves with outdegree $0$ and labelled by ``YES'' or ``NO''.
\end{enumerate}
There is a variable $Y_v$ attached to each vertex
$v$ of $T$.Moreover, an expression 
$Y_v = a * b$, where
$* \in \{+,-,\times\}$ is associated to each computation vertex $v$,
and $a,b$ are either real constants, or input variables, or variables
associated with predecessor vertices of $v$, or a combination of these.

At each branch vertex $v$, the variable $Y_v$ is assigned the value which is either a real
constant, or an input variable, or a variable associated with a predecessor vertex. The
three outgoing edges of a branch vertex $v$ labelled $0$, $1$ or $-1$, then correspond to signs 
$Y_v = 0$, $Y_v > 0$, $Y_v <0$ (respectively).

With each leaf $w$, 
we associate naturally a
a basic semi-algebraic set $S_w \subset \mathbb{R}^N$, defined
by a conjunction of inequalities of the type
$Y_v > 0$, $Y_v = 0$, $Y_v <0$,
where $v$ varies over all branch nodes on the path from from the root to the leaf $w$ and the sign of $Y_v$ corresponds to the sign of the branch
taken by the path. 
The semi-algebraic set $S_T$ associated to the tree $S$ is defined as the union
\[
S_T = \bigcup_{w \in \mathrm{Leaves}_1(T)} S_w,
\]
where $\mathrm{Leaves}_1(T)$ denotes the set of leaf nodes of $T$ 
labelled ``YES''.

We say that the algebraic computation tree $T$ tests membership in the semi-algebraic set $S$. The height of a tree is the length of the longest path from its root to a leaf node.
\end{definition}

The interpretation of this model of computation is as follows. On an input
\[
x = (x_1,\ldots,x_N) \in \mathbb{R}^N,
\]
the input variables $X_i$ get the corresponding values $x_i$,
the arithmetic operations are executed in computation vertices $v$, and 
$Y_v$ is assigned the corresponding real value.

At branch vertices, the outward branch is chosen corresponding to the assigned value of $Y_v$. 
The input $x$ belongs to the semi-algebraic set $S_w$ if this process ends in the leaf $w$. If $w$ is a leaf marked ``YES'', then $x$ is said to be \emph{accepted} by the tree $T$.

It was proved by Ben-Or \cite{Ben-Or} that for any semi-algebraic set
the height of an algebraic computation tree testing membership in a semi-algebraic set $S \subset \R^N$, is at least 
$c_1 \log b_0(S) - c_2 N$, where
$c_1,c_2 > 0$ are absolute constants. 
Yao \cite{Yao}
proves that for any locally closed and bounded semi-algebraic set
$S \subset \R^N$, the height of any algebraic computation tree testing membership in $S$ is bounded by
\[
c_1 \log b^{BM}(S) - c_2 N,
\]
where $b^{BM}(S)$ denotes the Borel-Moore Betti numbers of $S$ (they can also be defined as the dimensions of the compactly supported cohomology groups $\HH_c^*(S)$).
This result was extended recently  by Gabrielov and Vorobjov \cite{GV2017} who proved that the height of an algebraic computation tree testing membership in a semi-algebraic set $S \subset \R^N$ (not necessarily locally closed), is at least 
\[
c_1 \frac{\log b_m(S)}{m+1} - c_2 \frac{N}{m+1},
\]
where $c_1,c_2 > 0$ are absolute constants. 

We prove that the following theorem which in particular improves the original result of
Ben-Or in a direction that is independent of \cite{Yao} and \cite{GV2017}.

\begin{theorem}
\label{thm:lower-bound}
Let $S \subset \R^N$ be a semi-algebraic set contained in an real algebraic set $V$.
Suppose that $V$ is a union of 
non-singular real algebraic varieties,
then the height of an algebraic computation tree testing membership in $S$, is at least 
\[
\frac{1}{(1+ (\log \deg V))} (c_1 \log b_0(S) - c_2 \dim V),
\]
where
and $c_1, c_2 > 0$ absolute constants.
\end{theorem}

\begin{remark}
    Note that the above lower bounds are independent of the ambient dimension $N$.
    This is useful in situations where $\dim V$
    and $\log \deg V$ are very small compared to the dimension of the ambient space.

    \begin{example}
    For example, consider the Grassmannian variety $\Gr(m,N)$ of $m$-dimensional subspaces in $\C^N$, as a subvariety of $\PP_\C^{\binom{N}{m}-1}$.
    Let $V = V_{m,N}  \subset \C^{\binom{N}{m}}$ be the affine cone 
    over $\Gr(m,N)$. Then, it is well known 
    (see for example \cite[Corollary 3.2.14]{Manivel}) that
    \begin{equation}
        \label{eqn:deg:grassmannian}
        \deg (V_{m,N}) = \deg(\Gr_{m,N}) = \frac{0!1! \cdots (n-1)!}{(m+1)!\cdots (m+n-1)!}(mn)!,
    \end{equation}
    where $n = N - m$,
    while
    \[
    \dim V_{m,N} = m(N-m)+1.
    \]
    This implies that 
    $\deg (V_{m,N}) = O(1)^{m(N-m)}$, and hence
    $ \log(\deg (V_{m,N})) = O(m(N-m))$.

    Now suppose that $S \subset V_{m,N}(\R)$ is a semi-algebraic subset.
    Then, the lower bound in Theorem~\ref{thm:lower-bound} yields that the 
    depth of any algebraic computation tree testing membership in $S$ is 
    at least
\[
c_1 \frac{\log b_0(S)}{m(N-m)}  - c_2,
\]
while the result of Ben-Or implies a lower bound of 
\[
c_1\log b_0(S) -  c_2 \binom{N}{m}
\]
(with different constants $c_1,c_2$).

If $m >2$ and $b_0(S) < \frac{c_2}{c_1}\binom{N}{m}$, then the last inequality will not give a non-trivial lower bound,  while Theorem~\ref{thm:lower-bound} will yield a non-trivial bound.
\end{example}
\end{remark}

\section{Proofs}
\label{sec:proof}

\subsection{Proof of Theorem~\ref{thm:main:algebraic} 
}

In our proof of Theorem~\ref{thm:main:algebraic} we will need to use the following generalized form of Bezout's theorem
(see for instance \cite[Example 8.4.6, page 148]{Fulton}).

\begin{thmx}[Generalized Bezout]
\label{thmx:bezout}
Let $V_1,\ldots,V_s$ be algebraic subsets of some projective space over $k$,
such that each $V_i$ is a union of varieties of the same dimension,
and 
$Z_1,\ldots,Z_t$ be the irreducible components of the intersection
$V_1\cap \cdots \cap V_s$. Then
$$
\displaylines{
\sum_{1 \leq i \leq t} \deg(Z_i) \leq \prod_{1 \leq i \leq s} \deg(V_i).
}
$$
\end{thmx}

\subsubsection{Outline of the proof Theorem~\ref{thm:main:algebraic}}
The main idea (which is very similar to that used in \cite{Jeronimo-Sabia}, and also in \cite{BPR-tight}) is to first bound the number of zero-non-zero patterns whose realizations are zero dimensional using Theorem~\ref{thmx:bezout}. The main subtlety here is that the dimension of an algebraic set need not decrease upon intersection with another hypersurface. The problem of bounding the number of zero-nonzero patterns whose realizations have dimension $>0$ then can be reduced to the zero-dimensional case by taking generic hyperplane sections.  

The proof of Theorem~\ref{thm:main:algebraic} will now follow from the following two lemmas. 
We follow closely the proof of Lemma 7.3 in \cite{BPR-tight}.
Observe that it is sufficient to prove the theorem in case $V$ is irreducible of dimension $p$. 
Otherwise, we can take the irreducible decomposition of $V$ noting that their degrees add up to $D$.

\begin{notation}
For a finite set of polynomials $\mathcal{Q} \subset k[X_1,\ldots,X_N]$, 
and an algebraic set $W \subset k^N$,
we denote by $\ZZ(\mathcal{Q},W)$ the
set of common zeros of $\mathcal{Q}$ in $V$. If $\mathcal{Q} = \{Q\}$, we write $\ZZ(Q,W)$ for $\ZZ(\mathcal{Q},W)$.
\end{notation}

Let $V,\mathcal{P}$ be as in Theorem~\ref{thm:main:algebraic}.
For $\rho \in  \{0,1\}^{\mathcal{P}}$,
let ${\mathcal P}_\rho = \{P \in {\mathcal P} \mid \rho(P) = 0\}$ and
let $V_\rho = \ZZ({\mathcal P}_\rho,V)$.  
Let $V_\rho = V_{\rho,1} \cup \cdots \cup V_{\rho,n(\rho)}$ denote the
decomposition of $V_\rho$ into irreducible components.

We first prove a bound on the number of isolated points occurring 
as $\RR(\rho,V)$ for some $\rho \in \{0,1\}^{\mathcal{P}}$.

\begin{lemma}
\label{lem:isolatedpoints}
\[
\card(\{(\rho,i) \mid \dim V_{\rho,i} = 0\}) \leq \binom{s}{p} \cdot  D \cdot d^p.
\]
\end{lemma}

\begin{proof}
Let ${\mathcal P} = \{P_1,\ldots,P_s\}$ (the implicit ordering of
the polynomials in ${\mathcal P}$ will play a role in the proof).

First consider a sequence of $p$ polynomials of ${\mathcal P}$, 
$P_{i_1},\ldots,P_{i_p},$ with
$1 \leq i_1 <  \cdots <  i_p \leq s$.
Let $\{W_{\alpha_1}\}_{\alpha_1 = 1,2..}$ denote the irreducible
components of $\ZZ(P_{i_1},V)$. Similarly, let
$\{W_{\alpha_1,\alpha_2}\}_{\alpha_2 = 1,2,\ldots}$ denote the 
irreducible components  of $W_{\alpha_1} \cap \ZZ(P_{i_2},V)$ and so on.
We prove by induction that for $1 \leq \ell \leq p$
\begin{equation}
\label{eqn:bound}
\sum_{\alpha_1,\alpha_2,\ldots,\alpha_\ell} \deg(W_{\alpha_1,\ldots,\alpha_\ell}) \leq D \cdot d^\ell.
\end{equation}

The claim is trivially true for $\ell=1$. Suppose it is true up to $\ell -1 $.
Then using Theorem \ref{thmx:bezout} we get that
\[
\sum_{\alpha_1,\ldots,\alpha_{\ell-1}}
 \deg(\ZZ(P_\ell,W_{\alpha_1,\ldots,\alpha_{\ell-1}}))  
 \leq D \cdot d^{\ell-1}\cdot d = D \cdot d^\ell.
\]

Now suppose that $\dim V_{\rho,i} = 0$ and let $V_{\rho,i} = \{x\}$.

We first claim that
there must exist $P_{i_1},\ldots,P_{i_p} \in {\mathcal P}_{\rho}$, with
$1 \leq i_1 < i_2 < \cdots < i_p \leq s$,  and
irreducible affine algebraic varieties 
$V_0, V_1,\ldots,V_p$ satisfying:
\begin{enumerate}
\item
$V= V_0 \supset V_1 \supset V_2 \supset \cdots \supset V_p = \{x\}$, and
\item
for each $j, \;1 \leq j \leq p$,
$V_{j}$ is an irreducible component of $\ZZ(P_{i_j},V_{j-1})$
which contains $x$.
\end{enumerate} 
Note that in the sequence, $V_0,V_1,\ldots,V_p$,
$\dim (V_i)$ is necessarily  equal to $p-i$.
To prove the claim, let 
\[
{\mathcal P}_\rho = \{P_{j_1},\ldots,P_{j_m}\},
\]
with $1 \leq j_1 < j_2 < \cdots < j_m \leq s$.
Let $W_0 = V$ and
for each $h, 1 \leq  h \leq m$ define inductively
$W_{h}$ to be an irreducible component of $\ZZ(P_{j_h},W_{h-1})$
containing the point $x$. By definition,
$V = W_0 \supset W_1 \supset \cdots \supset W_m = \{x\}$ and each
$W_h$ is irreducible. Hence, there must exist precisely
$p$ indices,
$1 \leq \alpha_1 < \cdots < \alpha_p \leq m$ such that
$\dim W_{\alpha_{i}} = \dim  W_{\alpha_{i-1}} - 1$.
If $\beta \not\in \{\alpha_1,\ldots,\alpha_p\}$, then
$W_{\beta} = W_{\beta - 1}$.
Now let $i_1 = j_{\alpha_1},\ldots, i_p = j_{\alpha_p}$,
and $V_j = W_{\alpha_j}$ for $0 \leq j \leq p$. It is clear from construction
that the sequence $V_0,\ldots,V_p$ satisfies the properties stated above,
thus proving the claim.
 
The number of sequences of
polynomials of length $p$,
$P_{i_1},\ldots,P_{i_p},$ with
$1 \leq i_1 < \cdots < i_p \leq s$,  is $\binom{s}{p}$. The lemma
now follows from the bound in (\ref{eqn:bound}).
\end{proof} 

The following lemma is a generalization of the above lemma to positive
dimensions.

\begin{lemma}
\label{lem:positive}
For $0 \leq \ell \leq p$
$$
\displaylines{
\sum_{\rho \in \{0,1\}^{\mathcal{P}}} 
\sum_{1 \leq i \leq n(\rho), \dim V_{\rho,i} =\ell} 
\deg V_{\rho,i} \leq \binom{s}{p-\ell} \cdot D\cdot d^{p-\ell}.
}
$$
\end{lemma}

\begin{proof}
In order to bound the degree of $V_{\rho,i}$ with $\dim V_{\rho,i} =\ell$
it suffices to count the isolated zeros of the
intersection of $V_{\rho,i}$ with a generic affine subspace $L$ of dimension
$p-\ell$. Since there are only a finite number of $V_{\rho,i}$'s to consider
we can assume that
\begin{enumerate}[(a)]
\item  $L$ is generic for all of them,
\item
the sets $V_{\rho,i} \cap L$ are pairwise disjoint,
\item
for each $V_{\rho,i}$ with $\dim_{\C}V_{\rho,i} =\ell$,
$ \ZZ(P,L \cap V_{\rho,i}) = \emptyset$ for each 
$P \in {\mathcal P} \setminus {\mathcal P}_\rho$.
\end{enumerate}
Now restrict to the subspace $L$ and apply Lemma 
\ref{lem:isolatedpoints},
noting that an isolated point of
$V_{\rho,i} \cap L$ is an isolated point of the family
${\mathcal P}$ restricted to $L$.
\end{proof}

\begin{proof}[Proof of Theorem~\ref{thm:main:algebraic}]
Follows directly from Lemma~\ref{lem:positive}.
\end{proof}

\subsection{Proofs of Theorems~\ref{thm:main:ordered:a}, 
\ref{thm:main:ordered:a'}
and 
\ref{thm:main:ordered:a:projective}}

We first recall some preliminary definitions and results.
\subsubsection{Multiplicity and singular points of a variety}
\label{subsubsec:mult}
Given a variety $V$, and a point $x \in V$,
the multiplicity of $V$ at $x$, denoted $\mult_x(V)$,  
is defined as the degree
of the exceptional divisor of the blowup $\widetilde{V}$ of $V$ at $x$ 
\cite[Section 4.3, page 79]{Fulton}. This definition is equivalent to the definition of the Samuel multiplicity of the local ring $\mathcal{O}_{V,x}$ \cite[Example 4.3.1]{Fulton}.

For an algebraic set $V$ with irreducible components $V_1,\ldots,V_m$,
and for $x \in V$, we define
$\mult_x(V) = \sum_{i=1}^{m} \mult_x(V_i)$ (where we define
$\mult_x(V_i) = 0$ if $x \not\in V_i$). 

Given a variety $V$,
and a point $x \in V$,
$\mult_x(V) = 1$ if and only if the local ring
$\mathcal{O}_{V,x}$ is regular \cite[Example 4.3.5 (d)]{Fulton}, and
this is equivalent to $V$ being \emph{non-singular} at $x$
\cite[Section B.2.7, Page 429]{Fulton}.

The following lemma follows immediately from the above which we record for future use.

\begin{lemma}
\label{lem:singular}
If $V$ is an algebraic set, and $x \in V$. Then, $x$ is a non-singular point of $V$ if and only if $\mult_x(V) = 1$.
\end{lemma}

The set of points of $V$ which are not non-singular will be denoted by
$\Sing(V)$. It is a proper closed subset of $V$ (see \cite[Theorem 5.3]{Hartshorne}).
If $V$ is a variety, then $\dim \Sing(V) < \dim V$.

We say that $V$ is \emph{non-singular}  if $\Sing(V) = \emptyset$.

We also recall below  the notion of an extension of a real closed field $\R$ by the real closed field $\R\la\zeta\ra$ of algebraic Puiseux series in $\zeta$ (where $\zeta$ is an \emph{infinitesimal} i.e. $\zeta $ is positive but smaller than all positive elements of $\R$). These extensions are used for making infinitesimal perturbations of the given polynomials.

\subsubsection{Real closed extensions and Puiseux series}
\label{subsubsec:Puiseux}
We will need some
properties of Puiseux series with coefficients in a real closed field. We
refer the reader to \cite{BPRbook2} for further details.

\begin{notation}
  For $\R$ a real closed field we denote by $\R \left\langle \eps
  \right\rangle$ the real closed field of algebraic Puiseux series in $\eps$
  with coefficients in $\R$. We use the notation $\R \left\langle \eps_{1},
  \ldots, \eps_{m} \right\rangle$ to denote the real closed field $\R
  \left\langle \eps_{1} \right\rangle \left\langle \eps_{2} \right\rangle
  \cdots \left\langle \eps_{m} \right\rangle$. Note that in the unique
  ordering of the field $\R \left\langle \eps_{1}, \ldots, \eps_{m}
  \right\rangle$, $0< \eps_{m} \ll \eps_{m-1} \ll \cdots \ll \eps_{1} \ll 1$.
\end{notation}

\begin{notation}
\label{not:lim}
  For elements $x \in \R \left\langle \eps \right\rangle$ which are bounded
  over $\R$ we denote by $\lim_{\eps}  x$ to be the image in $\R$ under the
  usual map that sets $\eps$ to $0$ in the Puiseux series $x$.
\end{notation}

\begin{notation}
\label{not:extension}
  If $\R'$ is a real closed extension of a real closed field $\R$, and $S
  \subset \R^{k}$ is a semi-algebraic set defined by a first-order formula
  with coefficients in $\R$, then we will denote by $\E(S, \R') \subset \R'^{k}$ the semi-algebraic subset of $\R'^{k}$ defined by
  the same formula.
 It is well known that $\E(S, \R')$ does
  not depend on the choice of the formula defining $S$ 
  \cite[Proposition 2.87]{BPRbook2}.
\end{notation}

\subsubsection{Outline of the proof of Theorem~\ref{thm:main:ordered:a}}
\label{subsec:outline:proof:thm:main:ordered:a}
We first observe that it suffices to consider the case when $V$ is irreducible of 
degree $D$ and dimension $p$, satisfying that $\dim_x V(\R) = p$ for all
$x \in V(\R)$. 
We introduce an infinitesimal $\gamma_0$.
Using the conic structure theorem for semi-algebraic sets \cite[Proposition 5.49]{BPRbook2} we 
are able to restrict attention to realizations of sign conditions
which are contained in a ball of radius $(1/\gamma_0)^{1/2}$. This necessitates introduction of a new polynomial denoted by $P_0$ in the proof.
We then define infinitesimal perturbations (using $s+2$ additional infinitesimals
$\gamma_1,\ldots,\gamma_s,\eps,\delta$) of the 
polynomials in $\mathcal{P}$ and prove that it suffices to bound the number of
semi-algebraically connected components of the 
realizations 
on $\widetilde{S} := \E(V(\R),\widetilde{\R})$ (the extension of 
$V(\R)$ to the field $\widetilde{\R}$ obtained from $\R$ by adjoining
infinitesimal elements)
of a certain set of \emph{weak} sign conditions
on this new set of $4s+1$
polynomials (Lemma~\ref{lem:main:ordered:4}).
We then show that in order to bound these,  it suffices to bound the number of semi-algebraically connected components of
all the non-empty real algebraic sets 
defined by the perturbed set of $4s +1$ polynomials
which are contained in $\widetilde{S}$ (Lemma~\ref{lem:main:ordered:7}).

If we consider the real algebraic set defined as the intersection of $\widetilde{S}$ with $i, 0 \leq i \leq p$ of the perturbed polynomials, then we get a real algebraic set of dimension $p-i$ of degree bounded by $D d^i$. 
So it now suffices to bound the number of semi-algebraically connected components of each of these 
real algebraic sets. Taking a sum of squares we reduce to the case where the real algebraic set (say)  $\widetilde{Z}$, is defined as the intersection with
$\widetilde{S}$ with the $\ZZ(\widetilde{Q}, \widetilde{\R}^N)$
for a single polynomial $\widetilde{Q}$.

We bound $b_0(\widetilde{Z})$ in two steps.

\begin{enumerate}[(a)]
\item
We first prove (using the semi-algebraic curve selection lemma and 
the intermediate value theorem for real closed fields) that 
\[
b_0(\widetilde{Z}) \leq  b_0(\ZZ(\widetilde{Q} - \zeta, \E(\widetilde{S},\widetilde{R}\la\zeta\ra))
\]
where $\zeta$ is a new infinitesimal
(Lemma~\ref{lem:main:ordered:8}).
\item 
The real algebraic set
\[
\ZZ(\widetilde{Q} - \zeta, \E(\widetilde{S},\widetilde{R}\la\zeta\ra)
\]
if non-empty, is of dimension $p-1$, and satisfy the local dimension condition and are of degree $2 D d^i$.
If
$\pi$ denotes a generic projection  to a real subspace of dimension $p$,
then 
$\pi(\ZZ(\widetilde{Q} - \zeta, \E(\widetilde{S},\widetilde{R}\la\zeta\ra))$ 
is a real hypersurface of degree $2Dd^i$ and is 
defined by a real polynomial of that degree.
Denoting 
$\widetilde{Z} = 
\ZZ(\widetilde{Q} - \zeta, \E(\widetilde{S},\widetilde{R}\la\zeta\ra))$
we prove that 
$b_0(\widetilde{Z})$
is bounded by $b_0(\pi(\widetilde{Z}) -  \Sing(\pi(\widetilde{Z})))$
(Lemma~\ref{lem:main:ordered:1}).
Since, $\Sing(\pi(\widetilde{Z})))$ is a real algebraic subset of 
$\pi(\widetilde{Z})$, we can define its real points by taking the sum of squares of all partial derivatives of the polynomial defining $\pi(\widetilde{Z})$.
We then use the Ole{\u\i}nik-Petrovski{\u\i} inequality \cite{OP} bounding the Betti numbers of real algebraic sets to bound the last quantity
(Lemma~\ref{lem:main:ordered:2}). 
\item We also need a bound on $b_0(V(\R))$, since this term appears in the case $i=0$ in the arguments described above.
\end{enumerate}

Theorem~\ref{thm:main:ordered:a} now reduces to bounding the contributions from the different algebraic sets mentioned above. 
 We summarize this reduction step in Proposition~\ref{prop:master} below.

\subsubsection{Proof of Theorem~\ref{thm:main:ordered:a}}
Let 
\[
\mathcal{P} = \{P_1,\ldots,P_s\} \subset \R[X_1,\ldots,X_N].
\]

We denote 
\[
\widetilde{\mathcal{P}} = \bigcup_{i=1}^{s} \{P_i \pm \eps \gamma_i, P_i \pm \delta \gamma_i\} \cup
\{P_0\} \subset \widetilde{\R}[X_1,\ldots,X_N],
\]
where
\[
P_0 = \gamma_0\left( \sum_{i=1}^N X_i^2 \right) - 1,
\]
and
\[
\widetilde{\R} = \R\la\gamma_0,\gamma_1,\ldots,\gamma_s,\eps,\delta\ra.
\]

For $I  \subset [1,s]$ and $\alpha \in \{\pm \eps, \pm \delta\}^I$
we denote
\[
Q_{I,\alpha} = \sum_{i \in I} \left(P_i + \alpha(i)\gamma_i\right)^2, 
\]
or
\[
Q_{I,\alpha,0} = \left(P_0 - \eps \gamma_0 \right)^2 + \sum_{i \in I} \left(P_i + \alpha(i)\gamma_i\right)^2.
\]

The following proposition is a key step in the proof of Theorem~\ref{thm:main:ordered:a}:
it reduces the problem of bounding
$\sum_{\sigma\in\{0,1,-1\}^{\mathcal P}} b_0\bigl(\RR(\sigma,V(\R))\bigr)$
to that of bounding the number of semi-algebraically connected components of certain
auxiliary real algebraic sets constructed from $\mathcal{P}$.

\begin{proposition}[Master bound]
\label{prop:master}
Let $V\subset \C^N$ be a real variety, with $\dim V = p$.
Let $\mathcal P=\{P_1,\dots,P_s\}\subset \R[X_1,\dots,X_N]$, and let the notation
$\widetilde{\R}$, $P_0$, $Q_{I,\alpha}$, and $Q_{I,\alpha,0}$ be as above. 
Let $\zeta$ be an infinitesimal over $\widetilde{\R}$.

Define
\begin{align}
\label{not:prop:master:A}
A \;=\;
\sum_{\substack{I\subset[1,s]\\ 1\le |I|\leq p}}
\ \ \sum_{\alpha\in\{\pm\varepsilon,\pm\delta\}^I}
b_0\!\left(\ZZ\!\left(Q_{I,\alpha} - \zeta,\ \E\!\left(V(\R),\widetilde{\R}\langle\zeta\rangle\right)\right)\right),\\[1mm]
\label{not:prop:master:B}
B
\;=\;
\sum_{\substack{I\subset[1,s]\\ 1\le |I|<p}}
\ \ \sum_{\alpha\in\{\pm\varepsilon,\pm\delta\}^I}
b_0\!\left(\ZZ\!\left(Q_{I,\alpha,0}-\zeta,\ \E\!\left(V(\R),\widetilde{\R}\langle\zeta\rangle\right)\right)\right),
\\[1mm]
\label{not:prop:master:C}
C \;=\; b_0\bigl(V(\R)\bigr).
\end{align}
Then
\begin{equation}
    \label{eqn:prop:master}
\sum_{\sigma\in\{0,1,-1\}^{\mathcal P}} b_0\bigl(\RR(\sigma,V(\R))\bigr)
\;\le\; 
A+B+C.
\end{equation}
\end{proposition}

Theorem~\ref{thm:main:ordered:a} will follow from Proposition~\ref{prop:master} and 
bounds on 
$A$, $B$
and $C$ appearing on the right hand side of inequality~\eqref{eqn:prop:master} as given in the following two propositions.

\begin{proposition}[Bound on $C$]
\label{prop:master:C}
Let $V$ be as in Proposition~\ref{prop:master} with $\deg(V) = D$,
and such that for each $x \in V(\R)$,  $\dim_x V(\R) = \dim V$. Then
(following notation introduced in Proposition~\ref{prop:master} ~\eqref{not:prop:master:C})
\[
C \leq 14 \cdot D \cdot (4D-1)^{p}.
\]
\end{proposition}

\begin{proposition} [Bounds on $A$ and $B$]
\label{prop:master:AB}
    Let $V$ be as in Proposition~\ref{prop:master:C}, and
    $\mathcal{P} \subset \R[X_1,\ldots,X_N]_{\leq d}$.
    Suppose that $\deg(V) = D$ and and $\card(\mathcal{P}) = s$.
    Then 
    (following notation introduced in Proposition~\ref{prop:master} ~\eqref{not:prop:master:A} and \eqref{not:prop:master:B})
    \begin{eqnarray*}
        A & \leq & \sum_{j=1}^{p} 4^j \binom{s}{j} 14 \cdot 2Dd(8Dd -1)^{p-1}, \\
        B &
        \leq & \sum_{j=1}^{p-1} 4^j \binom{s}{j} 14 \cdot 2Dd(8Dd -1)^{p-1}.
    \end{eqnarray*}
\end{proposition}

We now prove Propositions~\ref{prop:master}, \ref{prop:master:C} and
\ref{prop:master:AB}.

We begin by proving Proposition~\ref{prop:master}.
We will use the following lemmas.

\begin{lemma}
\label{lem:main:ordered:4}
Let $\sigma \in \Sigma(\mathcal{P},V(\R))$ and $C
\in \Cc(\RR(\sigma,V(\R)))$.
Then there exists a unique element
$\widetilde{C}_\sigma \in \Cc(\RR(\widetilde{\phi}_\sigma,\E(V(\R),\widetilde{\R}))),
$
where 
\begin{align*}
\widetilde\phi_\sigma \;:=\;& (P_0\le 0)\ \wedge \\
&\bigwedge_{\substack{1\le i\le s\\ \sigma(P_i)=0}}
\bigl(-\delta\gamma_i \le P_i \le \delta\gamma_i\bigr)\ \wedge \\
&\bigwedge_{\substack{1\le i\le s\\ \sigma(P_i)=1}}
\bigl(P_i-\varepsilon\gamma_i \ge 0\bigr)\ \wedge \\
&\bigwedge_{\substack{1\le i\le s\\ \sigma(P_i)=-1}}
\bigl(P_i+\varepsilon\gamma_i \le 0\bigr),
\end{align*}
such that 
\[
\widetilde{C}_\sigma \cap V(\R) = C \cap V(\R).
\]
\end{lemma}

\begin{proof}
Let $x \in C$. Then, $x \in \RR(\widetilde{\phi}_\sigma,\E(V(\R),\widetilde{\R}))$, and
let $\widetilde{C}_\sigma$, be the semi-algebraically connected
component of $\RR(\widetilde{\phi}_\sigma,\E(V(\R),\widetilde{\R}))$
containing $x$.
Then 
$\widetilde{C}_\sigma \cap V(\R) = C \cap V(\R)$
(see \cite[Proposition 13.7]{BPRbook2}).
\end{proof}

\begin{lemma}
\label{lem:main:ordered:5}
    For $I \subset \{1,\ldots,s\}$ and
    $\alpha \in \{\pm \eps, \pm \delta\}^I$,
    
    \[
    \dim \left(V \cap \bigcap_{i \in I} \ZZ(P_i + \alpha(i)  \gamma_i,\C^N)\right) = p - \card(I).
    \]
    
\end{lemma}

\begin{proof}
    Follows from the fact that $\alpha(i)\gamma_i, i \in I$ are independent transcendentals over $\C$.
\end{proof}

\begin{remark}
    \label{rem:lem:main:ordered:5}
    Note that as an immediate consequence of Lemma~\ref{lem:main:ordered:5} we get that
    \[
    V \cap \bigcap_{i \in I} \ZZ(P_i + \alpha(i)  \gamma_i,\C^N) = \emptyset,
    \]
    whenever $\card(I) > p$.
\end{remark}
\begin{lemma}
\label{lem:main:ordered:6}
    Let $V \subset \R^N$ be a closed semi-algebraic set and $\mathcal{Q} \subset \R[X_1,\ldots,X_N]$ a finite subset. Let $C \in \Cc(S)$, where $S \subset V$  is the semi-algebraic set defined by
    \[
    S = 
    \{x \in V \mid \bigwedge_{Q \in \mathcal{Q}} (Q(x) \geq 0)
    \}. 
    \]
    Then there exists a subset $\mathcal{Q}' \subset \mathcal{Q}$, and
    $D \in \Cc(V \cap \ZZ(\mathcal{Q}',\R^N))$ such that $D \subset C$.
\end{lemma}

\begin{proof}
Let $\mathcal{Q}'$ be a maximal subset of $\mathcal{Q}$ with the property that  $\ZZ(\mathcal{Q}',\R^N) \cap C \neq \emptyset$. Then
$\mathcal{Q}'$ satisfies the required property in the lemma (see \cite[Proposition 13.1]{BPRbook2}). 
\end{proof}

\begin{lemma}
\label{lem:main:ordered:7}
    Let $\sigma \in \{0,1,-1\}^{\mathcal{P}}$ and  
    $C \in \Cc(\RR(\sigma,V(\R)))$.
   
    Then there exist
    $I \subset \{1,\ldots,s\}$, 
    $\alpha \in \{\pm \eps,\pm \delta\}^I$,
    and $D$ in the union of 
    \[
    \Cc\left(\bigcap_{i\in I}
    \ZZ(P_i + \alpha(i)\gamma_i,\E(V(\R),\widetilde{\R}))\right)
    \]
    and 
    \[ 
    \Cc\left(\ZZ(P_0 - \eps \gamma_0,\E(V(\R),\widetilde{\R}))
    \cap 
    \bigcap_{i\in I} \ZZ(P_i + \alpha(i)\gamma_i,\E(V(\R),\widetilde{\R}))\right),
    \]
    such that 
    \[
    \alpha(i) = \begin{cases}
                    \eps \text{ if } \sigma(P_i) = 1, \\
                    -\eps \text{  if  } \sigma(P_i) = -1, \\
                    \delta \text{ or } -\delta \text{ if } \sigma(P_i) = 0,
                \end{cases}
    \]
    and  
    \[
    D \subset \widetilde{C}_\sigma.
    \]
\end{lemma}

\begin{proof}
    Follows from Lemmas~\ref{lem:main:ordered:4} and
    \ref{lem:main:ordered:6}.
\end{proof}

\begin{lemma}
\label{lem:main:ordered:8}
    Let $S \subset \R^N$ be a closed and bounded semi-algebraic set and
    $Q \in \R[X_1,\ldots,X_N]$, such that at every point $x \in \ZZ(Q,S)$,
    $\dim_x \ZZ(Q,S) < \dim_x S$.
    Then,
    \[
    b_0(\ZZ(Q,S)) \leq b_0(\ZZ(Q-\zeta, \E(S,\R\la\zeta\ra))) + b_0(\ZZ(Q+\zeta, \E(S,\R\la\zeta\ra))).
    \]
    In particular, if $Q$ is non-negative over $S$, then 
    \[
    b_0(\ZZ(Q,S)) \leq b_0(\ZZ(Q-\zeta, \E(S,\R\la\zeta\ra))).
    \]
\end{lemma}

\begin{proof}
Let $\widetilde{S}$ be the intersection of $\E(S,\R\la\zeta\ra)$ with the set defined by $-\zeta \leq Q \leq \zeta$.
Let $C \in \Cc(\ZZ(Q,S))$ and $x \in C$. Then, it follows from the hypothesis that $\dim_x \ZZ(Q,C) < \dim_x S$. Hence, for all $r >0$,
denoting (here and elsewhere) by $B_N(x,r)$ the open Euclidean ball centered at $x$ of radius $r$,
$\ZZ(Q,C) \cap B_N(x,r)$ is non-empty and properly contained in
$S \cap B_N(x,r)$. So $S - \ZZ(Q,C) \neq \emptyset$ and $x \in \overline{S - \ZZ(Q,C)}$. Applying the semi-algebraic curve selection lemma (see for example \cite[Theorem 3.19]{BPRbook2}), there exists a semi-algebraic curve $\gamma:[0,1] \rightarrow S$, such that $\gamma(0) = x$, and 
$\gamma((0,1]) \subset S - \ZZ(Q,C)$. In particular, $Q(\gamma(1)) \neq 0$. 

Suppose without loss of generality that $Q(\gamma(1)) > 0$. Using the intermediate value property of the field $\R\la\zeta\ra$ (\cite[Proposition 4.3]{BPRbook2}), there exists 
$\widetilde{t} \in \R\la\zeta\ra, \widetilde{t} > 0$, such that 
$Q(\E(\gamma,\R\la\zeta\ra)(\widetilde{t})) = \zeta$, and let
$\widetilde{x} = \E(\gamma,\R\la\zeta\ra)(\widetilde{t})$. Then
$\widetilde{x} \in \ZZ(Q-\zeta, \E(S,\R\la\zeta\ra))$. Let $\widetilde{C} \in \Cc(\ZZ(Q-\zeta, \E(S,\R\la\zeta\ra)))$ with  $\widetilde{x} \in \widetilde{C}$.
Then, $\lim_\zeta \widetilde{C} \subset C$. 
Thus we have proved that for each $C \in \Cc(\ZZ(Q,S))$,  there exists 
$\widetilde{C} \in \Cc(\ZZ(Q-\zeta, \E(S,\R\la\zeta\ra))) \cup \Cc(\ZZ(Q+\zeta, \E(S,\R\la\zeta\ra)))$, with $\lim_\zeta \widetilde{C} \subset C$. The lemma follows.
\end{proof}

\begin{remark}
\label{rem:lem:main:ordered:8}
   Note that, the same conclusion in Lemma~\ref{lem:main:ordered:8} holds for the number of bounded components when $S$ is merely closed and not necessarily bounded.
\end{remark}

\begin{proof}[Proof of Proposition~\ref{prop:master}]
 Using Lemma~\ref{lem:main:ordered:7} we are reduced to bounding the number of semi-algebraically connected components of the real algebraic sets,
    \[
    \bigcap_{i\in I}
    \ZZ(P_i + \alpha(i)\gamma_i,\E(V(\R),\widetilde{\R})),
    \]
    and 
    \[
    \ZZ(P_0 - \eps \gamma_0,\E(V(\R),\widetilde{\R}))  \cap  \bigcap_{i\in I} \ZZ(P_i + \alpha(i)\gamma_i,\E(V(\R),\widetilde{\R})),
    \]
    where $I \subset [1,s]$, and $\alpha$ as defined in Lemma~\ref{lem:main:ordered:7}.
    Using Lemma
    \ref{lem:main:ordered:5} and Remark~\ref{rem:lem:main:ordered:5},
    we only need to consider $I$, with $\card(I) \leq p$
    for the algebraic sets
    $\bigcap_{i\in I}
    \ZZ(P_i + \alpha(i)\gamma_i,\E(V(\R),\widetilde{\R}))$
    and of cardinality at most $p-1$ for the algebraic sets
    $\ZZ(P_0 - \eps \gamma_0,\E(V(\R),\widetilde{\R}))  \cap  \bigcap_{i\in I} \ZZ(P_i + \alpha(i)\gamma_i,\E(V(\R),\widetilde{\R}))$.

Thus in order to bound 
$\sum_{\sigma \in \{0,1,-1\}^{\mathcal{P}}} b_0(\RR(\sigma,V(\R)))$ it suffices to bound (using Lemma~\ref{lem:main:ordered:8} and Remark~\ref{rem:lem:main:ordered:8})
\begin{enumerate}[(a)]
\item 
\label{itemlabel:proof:thm:main:ordered:a:1}
$b_0(\ZZ(Q_{I,\alpha} - \zeta, \E(V,\widetilde{\R}\la\zeta\ra)))$,
$I \subset [1,s], 1 \leq \card(I) \leq p$,
\item 
\label{itemlabel:proof:thm:main:ordered:a:2}
 $b_0(\ZZ(Q_{I,\alpha,0} - \zeta, \E(V,\widetilde{\R}\la\zeta\ra)))$,
$I \subset [1,s], 1 \leq \card(I) < p$,
and
\item 
\label{itemlabel:proof:thm:main:ordered:a:3}
$b_0(V(\R))$.
\end{enumerate}

Item \eqref{itemlabel:proof:thm:main:ordered:a:1} correspond to 
$A$,
item \eqref{itemlabel:proof:thm:main:ordered:a:2} to 
$B$, and 
item \eqref{itemlabel:proof:thm:main:ordered:a:3} to
$C$ in the right hand side of 
inequality~\eqref{eqn:prop:master}.
\end{proof}

We now prove Propositions~\ref{prop:master:C} and \ref{prop:master:AB}.

\begin{lemma}
\label{lem:main:ordered:0}
    Let $V \subset \C^N$ be a variety  with $\dim V = p$.
    Then for all generically chosen linear maps $\pi:\C^N \rightarrow \C^{p+1}$,
    and set of points $w \in W = \pi(V)$, such that
    $\card(\pi^{-1}(w) \cap V)=1$ is Zariski dense in $W$.
\end{lemma}

\begin{proof}
It follows from the
Noether normalization theorem
(see for example \cite[Theorem 2.2.9]{deJong-Pfister-book})
that for a generic linear projection map $P:\C^{N} \rightarrow \C^p$
of rank $p$,
$V_\ell := P^{-1}(\ell) \cap V$ is (non-empty and) finite for all $\ell \in L = \mathrm{Im}(P) \subset \C^N$.

Let $P_0: \C^N \rightarrow \C^N$ be a generic projection of rank $p$, and
denote $L_0 = \mathrm{Im}(P_0) \subset \C^N$.

The set $\mathcal{Q}$ of 
linear projections $Q:\C^N \rightarrow \C^N$ of rank $p+1$ such
that $L_0 \subset \mathrm{Im}(Q) $
is a subvariety of the variety of all projections of rank $p+1$.

Let
\[
\Gamma = \{\ell,Q) \in L_0 \times \mathcal{Q} 
\mid Q|_{V_\ell} \text{is not injective}\}.
\]
Notice that $\dim \Gamma < p + \dim \mathcal{Q}$.

Let $\pi_1: \Gamma \rightarrow U$, $\pi_2: \Gamma \rightarrow \mathcal{Q}$ be the two projection maps.

For a generic choice of $Q_0 \in \mathcal{Q}$, 
\begin{equation}
\label{eqn:proof:lem:main:ordered:0}
\dim \pi_2^{-1}(Q_0) < p,
\end{equation}
since otherwise
$\dim \Gamma = p + \dim \mathcal{Q}$.

Let $W' = Q_0(P_0^{-1}(\pi_1(\pi_2^{-1}(Q_0))))$.
Then $\dim W' < p$ (using \eqref{eqn:proof:lem:main:ordered:0} and the fact that $P_0|_V$ is a finite map).

Now define $\pi: \C^{N} \rightarrow \mathrm{Im}(Q_0) \cong \C^{p+1}$, by
$\pi(x) = Q_0(x)$. Then, with $W = \pi(V)$, 
$W' \subset W$, with $\dim(W') < p = \dim(W)$.
Hence, we have for each $w \in W - W'$,
$\card(\pi^{-1}(w) \cap V) = 1$. This proves the lemma.
\end{proof}

\begin{lemma}
\label{lem:main:ordered:1}
    Let $V \subset \C^N$ be a variety with $\dim V = p$.
    Then for all generically chosen linear maps $\pi:\C^N \rightarrow \C^{p+1}$,
    $W = \pi(V) \subset \C^{p+1}$ is a hypersurface, and the subset $W' \subset W$, defined by 
    $W' = \{w \in W  \mid \card(\pi^{-1}(w) \cap V) > 1\}$ is contained
    in $\Sing(W)$.
\end{lemma}

\begin{proof}
    It follows from Noether normalization theorem
    (see for example \cite[Theorem 2.2.9]{deJong-Pfister-book})
    that for $\pi$ generic,
    $\pi(V)$ is a hypersurface in $\C^{p+1}$.
    
Let $w \in W = \pi(V)$, and let $V_w = \pi^{-1}(w) \cap V$. 
It follows from Lemma~\ref{lem:main:ordered:0}, that 
there exists a neighborhood $U_w \subset W$ of $w$ in $W$, such that
for each $z \in V_w$ has a neighborhood
$U_{z}$ in  $\pi^{-1}(U_w) \cap V$, such that
$\pi|_{U_z}$ is injective, and moreover 
for $z \neq z' \in V_w$, the neighborhoods 
$U_{z}$ and $U_{z'}$ are disjoint, and moreover it follows from
Lemma~\ref{lem:main:ordered:0} that $\pi(U_z) \neq  \pi(U_{z'})$.

Let $\widetilde{\pi}: \widetilde{U} \rightarrow U_w$ (shrinking $U_w$ if necessary) be a
blowup of $U_w$ at $w$. By the universal property of the blow-up morphism (see for example \cite[Chapter II, Proposition 7.14]{Hartshorne}) there exists a unique morphism $\pi': \pi^{-1}(U_w) \rightarrow \widetilde{U}$, such that
$\pi|_{\pi^{-1}{U_w}} = \widetilde{\pi} \circ \pi'$. It follows from the fact that 
$\pi(U_z) \neq  \pi(U_{z'})$, that $\pi'(z) \neq \pi'(z')$. But $\pi'(z), \pi'(z')$ belong to the exceptional divisor of the blow-up $\widetilde{\pi}$. It now follows from the definition of multiplicity (see Section~\ref{subsubsec:mult})
that 
$\mult_w(W) >1$, and 
it follows from Lemma~\ref{lem:singular} that in this case
$w \in \Sing(W)$.
\end{proof}

We are now going to use the equi-dimensionality property of $V$ 
in the hypothesis of Proposition~\ref{prop:master:C} (and also in Theorem~\ref{thm:main:ordered:a})  i.e.
the local real dimension at all the real points of $V$ equals $\dim V$.

\begin{lemma}
\label{lem:main:ordered:2}
   Let $V \subset \C^N$ be  
   a real variety such that $\dim_x V(\R) = p$ for all $x \in V(\R)$.
   Then, for all generically chosen linear maps
   $\pi:\R^N \rightarrow \R^{p+1}$, 
   \[
        b_0(V(\R)) 
        \leq b_0(\pi_\C(V)(\R) - \Sing(\pi_\C(V))),
   \]
   where $\pi_\C: \C^N \rightarrow \C^{p+1}$ obtained from $\pi$ by 
   extending the base field to $C$.
\end{lemma}

\begin{proof}
      It follows from Lemma~\ref{lem:main:ordered:1}, that for any
    $z \in \pi_\C(V) - \Sing(\pi_\C(V))$, 
    \[
    \card(\pi_\C^{-1}(z) \cap V) = 1.
    \]
    
    Now, since $V$ is defined over $\R$, for $x \in \pi(V(\R))$,
    $\card(\pi_\C^{-1}(x) \cap V) = 1$ implies that 
    $\pi_\C^{-1}(x) \cap V \subset V(\R)$, since otherwise $\pi_\C^{-1}(x) \cap V$ will contain at least one pair of conjugate points $z \neq \bar{z}, z,\bar{z} \in \pi_\C^{-1}(x) \cap V$.

    Taken together the previous two statements imply: 
    \[
    \pi_\C^{-1}(\pi_\C(V)(\R) - \Sing(\pi_\C(V))) \cap V \subset V(\R). 
    \]

    We now prove that for each $C \in \Cc(V(\R))$, there exists 
    $D \in \Cc(\pi_\C(V)(\R) - \Sing(\pi_\C(V)))$, such that 
    $\pi^{-1}(D) \subset C$. It would follow that

    \[
    b_0(V(\R)) \leq b_0(\pi_\C(V)(\R) - \Sing(\pi_\C(V))).
    \]

    Let $C \in \Cc(V(\R))$. Since $\pi(V(\R)) \subset \pi_\C(V)(\R)$there exists $D' \in \Cc(\pi_\C(V)(\R))$
    such that $\pi(C) \subset D'$. 

Since the local real dimension of $V(\R)$ is equal to $p$ at all points of $V(\R)$,
$\dim_\R C = \dim_\R D' = p$. (This is the place where the equi-dimensionality property of $V$ plays a role.)
Moreover, $\dim_\R D' \cap \Sing(\pi_C(V)) < p$, and hence
    $D' - \Sing(\pi_C(V)) \neq \emptyset$.
    Now let 
    \[
    D \in \Cc(D' - \Sing(\pi_C(V))) \subset \Cc(\pi_\C(V)(\R) - \Sing(\pi_\C(V))).
    \]
Then, $\pi^{-1}(D) \subset C$ and we are done.
\end{proof}

\begin{lemma}
\label{lem:main:ordered:3}
    Let $V = \ZZ(P,\C^{m})$ and $P \in \R[X_1,\ldots,X_m]_{\leq d}$.
    Then
    \[
        b_0(V(\R) - \Sing(V)) \leq 14 \cdot d \cdot (4d-1)^{m-1}.
    \]
\end{lemma}

\begin{proof}
    Let $Q = \sum_{i=1}^m \left(\frac{\partial P}{\partial X_i}\right)^2$. Then, there exists $R > 0$, such that 
    \[
    b_0(V(\R) - \Sing(V)) = b_0((V(\R) - \Sing(V)) \cap B_m(0,R)),
    \]
    and
    \[
     b_0((V(\R) - \Sing(V)) \cap B_m(0,R)) = b_0(S)   
    \]
    where $S \subset \R\la\eps\ra^N$ is defined by 
    \[
    (P = 0) \wedge  (Q \geq \eps) \wedge (\sum_{i=1}^m X_i^2 -R \leq 0).
    \]

    It follows from Lemma~\ref{lem:main:ordered:6} that $b_0(S)$ is at most 
    the number of semi-algebraically connected components of
    the  $\binom{3}{1} + \binom{3}{2} + \binom{3}{3} = 7$
    algebraic sets
    defined by taking one, two or three polynomials 
    $P,Q - \eps, \sum_{i=1}^m X_i^2 -R$ at a time. 
    
    The number of semi-algebraically connected components of each one them is bounded by 
    \[
    2d(4d -1)^{m-1},
    \]
    using the basic inequality on the Betti numbers of real algebraic sets due to Ole{\u\i}nik and Petrovski{\u\i} \cite{OP}
    (see for example \cite[Proposition 7.28]{BPRbook2}).
    The lemma follows.
\end{proof}

\begin{proof}[Proof of Proposition~\ref{prop:master:C}]
    The proposition follows from Lemmas~\ref{lem:main:ordered:0}, 
    \ref{lem:main:ordered:1}, \ref{lem:main:ordered:2} and \ref{lem:main:ordered:3}. 
\end{proof}

\begin{proof}[Proof of Proposition~\ref{prop:master:AB}]
Note that for any $I \subset [1,s]$
$\ZZ(Q_{I,\alpha,0} - \zeta, \E(V,\widetilde{\R}\la\zeta\ra))$ and
$\ZZ(Q_{I,\alpha} -  \zeta, \E(V,\widetilde{\R}\la\zeta\ra))$ 
are real algebraic sets in  $\widetilde{R}\la\zeta\ra[i]^N$ of dimension $p - 1$ if non-empty,
satisfying the property of constancy of local real dimension, and
degree bounded by $2 D d$.

The number of summands appearing in the definition of 
$A$
(Eqn. \eqref{not:prop:master:A}) is equal to 
$\sum_{1 \leq j \leq p} 4^j \binom{s}{j}$, and the 
number of summands appearing in the definition of  
$B$
(Eqn. \eqref{not:prop:master:B})
$\sum_{1 \leq j \leq p-1}4^j \binom{s}{j}$.

  The proposition follows from applying Lemmas~\ref{lem:main:ordered:2} and \ref{lem:main:ordered:3} to the algebraic sets 
$\ZZ(Q_{I,\alpha} -  \zeta, \E(V(\R),\widetilde{\R}\la\zeta\ra))$ and
$\ZZ(Q_{I,\alpha,0} - \zeta, \E(V(\R),\widetilde{\R}\la\zeta\ra))$.
and the combinatorial count on the number of terms appearing in right hand side of the definitions of 
$A$ and $B$
stated in the previous paragraph.
\end{proof}

\begin{proof}[Proof of Theorem~\ref{thm:main:ordered:a}]
    First of all observe that using the fact that 
    \[
    \deg(V) = \sum_{i=1}^m \deg(V_i),
    \]
    and for each $\sigma \in \{0,1,-1\}^{\mathcal{P}}$,
    \[
    b_0(\RR(\sigma,V(\R))) \leq \sum_{i=1}^m b_0(\RR(\sigma,V_i(\R))),
    \]
    we can reduce to the case $m=1$ and assume that $V = V_1$.
    
   The theorem now follows directly from Propositions~\ref{prop:master}, \ref{prop:master:C} and \ref{prop:master:AB}. 
\end{proof}

We now prove Theorem~\ref{thm:main:ordered:a'} by modifying the proof of Theorem~\ref{thm:main:ordered:a} in the special case when
$\mathcal{P} = \emptyset$. The constancy of the local dimension of  $V$ is used in the proof of Theorem~\ref{thm:main:ordered:a} only in Lemma~\ref{lem:main:ordered:2}. The following lemma replaces
Lemma~\ref{lem:main:ordered:2} in the proof of Theorem~\ref{thm:main:ordered:a'}.
Note that as in Lemma~\ref{lem:main:ordered:2}, the projection step in the following lemma removes the dependency on $N$ from the bound, and the subsequent bound (Lemma~\ref{lem:main:ordered:3}) is for a \emph{hypersurface} in $\R^{p+1}$ of controlled degree.

\begin{lemma}
\label{lem:main:ordered:2'}
   Let $V \subset \C^N$ be an algebraic subset of dimension $p$ defined over $\R$. Then, for all generically chosen linear maps
   $\pi:\R^N \rightarrow \R^{p+1}$, 
   \[
        b_0(V(\R)_p) \leq  b_0(\pi_\C(V)(\R) - \Sing(\pi_\C(V))),
   \]
   where $\pi_\C: \C^N \rightarrow \C^{p+1}$ obtained from $\pi$ by 
   extending the base field to $\C$.
\end{lemma}

\begin{proof}
We follow closely the proof of Lemma~\ref{lem:main:ordered:2}.
Using the same argument as in the proof of Lemma~\ref{lem:main:ordered:2},
we have 
    \[
    \pi_\C^{-1}(\pi_\C(V)(\R) - \Sing(\pi_\C(V))) \cap V \subset V(\R). 
    \]

    We now prove that for each $C \in \Cc(V(\R)_p)$, there exists 
    $D \in \Cc(\pi_\C(V)(\R) - \Sing(\pi_\C(V)))$, such that 
    $\pi^{-1}(D) \subset C$. It would follow that

    \[
    b_0(V(\R)_p) \leq b_0(\pi_\C(V)(\R) - \Sing(\pi_\C(V))).
    \]

    Let $C \in \Cc(V(\R)_p)$. Since $\pi(V(\R)) \subset \pi_\C(V)(\R)$ there exists 
     $D' \in \Cc(\pi_\C(V)(\R))$
    such that $\pi(C) \subset D'$. 

    Since the local real dimension of $V(\R)_p$ is equal to $p$ at all points of $V(\R)_p$,
    $\dim_\R C = \dim_\R D' = p$.
    Moreover, since $\dim_\R (D' \cap \Sing(\pi_C(V))) < p$, 
    $D' - \Sing(\pi_C(V)) \neq \emptyset$.
    Now let 
    \[
    D \in \Cc(D' - \Sing(\pi_C(V))) \subset \Cc(\pi_\C(V)(\R) - \Sing(\pi_\C(V))).
    \]
    Then, $\pi^{-1}(D) \subset C$ and we are done.
\end{proof}

\begin{proof}[Proof of Theorem~\ref{thm:main:ordered:a'}]
The theorem follows from Lemma~\ref{lem:main:ordered:2'} and 
Lemma~\ref{lem:main:ordered:3}.
\end{proof}

\begin{proof}[Proof of Theorem~\ref{thm:main:ordered:a:projective}]
Let $H\subset \PP^N_{\C}$ be a \emph{generic} real hyperplane.  Then
\begin{align*}
\sum_{\pi\in\{0,1\}^{\mathcal P}} b_0\bigl(\RR(\pi,V(\R))\bigr)
\;\le\;&
\sum_{\pi \in\{0,1\}^{\mathcal P}} b_0\bigl(\RR(\pi,\,V(\R)\setminus H)\bigr)
\;+\; \\
&
\sum_{\pi\in\{0,1\}^{\mathcal P}} b_0\bigl(\RR(\pi,\,(V\cap H)(\R))\bigr).
\end{align*}
Indeed, for each $\pi \in \{0,1\}^{\mathcal{P}}$ 
the semi-algebraic set $\RR(\pi,V(\R))$ is the disjoint union of
$\RR(\pi,V(\R)\setminus H)$ and $\RR(\pi,(V\cap H)(\R))$, and every semi-algebraically
connected component of $\RR(\pi,V(\R))$ contains a semi-algebraically connected component
of one of these two sets.

The first sum can be bounded by the affine estimate (Theorem~\ref{thm:main:ordered:a}).
For the second sum, note that the hypothesis that the local real dimension is constant on
each $V_i(\R)$ is preserved after intersecting with a generic real hyperplane; hence
$V\cap H$ satisfies the same regularity assumption.  Since $\dim(V\cap H)=p-1$, we may apply
the inductive hypothesis to $V\cap H$ to bound
$\sum_{\pi\in\{0,1\}^{\mathcal P}} b_0\bigl(\RR(\pi,\,(V\cap H)(\R))\bigr)$,
thereby completing the induction step.
\end{proof}

\subsection{Proofs of Theorems~\ref{thm:main:ordered:c} and \ref{thm:main:ordered:c'}}

We can assume without loss of generality that $V$ is a complete intersection variety having degree $D$ and dimension $p$. If the $V$ is a union of complete intersection varieties, $V_1,\ldots,V_m$, then the sum of the degrees of the 
of $V_i$ is bounded by $D$, and 

\[
\sum_{\sigma \in \{0,1,-1\}^{\mathcal{P}}} b_0(\RR(\sigma,V(\R)))
\leq 
\sum_{i=1}^{m}
\sum_{\sigma \in \{0,1,-1\}^{\mathcal{P}}}b_0(\RR(\sigma,V_i(\R))).
\]

Thus, it suffices to prove the inequality in the proposition for each $V_i$ and then take the sum of the upper bounds.

We will use the following lemma.

\begin{lemma}
\label{lem:CI}
Let $V$ is a complete intersection variety in $\C^N$ of degree $D$ and dimension $p$. 
Then $V$ is contained in an affine subspace in $\C^N$ of dimension $\leq p + \log_2 D$.
\end{lemma}

\begin{proof}
There exist polynomials,
$Q_1,\ldots,Q_{N-p} \in \C[Z_1,\ldots,Z_N]$ such that $V$ is the set of common zeros in $\C^N$ of $Q_1,\ldots, Q_{N-p}$,
and $D = d_1 \cdots d_{n-p}$, where $d_i = \deg(Q_i), 1 \leq i \leq N-p$. We can also assume without loss of generality that $d_1 \geq \cdots \geq d_{N-p}$. Let $q, 1 \leq q \leq N-p$ the largest integer such that
$d_q > 1$. Then $d_1,\ldots,d_q \geq 2$ and $d_{q+1} = \cdots = d_{n-p} = 1$. 
So $D = d_1 \cdots d_q \geq 2^q$. Hence, $q \leq \log_2 D$.
Thus, $V$ is contained in an affine subspace in $\C^N$ of dimension $N - (N - p - q) = p+q \leq p + \log_2 D$.  
\end{proof}

\subsubsection{Outline of the proof of Theorem~\ref{thm:main:ordered:c}}
 There are several new ingredients. Firstly, as in the proof of Theorem~\ref{thm:main:ordered:a}, adding a new inequality (that defining a ball of large radius) we can restrict our attention to the realizations of sign conditions which are bounded.
 Then using infinitesimal deformations
of the kind already used in the proof of Theorem~\ref{thm:main:ordered:a},
we go to a situation where the perturbed polynomials are in general
position (meaning that their zeros in $\mathbb{P}^N_\C$ are non-singular complete intersections), and for each semi-algebraically connected component $C$ of the realization on $V(\R)$ of any realizable sign condition on the original family, there exists a semi-algebraically connected component $\widetilde{C}$ of the realization on the extension of $V(\R)$ of 
a \emph{strict}  sign condition on the perturbed family such that
$\widetilde{C}$ is semi-algebraically homotopy equivalent to the extension of $C$ (Lemma~\ref{lem:main:ordered:c:1}). This is stronger than just requiring $\widetilde{C} \cap \R^N = C$ (as in the proof of Theorem~\ref{thm:main:ordered:a}), but we need the stronger property in this proof.

We then observe that each bounded semi-algebraically connected component $C$ of a realizable sign condition of $\mathcal{P}$ on $V(\R)$ is semi-algebraically homotopically equivalent to 
a semi-algebraically connected component of the complement in
$\overline{V}(\R)$ of the union of the set of zeros of the homogenizations of the perturbed family (Lemma~\ref{lem:main:ordered:c:2}).

Now using the fact that $\overline{V}(\R)$ is non-singular (and hence a compact semi-algebraic manifold), and 
Lefschetz duality, 
in order to bound the sum of all the Betti numbers of 
the realizations on $V(\R)$ of all realizable sign conditions on the original family, it suffices to bound the sum of the Betti numbers of 
$V$ and the those of the union of the algebraic sets defined by the perturbed family (Lemma~\ref{lem:main:ordered:c:3}). 

At this point we use a standard reduction via the Mayer-Vietoris spectral sequence to reduce this latter problem to bounding the Betti numbers of 
the algebraic sets on $\overline{V}(\R)$ defined by the perturbed family
(Lemma~\ref{lem:main:ordered:c:4}).

These intersections are non-singular complete intersections. For the corresponding algebraic sets over $\C$ which are projective non-singular complete intersections, one can bound the Betti numbers using a classical formula coming from Chern class computation of the tangent bundle (Lemma~\ref{lem:main:ordered:c:6}). This gives us an upper bound on the Betti numbers of the real parts of these varieties using the Thom-Smith inequalities (Lemma~\ref{lem:main:ordered:c:5}). The theorem follows by aggregrating the contributions to the Betti numbers from the different algebraic sets. 

\subsubsection{Proof of Theorem~\ref{thm:main:ordered:c}}
We will need the following notation and lemmas.

\begin{notation}
    For any closed and bounded semi-algebraic set $S$ we denote by
    $b(S) = \sum_{i \geq 0} b_i(S)$ the sum of the $\Z/2\Z$-Betti numbers of $S$.
\end{notation}

\begin{notation}
\label{not:lem:main:ordered:c:1}
Let $\mathcal{P} = \{P_1,\ldots,P_s\}  \subset \R[X_1,\ldots,X_N]$ be a finite set.
Let $H_1,\ldots,H_s \in \R[X_1,\ldots,X_N]$ with
$H_1(x),\ldots, H_N(x) > 0$, for all $x \in S$,
and 
\[
\widetilde{\mathcal{P}} = \bigcup_{i=1}^{s} \{P_i \pm \eps H_i, P_i \pm \delta H_i\} \subset \R\la\eps,\delta\ra[X_1,\ldots,X_N]. 
\]

For $\sigma \in \sigma \in \{0,1,-1\}^{\mathcal{P}}$ define
    $\widetilde{\sigma} \in \{0,1,-1\}^{\widetilde{\mathcal{P}}}$ as follows. For each $i, 1 \leq i \leq s$,
    If $\sigma(P_i) = 0$:
    \begin{eqnarray*}
    \widetilde{\sigma}(P_i + \eps H_i) &=& +1, \\
    \widetilde{\sigma}(P_i - \eps H_i) &=&  -1,\\
    \widetilde{\sigma}(P_i + \delta H_i) &=& +1, \\
    \widetilde{\sigma}(P_i -  \delta\ H_i) &=& -1. 
    \end{eqnarray*}
    If $\sigma(P_i) = 1$:
    \begin{eqnarray*}
    \widetilde{\sigma}(P_i + \eps H_i) &=& 1, \\
    \widetilde{\sigma}(P_i - \eps H_i) &=&  1,\\
    \widetilde{\sigma}(P_i + \delta H_i) &=& 1, \\
    \widetilde{\sigma}(P_i -  \delta\ H_i) &=& 1. 
    \end{eqnarray*}

    If $\sigma(P_i) = -1$:
    \begin{eqnarray*}
    \widetilde{\sigma}(P_i + \eps H_i) &=& -1, \\
    \widetilde{\sigma}(P_i - \eps H_i) &=&  -1,\\
    \widetilde{\sigma}(P_i + \delta H_i) &=& -1, \\
    \widetilde{\sigma}(P_i -  \delta\ H_i) &=& -1. 
    \end{eqnarray*}
\end{notation}

\begin{lemma}
\label{lem:main:ordered:c:1}
Let $S \subset \R^N$ be a closed and bounded semi-algebraic set.  
Then using the notation in Notation~\ref{not:lem:main:ordered:c:1},
for each $i \geq 0$,
\[
\sum_{\sigma \in \{0,1,-1\}^{\mathcal{P}}} b_i(\RR(\sigma,S)) 
\leq 
\sum_{\widetilde{\sigma} \in \{1,-1\}^{\widetilde{\mathcal{P}}}} b_i(\RR(\widetilde{\sigma},\E(S,\R\la\eps,\delta\ra))). 
\]
\end{lemma}

\begin{proof}
Using the fact that $S$ is closed and bounded it 
is immediate that
that 
  \[
  \E(\RR(\sigma,S),\R\la\eps,\delta\ra) \sim
  \RR(\widetilde{\sigma},\E(S,\R\la\eps,\delta\ra)),
  \]
  where $\sim$ denotes semi-algebraic homotopy equivalence.
\end{proof}

\begin{lemma}
\label{lem:main:ordered:c:2}
    Let $\overline{V} \subset \PP_\C^N$ and 
    $V(\R)  = \overline{V} \cap \R^N$, where
    $\overline{V}$ is a 
    variety. Let
    $\mathcal{P} = \{P_1,\ldots,P_s\} \subset \R[X_1,\ldots,X_N]$ be a finite subset, and let
    $P_0 = \sum_{i=1}^N X_i^2 - R$ with $R \in \R, R>0$.
    Let for $i=0,\ldots,s$, $H_i \in \R[X_1,\ldots,X_N]$ be generic polynomials with $H_i(x) > 0$ for all $x \in \R^N$, and $\deg(H_i) = \deg(P_i)$.
    Let $\widetilde{\mathcal{P}} = \bigcup_{i=0}^{s} \{P_i \pm \eps H_i, P_i \pm \delta H_i\} \subset \R\la\eps,\delta\ra[X_1,\ldots,X_N]$.

    Then, for all large $R > 0$, and $i \geq 0$,
\[
\sum_{\sigma \in \{0,1,-1\}^{\mathcal{P}}} b_i(\RR(\sigma,V(\R))) 
\leq 
\sum_{\widetilde{\sigma} \in \{1,-1\}^{\widetilde{\mathcal{P}}}} b_i(\RR(\widetilde{\sigma},\E(V(\R),\R\la\eps,\delta\ra))). 
\]
\end{lemma}

\begin{proof}
   It follows from the conic structure at infinity of semi-algebraic sets 
   \cite{BCR} that for all large enough $R> 0$, for each $\sigma \in \{0,1,-1\}^{\mathcal{P}}$, $\RR(\sigma, V(\R))$ is semi-algebraically homeomorphic to  $\RR(\sigma, V(\R)\cap B_N(0,R))$. Let $\sigma_0$ denote the extension of the sign condition $\sigma$ to
   $\mathcal{P}_0 = \{P_0\} \cup \mathcal{P}$, by setting $\sigma_0(P_0) = -1$.
   Then, for all large enough $R > 0$, $\RR(\sigma,V(R))$ is semi-algebraically homeomorphic to $\RR(\sigma_0,V(\R))$. Now let $S = V(\R) \cap \overline{B_N(0,2R)}$. Notice that $S$ is a closed and 
   bounded semi-algebraic set, and $\RR(\sigma_0,V(\R)) =
   \RR(\sigma_0,S)$. The lemma now follows from Lemma~\ref{lem:main:ordered:c:1}.
\end{proof}

In the following lemma we will consider a projective variety in $\PP^N_\C$ with homogeneous coordinates
$[X_0:\cdots:X_N]$. We will identify $\C^N$ (resp. $\R^N$) with the affine chart in $\PP^N_\C$ (resp. $\PP^N_\R$) defined by $X_0 \neq 0$. For any polynomial
$Q \in \C[X_1,\ldots,X_N]$, we will denote by $Q^h \in \C[X_0,\ldots,X_N]$ its homogenization.

Finally, for any semi-algebraic set $S \subset \R^N$, we denote by $\Cc_b(S)$, the subset of $\Cc(S)$, consisting of the semi-algebraically connected components of $S$ which are bounded.

\begin{lemma}
\label{lem:main:ordered:c:3}
    Let $\overline{V} \subset \PP_\C^N$ and 
    $V(\R)  = \overline{V} \cap \R^N$, where
    $\overline{V}$ is a non-singular real variety. Let
    $\mathcal{Q} \subset \R[X_1,\ldots,X_N]$ be a finite subset.

    Then, for each $i, 0 \leq i \leq \dim \overline{V}$,
    \[
        \sum_{\sigma \in \{1,-1\}^{\mathcal{Q}}} 
        \sum_{C \in \Cc_b(\RR(\sigma,V(\R))} b_i(C) \leq
        b_i(\overline{V}(\R)) + b^{\dim \overline{V}- i -1}(\bigcup_{Q \in \mathcal{Q}} \ZZ(Q^h,\overline{V}(\R)).
    \]
\end{lemma}

\begin{proof}
    Let $K = \bigcup_{Q \in \mathcal{Q}} \ZZ(Q^h,\overline{V}(\R))$.
   Since $\overline{V}(\R)$ is a compact (real) manifold (orientable over $\mathbb{Z}/2\mathbb{Z}$), 
   and $K$ is a closed and bounded semi-algebraic subset, we obtain from 
   Lefschetz duality theorem (see for example \cite[page 297, Theorem 18]{Spanier}),
   that for each $i$,
   $b_i(\overline{V}(\R) - K) = b_{\dim \overline{V} - i}(\overline{V}(\R),K)$.
   \footnote{By a standard application of the Tarski-Seidenberg transfer principle we can assume that this equality holds over a general real closed field $\R$.}
   Moreover, using the cohomology long exact sequence one has
   \[
    \dim \HH^{\dim \overline{V} - i}(\overline{V}(\R),K) \leq 
    \dim \HH^{\dim \overline{V} - i}(\overline{V}(\R)) + \dim \HH^{\dim \overline{V} - i-1}(K).
   \]
   
   Thus,  we get
   \begin{equation}
   \label{eqn:proof:lem:main:ordered:c:3:a}
     b_i(\overline{V}(\R) - K) \leq  b_{\dim \overline{V} - i}(\overline{V}(\R)) + b_{\dim \overline{V} - i-1}(K).  
   \end{equation}

   Now suppose that $C \in \Cc_b(\RR(\sigma,V(\R))$ for some 
   $\sigma \in \{1,-1\}^{\mathcal{Q}}$. 
   Then, 
   $C \in \Cc(\overline{V}(\R) - \bigcup_{Q \in \mathcal{Q}} \ZZ(Q^h,\overline{V}(\R)))$.

    This implies that for each $i \geq 0$,
    \begin{equation}
    \label{eqn:proof:lem:main:ordered:c:3:b}
\sum_{\sigma \in \{1,-1\}^{\mathcal{Q}}} 
        \sum_{C \in \Cc_b(\RR(\sigma,V(\R))} b_i(C) 
        \leq 
        b_i(\overline{V}(\R) - \bigcup_{Q \in \mathcal{Q}} \ZZ(Q^h,\overline{V}(\R))).
    \end{equation}

   The lemma now follows from inequalities \eqref{eqn:proof:lem:main:ordered:c:3:a} and
   \eqref{eqn:proof:lem:main:ordered:c:3:b}.
\end{proof}

\begin{lemma}
    \label{lem:main:ordered:c:4}
    Let $\overline{V} \subset \PP^N_\C$ be a real variety and 
    $\mathcal{Q} = \{Q_1,\ldots,Q_s\} \subset \R[X_0,\ldots,X_N]$, be a finite set of  homogeneous polynomials. Then, for each $n \geq 0$,
    \begin{eqnarray*}
    b_n(\overline{V}(\R) \cap \bigcup_{i \in [1,s]} \ZZ(Q_i,\PP^N_\R)) &\leq& 
    \sum_{i=0}^{n} \sum_{H \subset [1,s], \card(H) = i+1} b_{n-i}(\overline{V}(\R) \cap \bigcap_{h \in H} \ZZ(Q_{h},\PP^N_\R)) \\
    &\leq&
    \sum_{i=0}^{n} \sum_{H \subset [1,s], \card(H) = i+1} b(\overline{V} \cap \bigcap_{h \in H} \ZZ(Q_{h},\PP^N_\C)).
    \end{eqnarray*}
\end{lemma}
\begin{proof}
    The first inequality follows from the Mayer-Vietoris spectral sequence associated to the cover of the closed semi-algebraic set 
    $S = \overline{V}(\R) \cap \bigcup_{i \in I} \ZZ(Q_i,\PP^N_\R)$ by the closed semi-algebraic  subsets $S_i = \overline{V}(\R) \cap \ZZ(Q_i,\PP^N_\R)$. The $E_1$-page of this spectral sequence is given by
    \[
    E_1^{i,j} = \bigoplus_{H \subset [1,s], \card(H) = i+1} \HH^{j}(\overline{V}(\R) \cap \bigcap_{h \in H} \ZZ(Q_{h},\PP^N_\R)),
    \]
    with
    \[
    \HH^n(S) = \bigoplus_{i+j=n} E_\infty^{i,j}.
    \]
    The first inequality in the lemma is an immediate consequence.
\end{proof}

\begin{lemma}[Thom-Smith inequality]
\label{lem:main:ordered:c:5}
    Let $\overline{V} \subset \PP^N_\C$ be a real variety and $Q \in \R[X_0,\ldots,X_N]$ a homogeneous polynomial. Then,
    \[
    b(\overline{V}(\R) \cap \ZZ(Q,\PP^N_\R)) \leq b(\overline{V} \cap \ZZ(Q,\PP^N_\C)).
    \]
\end{lemma}

\begin{proof}
    If $X$ is a compact space  (or a regular complex)  equipped with an involution map $c:X \rightarrow S$, and 
$\mathrm{Fix}(c) \subset X$ denotes the subspace of fixed points of $X$. The Smith exact sequence
(see for example 
\cite[page 126]{Bredon-book})
then implies that
\begin{eqnarray}
\label{eqn:Smith}
b(\mathrm{Fix}(c),\Z/2\Z) &\leq& b(X,\Z/2\Z).
\end{eqnarray}
To prove the lemma, take the involution $c$ to be the complex conjugation. 
\end{proof}
\begin{lemma}
\label{lem:main:ordered:c:6}
    Let $\overline{V} \subset \PP^N_\C$ be a complete intersection
    non-singular variety with $\dim \overline{V} = p, \deg(\overline{V}) = D$.
    Let $P_1,\ldots,P_q \in \C[X_0,\ldots,X_N]$ be generic homogeneous polynomials with $\deg(P_i) \leq d$, $D \leq d$, $q \leq p$.
    Then,
    \[
        b(\overline{V} \cap \bigcap_{i}\ZZ(P_i,\PP^N_\C)) \leq 2^{N+1} \cdot D \cdot  d^p.
    \]
\end{lemma}

\begin{proof}
If $\overline{V} \subset \PP^N_\C$ is a non-singular complete intersection variety
with $\dim \overline{V} = N - \ell$, and defined as the intersection of non-singular hypersurfaces of degrees $e_1,\ldots,e_{\ell}$, then
$b(\overline{V})$ is given by the following classical formula
(see for example \cite[\S 5.7.2-3]{Eisenbud-Harris})
\begin{eqnarray}
\nonumber
b(\overline{V}) &=& (1 + (-1)^{N-\ell+1})\cdot (N-\ell+1) + \\
\nonumber
&& e_1 \cdot e_2\cdots e_\ell\cdot \left(\sum_{i=0}^{N-\ell}(-1)^{i}\cdot\binom{N+1}{i} \cdot h_{N-\ell -i}(e_1,\ldots,e_\ell)  \right) \\
\label{eqn:proof:lem:main:ordered:c:6:1}
&=& (1 + (-1)^{N-\ell+1})\cdot (N-\ell+1) + \\
\nonumber
&& e_1 \cdot e_2\cdots e_\ell\cdot \left(\sum_{j=0}^{N-\ell}(-1)^{N-\ell -j}\cdot\binom{N+1}{j+\ell+1} \cdot h_{j}(e_1,\ldots,e_\ell)  \right),
\end{eqnarray}
where
$h_j(e_1,\ldots,e_\ell)$
is the complete homogeneous symmetric polynomial of degree $j$ in $(e_1,\ldots, e_\ell)$.

In our case, we can assume that the degree sequence of the polynomials defining $\overline{V}$ is of length $\ell= N-p+q$. Suppose that the degree sequence is
$e_1,\ldots,e_{N-p+q}$. We can assume without loss of generality that
$e_1 \cdots e_{N-p} =D$, and $e_{N-p+1} \leq  d,\ldots,e_{N-p+q} \leq d$.

Thus, 
\begin{equation}
\label{eqn:proof:lem:main:ordered:c:6:2}
e_1 \cdots e_\ell \leq  D \cdot d^q,
\end{equation}
and
\begin{eqnarray}
\nonumber
 \left(\sum_{j=0}^{N-\ell}(-1)^{N-\ell -j}\cdot\binom{N+1}{j+\ell+1} \cdot h_{j}(e_1,\ldots,e_\ell)  \right) 
&\leq& 
2^{N+1} h_{N-\ell}(e_1,\ldots,e_\ell) \\
\label{eqn:proof:lem:main:ordered:c:6:3}
&\leq &
2^{N+1} d^{N-\ell} \\
\nonumber
&=& 2^{N+1} d^{p-q}. 
\end{eqnarray}

The lemma now follows from the equality
\eqref{eqn:proof:lem:main:ordered:c:6:1}, and 
the inequalities 
\eqref{eqn:proof:lem:main:ordered:c:6:2}
and
\eqref{eqn:proof:lem:main:ordered:c:6:3}.
\end{proof}

\begin{proof}[Proof of Theorem~\ref{thm:main:ordered:c}]
   Let $\overline{V} \subset \PP_\C^N$ be a non-singular complete intersection real variety 
with $\deg(V) = D$ and $\dim(V) = p$, and 
$V = \overline{V} \cap \C^N$.

  Let
    $\mathcal{P} \subset \R[X_1,\ldots,X_N]_{\leq d}$ be a finite set of polynomials with $d \geq D$ and  $\card(\mathcal{P}) = s$.

First observe that it follows from the theorem of conic structure at infinity \cite[Proposition 5.49]{BPRbook2}, that for all large enough $R > 0$, 
\[
\sum_{\sigma \in \{0,1,-1\}^{\mathcal{P}}} b_0(\RR(\sigma,V(\R))) 
\leq
\sum_{\sigma \in \{0,1,-1\}^{\mathcal{P}}} b_0(\RR(\sigma,V(\R) \cap \overline{B_N(0,R)})).
\]

Now let $\mathcal{P}_0 = \mathcal{P} \cup \{\sum_{i=1}^N X_i^2 - R\}$.
We have that 
\begin{eqnarray*}
\sum_{\sigma \in \{0,1,-1\}^{\mathcal{P}}} b_i(\RR(\sigma,V(\R)))
&\leq&
\sum_{\sigma \in \{0,1,-1\}^{\mathcal{P}}} b_i(\RR(\sigma,V(\R) \cap \overline{B_N(0,R)})) \\
&\leq&
\sum_{\sigma' \in \{0,1,-1\}^{\mathcal{P}_0}, 
C \in \Cc_b(\RR(\sigma',V(\R)))} b_i(C) \\
&\leq& 
\sum_{\widetilde{\sigma} \in \{1,-1\}^{\widetilde{\mathcal{P}}}} b_i(\RR(\widetilde{\sigma},V(\R\la\eps,\delta\ra)))  \text{ (using Lemma~\ref{lem:main:ordered:c:2}) }\\
&\leq& 
b(\overline{V}(\R)) + b_{p-i-1}(\bigcup_{\widetilde{P} \in \widetilde{\mathcal{P}}} \ZZ(\widetilde{P}^h,\overline{V}(\R\la\eps,\delta\ra))) \\
&&\text{ (using Lemma~\ref{lem:main:ordered:c:3})}.
\end{eqnarray*}

The theorem now follows from Lemmas~\ref{lem:main:ordered:c:4},
\ref{lem:main:ordered:c:5} and \ref{lem:main:ordered:c:6}.
\end{proof}

\subsubsection{Outline of the proof of Theorem~\ref{thm:main:ordered:c'}}
In Theorem~\ref{thm:main:ordered:c'} (unlike in Theorem~\ref{thm:main:ordered:c}) we do not assume that the given complete intersection is
non-singular. Our conclusion is also weaker, both from the point of view of the bound and the quantity being bounded. 

The main idea is that
after making an initial perturbation we can restrict our attention to strict conditions on the perturbed family of polynomials $\widetilde{P}$
(Lemma~\ref{lem:main:ordered:c':1}).

We now make an infinitesimal perturbation of the polynomials defining the complete intersection $V$ to obtain a non-singular complete intersection $\widetilde{V} \subset \widetilde{\C}^N$ 
(Lemma~\ref{lem:main:ordered:c':2})
We cannot ensure that
if a semi-algebraically connected component $C$ of $\mathcal{P}$ meets $V(\R)$, then
the corresponding semi-algebraically connected component $\widetilde{C}$ of the realization of a strict sign condition on $\widetilde{P}$ will meet $\widetilde{V}(\R)$. Nevertheless, identifying $\C^N$ with
$\R^{2N}$, and the interpreting the polynomials in $\mathcal{P}$ and $\widetilde{\mathcal{P}}$ as polynomials in $2N$ variables, it is true that 
the variety $\widetilde{V} \subset \C^N = \R^{2N}$ will meet 
$\widetilde{C} \times \widetilde{\R}^N$
(Lemma~\ref{lem:main:ordered:c':4}).

Moreover, separating the real and complex parts of the polynomials defining the variety $\widetilde{V}$ one obtains a complex non-singular complete intersection in $\PP^{2 N}_\C$, of degree 
$D^2$ and dimension $2p$. We can now apply Theorem~\ref{thm:main:ordered:c}.

\subsubsection{Proof of Theorem~\ref{thm:main:ordered:c'}}
We need the following lemmas.
\begin{lemma}
\label{lem:main:ordered:c':1}
    Let $\mathcal{P} = \{P_1,\ldots,P_s\} \subset \R[X_1,Y_1,\ldots,X_N,Y_N]$ be a finite set, and
    $P_0 = \sum_{i=1}^{N} (X_i^2 + Y_i^2) - R$, where $R >0, R\in \R$.
    For each $i,0 \leq i \leq s$, let $H_i \in \R[X_1,Y_1,\ldots,X_N,Y_N]$ be a generic polynomial with, $H_i(x) > 0$ for all $x \in \overline{B_{2N}(0,R)}$, and $\deg(H_i) = \deg(P_i)$. 
    
    For each $\sigma \in \{0,1,-1\}^{\mathcal{P}}$, and 
    $C \in \Cc(\RR(\sigma, \overline{B_{2 N}(0,R)})$, there exists
    a unique $\widetilde{\sigma} \in \{-1,1\}^{\widetilde{\mathcal{P}}}$ and 
    a unique 
    $\widetilde{C} \in \Cc(\RR(\widetilde{\sigma},\overline{B_{2N}(0,R)}))$,
    such that $\widetilde{C} \cap \R^{2N} = C$, 
    where 
    \[
    \widetilde{P} = \bigcup_{0 \leq i \leq s} \{P_i\pm \delta_i H_i, P_i \pm \eta_i H_i\}.
    \]
\end{lemma}

\begin{proof}
    Similar to the proof of Lemma~\ref{lem:main:ordered:c:1} and omitted.
\end{proof}

\begin{lemma}
\label{lem:main:ordered:c':2}
    Let $V \subset \C^N$ be a complete intersection variety of dimension $p$ defined by polynomials 
    $\mathcal{Q} = \{Q_1,\ldots,Q_{N-p}\} \subset \C[Z_1,\ldots,Z_N]$.
    
    Then there exists polynomials 
    $\widetilde{\mathcal{Q}} = \widetilde{Q}_{1},\ldots,\widetilde{Q}_{N-p} \subset \C\la\eps\ra[Z_1,\ldots,Z_N]$ such that:
    \begin{enumerate}[(a)]
    \item 
    \label{itemlabel:lem:main:ordered:c:1:a}
    for each $i, 1\leq i \leq N-p$, $\deg(\widetilde{Q}_i) = \deg(Q_i)$;
    \item 
    \label{itemlabel:lem:main:ordered:c:1:b}
    the homogenizations $\widetilde{Q}_1^h,\ldots,\widetilde{Q}_{N-p}^h$ define a non-singular complete intersection in $\PP_{\C\la\eps\ra}^{N}$;
    \item 
    \label{itemlabel:lem:main:ordered:c:1:c}
    for all $R \in \R, R > 0$,  
    $\lim_\eps \ZZ(\widetilde{\mathcal{Q}},B_{2N}(0,R+1)) \cap \overline{B_{2N}(0,R)} = \ZZ(\mathcal{Q},\overline{B_{2N}(0,R)})$,
    where
    $B_{2N}(0,R) \subset \C^N$ is the open ball of radius $R$ in $\C^N = \R^{2 N}$.
    \end{enumerate}
\end{lemma}

\begin{proof}
For $1 \leq i \leq N-p$,  let $H_1,\ldots,H_{N-p}$ be  
generically chosen polynomials with $\deg(H_i) = \deg(Q_i)$. Now let
$\widetilde{Q}_i = (1-\eps)Q_i + \eps H_i$. Then properties 
\eqref{itemlabel:lem:main:ordered:c:1:a} and \eqref{itemlabel:lem:main:ordered:c:1:b} follows from the fact that
same holds for the polynomials $H_1,\ldots,H_{N-p}$ and a standard 
transfer argument.

For property \eqref{itemlabel:lem:main:ordered:c:1:c} first observe that clearly 
\[
\lim_\eps \ZZ(\widetilde{\mathcal{Q}},B_{2N}(0,R+1)) \subset \ZZ(\mathcal{Q},\overline{B_{2N}(0,R+1)})
\]
from which it follows that 
\[
\lim_\eps \ZZ(\widetilde{\mathcal{Q}},B_{2N}(0,R+1)) \cap \overline{B_{2N}(0,R)} = \ZZ(\mathcal{Q},\overline{B_{2N}(0,R)}).
\]

To prove the reverse inclusion let $x \in (\mathcal{Q},\overline{B_{2N}(0,R)})$. We prove that there exists $\widetilde{x} \in \ZZ(\widetilde{\mathcal{Q}},B_{2N} (0,R+1))$ such that 
$\lim_\eps \widetilde{x} = x$.

Let $\PP^{N-p}_\C \subset \PP^{N}_\C$ be a generic $(N-p)$-dimensional plane through $x$. Then, using Bezout's theorem,  
$\ZZ(\mathcal{Q}^h,\PP^{N-p}_\C)$ contains $x$ and has 
$d_1 \cdots d_{N-p}$ points counted with multiplicities. 
But $\ZZ(\widetilde{\mathcal{Q}}^h,\PP^{N-p}_\C)$ also contains
$d_1 \cdots d_{N-p}$ points each of multiplicity one, and moreover
$\lim_\eps \ZZ(\widetilde{\mathcal{Q}}^h,\PP^{N-p}_\C) = \ZZ(\widetilde{\mathcal{Q}}^h,\PP^{N-p}_\C)$. From this it follows that
there must exists $\widetilde{x} \in \ZZ(\widetilde{\mathcal{Q}}^h,\PP^{N-p}_\C)$ such that $\lim_\eps \widetilde{x} = x$, and clearly
$\widetilde{x} \in \ZZ(\widetilde{\mathcal{Q}},B_{2N}(0,R+1))$.
\end{proof}

Now let $V \subset \C^N$ be a complete intersection variety of dimension $p$ defined by polynomials  
\[
\mathcal{Q} = \{Q_1,\ldots,Q_{N-p}\} \subset \C[Z_1,\ldots,Z_N],
\]
and let $\widetilde{\mathcal{Q}}$ be the set of corresponding polynomials
in Lemma~\ref{lem:main:ordered:c':2}.

\begin{lemma}
\label{lem:main:ordered:c':3}
    Let $\widetilde{P} \subset \R[X_1,Y_1,\ldots,X_N,Y_N]$ be a finite set, $\sigma \in 
    \{-1,1\}^{\widetilde{\mathcal{P}}}$, $C \in \Cc(\RR(\sigma,\C^N))$ such that $C \cap V \neq \emptyset$. Then, 
    $\E(C,\C\la\eps\ra) \cap \widetilde{V} \neq \emptyset$,
    where $\widetilde{V} = \ZZ(\widetilde{\mathcal{Q}},\C\la\eps\ra^N)$.
\end{lemma}

\begin{proof}
This follows immediately from  Lemma~\ref{lem:main:ordered:c':1} noting that $\RR(\sigma,\C^N)$ is an open semi-algebraic subset.  
\end{proof}

\begin{lemma}[Reduction from weak to strict sign conditions]
\label{lem:main:ordered:c':4}
Let $V,\widetilde{V}$ be as in Lemma~\ref{lem:main:ordered:c':3} and
$\mathcal{P},\widetilde{\mathcal{P}}$ be as in Lemma~\ref{lem:main:ordered:c':1}.
Then
\begin{align*}
\sum_{\sigma \in \{0,1,-1\}^{\mathcal{P}}} \card(\Cc(\RR(\sigma),V(\R)))&\leq& 
\sum_{\widetilde{\sigma} \in \{1,-1\}^{\widetilde{\mathcal{P}}}} 
\card(\Cc(\RR(\widetilde{\sigma}),\widetilde{V})) \\
&\leq&
\sum_{\widetilde{\sigma} \in \{1,-1\}^{\widetilde{\mathcal{P}}}} 
b_0(\RR(\widetilde{\sigma},\widetilde{V})).
\end{align*}
(Note that in the above inequality $\widetilde{V}$ is being considered as a semi-algebraic subset of $\widetilde{\C}^N = \widetilde{\R}^{2 N}$.)
\end{lemma}

\begin{proof}
    Follows immediately from (Lemma~\ref{lem:main:ordered:c':3}).
\end{proof}

\begin{proof}[Proof of Theorem~\ref{thm:main:ordered:c'}]
 Using Lemma~\ref{lem:main:ordered:c':4}, it suffices to bound   
\[
\sum_{\widetilde{\sigma} \in \{1,-1\}^{\widetilde{\mathcal{P}}}} 
b_0(\RR(\widetilde{\sigma},\widetilde{V})).
\]

Now let for $j=1,\ldots,N-p$, let 
\[
\widetilde{Q}_j = \widetilde{Q}_{j,0} + \sqrt{-1}\cdot \widetilde{Q}_{j,1},
\]
where $\widetilde{Q}_{j,0}, \widetilde{Q}_{j,1} \in \R[X_1,Y_1,\ldots,X_n,Y_n]$ have generic coefficients.

Thus, 
\[
\widetilde{Q}_{1,0},\widetilde{Q}_{1,1},\ldots,
\widetilde{Q}_{N-p,0},\widetilde{Q}_{N-p,1}
\]
define a non-singular complete intersection in $\widetilde{W} \subset \C^{2N}$, with $\deg(\widetilde{W}) = D^2$, and $\dim \widetilde{W} = 2p$
and with $\widetilde{W}(\R) = \widetilde{V}$.

Now apply Theorem~\ref{thm:main:ordered:c} with non-singular complete intersection $\widetilde{W}$ and the family of polynomials
$\widetilde{\mathcal{P}}$ to obtain
\begin{eqnarray*}
\sum_{\widetilde{\sigma} \in \{1,-1\}^{\widetilde{\mathcal{P}}}} 
b_0(\RR(\widetilde{\sigma},\widetilde{V}))
&\leq&
\sum_{\sigma \in \{0,1,-1\}^{\widetilde{\mathcal{P}}}} 
b_0(\RR(\widetilde{\sigma},\widetilde{W}(\R))) \\
&\leq&
(D^2)^2 \cdot (O(s d))^{2 p} \\
&=&
D^4 \cdot (O(s d))^{2 p}.
\end{eqnarray*}
\end{proof}

\subsection{Proof of Theorem~\ref{thm:BPR+}}

\subsubsection{Outline of the proof of Theorem~\ref{thm:BPR+}}
The proof borrows from the proof of Theorem~\ref{thm:main:ordered:c}
the main technique of making perturbations of the given family $\mathcal{P}$, and using the inequalities arising from the Mayer-Vietoris spectral sequence to reduce to the case of bounding the Betti numbers of certain non-singular projective varieties. 
Instead of using the Thom-Smith inqualities and the Chern class formula
as we do in the proof of Theorem~\ref{thm:main:ordered:c}, we make a further reduction using the same technique as in the proof of Theorem~\ref{thm:main:ordered:a} to reduce to bounding the Betti numbers of non-singular real projective varieties of dimensions $p-1$ (and also 
bounding that of $V(\R)$ which is of dimension $p$). We use in this case
the bound provided by Theorem~\ref{thmx:LV} and the adjoining Remark~\ref{rem:LV} to bound these.

\subsubsection{Proof of Theorem~\ref{thm:BPR+}}
We already have the necessary ingredients to start the proof without having to prove additional lemmas.

\begin{proof}[Proof of Theorem~\ref{thm:BPR+}]
    Using the conic structure at infinity for semi-algebraic sets
    (see for example \cite[Proposition 5.49]{BPRbook2}) choose
    $R > 0, R \in \R$, such that
    for each $\sigma \in \{0,1,-1\}^{\mathcal{P}}$, 
    $\RR(\sigma,V(\R))$ is semi-algebraically homeomorphic to 
    $\RR(\sigma,V(\R) \cap B_N(0,R))$.
    We denote $\mathcal{P}_0 = \mathcal{P} \cup P_0$,
    where $P_0 = \sum_{i=1}^N X_i^2 - R$.
    
    Applying Lemma~\ref{lem:main:ordered:c:2} (using the same notation as in Lemma~\ref{lem:main:ordered:c:2}) we have
    $$\displaylines{
    \sum_{\sigma \in \{0,1-1\}^{\mathcal{P}}} b_i(\RR(\sigma,V(\R))) 
    \leq 
    \sum_{\widetilde{\sigma} \in \{1,-1\}^{\widetilde{P}_0}} b_i(\RR(\widetilde{\sigma}),\E(V(\R),\widetilde{\R}),
    }
    $$
    and finally using Lemma~\ref{lem:main:ordered:c:3} that,
$\sum_{\widetilde{\sigma} \in \{1,-1\}^{\widetilde{P}_0}} b_i(\RR(\widetilde{\sigma}),\E(V(\R),\widetilde{\R}))$
is bounded by 
\[
b_{p-i-1}\left(\bigcup_{\widetilde{P} \in \widetilde{\mathcal{P}}_0} \ZZ(\widetilde{P}^h,
  \E(\overline{V}(\R),\widetilde{\R}))  
  \cup
  \ZZ(X_0, \E(\overline{V}(\R),\widetilde{\R}))
\right)
+ b_{p-i}( \E(\overline{V}(\R),\widetilde{\R})).
\]

We now use the Mayer-Vietoris spectral sequence to observe that 
the right hand side of the previous equation is bounded by the 
sum of three terms (we use the convenient notation $A \subset_n B$ to denote $A \subset B$ and $\card(A) = n$):
\begin{enumerate}[(a)]
    \item 
    \label{itemlabel:proof:thm:BPR+:1}
    \[
    \sum_{0 \leq j \leq p-i-1} \sum_{\widetilde{\mathcal{P}}_0' \subset_{j+1} \widetilde{\mathcal{P}}_0} b_{p-i-1-j}(\ZZ((\widetilde{\mathcal{P}}_0')^h, \E(\overline{V}(\R),\widetilde{\R}))),
    \]
    \item 
    \label{itemlabel:proof:thm:BPR+:2}
    \[
    \sum_{0 \leq j \leq p-i-2} \sum_{\widetilde{\mathcal{P}}_0' \subset_{j+1} \widetilde{\mathcal{P}}_0} b_{p-i-1-j}(\ZZ(\{X_0\} \cup (\widetilde{\mathcal{P}}_0')^h, \E(\overline{V}(\R),\widetilde{\R}))),
    \]
    
    \item 
    \label{itemlabel:proof:thm:BPR+:3}
    \[
    b( \E(\overline{V}(\R),\widetilde{\R})).
    \]
    
\end{enumerate}

We are now going to bound from above each one of the three items.
Firstly, observe that using Remark~\ref{rem:LV} 
and the fact that
$V$ is non-singular we can bound the quantity in Part~\eqref{itemlabel:proof:thm:BPR+:3} by $D^{p+1}$.

We now bound for each 
$\widetilde{\mathcal{P}}_0' \subset_{j+1} \widetilde{\mathcal{P}}_0$, the Betti number
\[
b_{p-i-1-j}(\ZZ((\widetilde{\mathcal{P}}_0')^h, \E(\overline{V}(R),\widetilde{\R}))).
\]

Note that $\ZZ((\widetilde{\mathcal{P}}_0')^h, \E(\overline{V}(\R),\widetilde{\R}))$ is the real part of a non-singular projective hypersurface with degree bounded by $D d^{j+1}$ and dimension $p-j -1$.
Using Remark~\ref{rem:LV} and the Thom-Smith inequality will give us
\[
b_{p-i-1-j}(\ZZ((\widetilde{\mathcal{P}}_0')^h, \E(\overline{V}(\R),\widetilde{\R}))) \leq (D d^{j+1})^{p-j} = D^{p-j} d^{(j+1)(p-j)}.
\]

Using perturbation techniques we prove a slightly better bound as follows.

Let $S = \ZZ((\widetilde{\mathcal{P}}_0')^h, \E(\overline{V}(\R),\widetilde{\R}))$.

First let 
\[
\widetilde{Q} = \sum_{\widetilde{P} \in \widetilde{\mathcal{P}}_0'} \widetilde{P}^2. 
\]

Let $H \in \widetilde{\R}[X_0,\ldots,X_N]$ be a generic homogeneous polynomial
of the same degree as $\widetilde{Q}$ and which in non-negative in $\R^{N+1}$.
Note that in $\PP_\R^N$ the sign of a homogeneous polynomial of even degree is well defined. It follows from \cite[Lemma 16.17]{BPRbook2} that
the semi-algebraic set $\widetilde{S} \subset \E(\overline{V}(R),\widetilde{R}\la\zeta\ra))$ defined as the intersection of 
$\E(\overline{V}(R),\widetilde{R}\la\zeta\ra))$ with the semi-algebraic set defined by $(1-\zeta)\widetilde{Q} - \zeta H \leq 0$ is semi-algebraically homotopy equivalent to $ \E(S,\widetilde{\R}\la\zeta\ra)$.

Moreover, it follows from the Mayer-Vietoris exact sequence that 
\[
b(\widetilde{S}) \leq b(\ZZ((1-\zeta)\widetilde{Q} - \zeta H, \E(\overline{V}(\R),\widetilde{\R}\la\zeta\ra))) + b(V(\R)).
\]
Now observe that $\ZZ((1-\zeta)\widetilde{Q} - \zeta H, \E(\overline{V}(\R),\widetilde{\R}\la\zeta\ra))$ is the real part of a non-singular projective variety of degree bounded by $2 D d$ and dimension $p-1$.
It follows (using Remark~\ref{rem:LV} and the Thom-Smith inequality again) that 
\[
b(\ZZ((1-\zeta)\widetilde{Q} - \zeta H, \E(\overline{V}(\R),\widetilde{\R}\la\zeta\ra))) \leq (2 D d)^{p}.
\]
Thus, the sum in Part~\eqref{itemlabel:proof:thm:BPR+:1} is bounded by
\[
\left(\sum_{j=0}^{p-i-1} \binom{s}{j+1}\right)(D^p(2d)^{p} + D).
\]

We use the same method to bound the sum in Part~\eqref{itemlabel:proof:thm:BPR+:2}, noting that all subvarieties are
contained in the $\PP_\R^{N-1}$ defined by $X_0=0$. This yields that the
sum in Part~\eqref{itemlabel:proof:thm:BPR+:2} is bounded by 
\[
\left(\sum_{j=0}^{p-i-2} \binom{s}{j+1}\right)(D^{p-1}(2d)^{p-1} + D).
\]

Finally, summing the three parts together we obtain that:
\begin{eqnarray*}
\sum_{\sigma \in \{0,1-1\}^{\mathcal{P}}} b_i(\RR(\sigma,V(\R))) &\leq& 
 \left(\sum_{j=0}^{p-i-1} \binom{s}{j+1}\right)(D^p(2d)^{p} + D) + \\
&& \left(\sum_{j=0}^{p-i-2} \binom{s}{j+1}\right)(D^{p-1}(2d)^{p-1} + D) + \\
&&
D^{p+1} \\
&\leq& O(1)^p s^{p-i} D^p(d^p + D).
\end{eqnarray*}
\end{proof}

\subsection{Proof of Theorem~\ref{thm:entropy}}
We first recall a notion from \cite{Zhang-Kileel2023} that plays a key role.

\begin{definition}
\label{def:reg_set_informal}
    A set $V \subset \R^N$ is called a $(K, n)$-\emph{regular set} if
    \begin{enumerate}[(a)]
        \item For almost all affine planes $L$ of codimension $n' \leq n$ in $\R^N$, $V \cap L$ has at most $K$ path connected components. 
        \item For almost all affine planes $L$ of codimension $n' \geq n+1$ in $\R^N$, $V \cap L$ is empty.
    \end{enumerate}
\end{definition}

The main theorem in \cite{Zhang-Kileel2023} is the following theorem.

\begin{thmx}
Let $V$ be $(K, n)$-regular set in the cube $[-1, 1]^N$, the covering number is bounded by
\begin{equation}\label{eq:main-bound}
    \log\mathcal{N}(V, \varepsilon) \leq n \log(1/\varepsilon) + \log K + C n \log N,
\end{equation}
for some absolute constant $C$.
\end{thmx}

\begin{proof}[Proof of Theorem~\ref{thm:entropy}]
We begin by observing that in this proof we work over $\R=\mathbb R$, and that
\[
V(\R)\cap[-1,1]^N = V(\R),
\]
since, by hypothesis, $V(\R)$ is contained in the unit ball.

We claim that $V(\R)_p$ is a $(K,p)$-regular set with
\[
K = 2^{2p+4}D^{p+1}.
\]
Let $L\subset \R^N$ be a generic affine subspace of codimension $q\le p$.
By genericity,
\[
V(\R)_p\cap L = \bigl(V(\R)\cap L\bigr)_{p-q}.
\]
Applying \eqref{eqn:thm:main:ordered:a'} from Theorem~\ref{thm:main:ordered:a'} to the variety
$V\cap L$, we obtain
\begin{align*}
b_0\bigl(V(\R)_p\cap L\bigr)
&= b_0\bigl((V(\R)\cap L)_{p-q}\bigr) \\
&\le 2^{2(p-q)+4}\,D^{p-q+1}
\;\le\; 2^{2p+4}\,D^{p+1}.
\end{align*}
If instead $\mathrm{codim}(L)>p$, then $V(\R)\cap L=\emptyset$, and hence
$b_0(V(\R)_p\cap L)=0$.  This verifies the $(K,p)$-regularity of $V(\R)_p$ with the above
choice of $K$.  The theorem now follows immediately from \eqref{eq:main-bound}.
\end{proof}

\subsection{Proof of Theorem~\ref{thm:lower-bound}}

\begin{proof}[Proof of Theorem~\ref{thm:lower-bound}]
    First suppose that $S \subset V(\R)$, and $V$ is a union of non-singular real varieties, with $\dim V = p$ and $\deg V = D$. Suppose there exists an algebraic computation tree of height $t$ testing membership in $S$. Then, the number of leaves of the tree (in particular the number of leaves labelled YES) is bounded by $2^t$. Let $w$ denote a leaf of the tree labelled YES. Then the path from the root to $w$ defines a basic semi-algebraic set $S_w \subset V(\R) \times \R^{t_w}$,
    where $t_w \leq t$ is the number of variables introduced in the path to $w$ (i.e. number of computation vertices on this path, see Definition~\ref{def:tree}), and $S_w$ is defined by a set $\mathcal{P}_w$ of at most $t$ polynomials of degree at most $2$.

    Then, $b_0(S_w)$ is bounded by 
    $2^t O(D)^{p+t_w} = O(D)^{p+t}$  using Theorem~\ref{thm:main:ordered:a}.

    This implies that 
    \[
    b_0(S) \leq O(D)^{p+t}.
    \]
    
    Taking logarithm of both sides we obtain the desired lower bound on $t$.
\end{proof}

\section{Relative ranks in vector spaces and algebras and quantum circuit complexity}
\label{subsec:rank}
The notion of rank of tensors plays a very important role
in many applications -- from understanding the complexity of matrix multiplication, to machine learning and  quantum information theory
(see for example \cite{Landsberg2019} for an overview). We introduce here 
a notion of relative rank for finite dimensional vector spaces and algebras (relative to an arbitrary algebraic subset $\Delta$ of the vector space or algebra) which generalizes the notion of rank for tensors. We prove lower bounds on the relative rank for algebraically defined subsets $W$ of the vector space or algebra,  in terms of the dimension and degree of $\Delta$ and the cardinality (or the number of semi-algebraically connected components) of the set (Theorems~\ref{thm:rank:vector:algebraic}, \ref{thm:rank:vector:ordered},
\ref{thm:rank:algebra:algebraic} and \ref{thm:rank:algebra:ordered} below).

Note that while proving upper bounds on the classical rank of specific tensors, such as the matrix multiplication tensor (see for example \cite{Landsberg2017,Landsberg2019}) is an important and active area of research, proving lower bounds on the rank of sets of 
tensors (particularly in the relative setting) has not been equally well studied. We believe the results stated below will prove useful in applications in addition to quantum complexity theory that we discuss here. 

\subsection{Relative rank in vector spaces}
\begin{definition}[Relative rank in vector spaces]
\label{def:rank:vector}
Let $V$ be a $k$-vector space, and $\Delta \subset V$. For $v \in V$, we define
\[
\rank_{V,\Delta}(v)
\]
as the smallest number $t$, if it exists, such that there exists $v_1,\ldots,v_t \in \Delta$ such that $v = v_1 + \cdots + v_t$,
and $\infty$ otherwise. We will call
$\rank_{V,\Delta}(v)$ the \emph{rank of $v$ relative to $\Delta$}.

For any subset $W \subset V$, we will denote 
\[
\rank_{V,\Delta}(W) = \max_{w \in W} \rank_{V,\Delta}(w).
\]
\end{definition}
\begin{remark}
\label{rem:rank:vector:monotone}
    Note that if $\Delta \subset \Delta'$, then $\rank_{V,\Delta}(v) \geq \rank_{V,\Delta'}(v)$ for all $v \in V$, and 
    $\rank_{V,\Delta}(W) \geq \rank_{V,\Delta'}(W)$ for all subsets 
    $W \subset V$.
\end{remark}

\begin{remark}
    Notice that Definition~\ref{def:rank:vector} generalizes the standard notion of rank of a tensor. For example, if $V = V_1 \otimes \cdots \otimes V_m$, and $\Delta = \{v_1 \otimes \cdots \otimes v_m \mid v_i \in V_i, 1 \leq i \leq m\}$ the subset of tensors of rank at most one, then for all $v \in V$, $\rank_{V,\Delta}(v)$ is the usual rank of $v$ as an element of the tensor product  $V_1 \otimes \cdots \otimes V_m$.
\end{remark}

\begin{theorem}
\label{thm:rank:vector:algebraic}
Let $k$ be an algebraically closed field and $V$ a finite dimensional
vector space.
Let $\Delta \subset V$ be an algebraic subset with $\dim(\Delta) = p$ and $\deg(\Delta)  = D$, and 
$\mathcal{P} \subset k[V]_{\leq \delta}$ with $\card(\mathcal{P}) = s$.
Then, for any finite $\mathcal{P}$-constructible subset $W \subset V$ 
\begin{equation}
\label{eqn:thm:rank:vector:algebraic}
\rank_{V,\Delta} (W) \in  \Omega \left(\frac{\log(\card(W))}{p (\log s + \log \delta + \log \log \card(W)) + \log D }\right)
\end{equation}
(where the constant in the $\Omega(\cdot)$ is an absolute constant).
\end{theorem}

In the real closed case we have a similar lower bound (with
an added equi-dimensional assumption on $V$). 

Let $\R$ be a real closed field, $V$ a finite dimensional $\R$-vector space and 
$\Delta \subset V$.
The following definition generalizes the notion of rank in Definition~\ref{def:rank:vector}.

\begin{definition}[Relative max-min rank in real vector spaces]
\label{def:mmrank:V}
Let $W$ be a semi-algebraic subset of $V$ (not necessarily finite). 
We define the max-min-rank, 
\[
\mmrank_{V,\Delta}(W),
\]
of $W$ relative to $\Delta$ by
\[
\mmrank_{V,\Delta}(W) = \max_{C \in \Cc(W)} 
\min_{\alpha \in C} \rank_{V,\Delta}(\alpha).
\]
(Note that if $W$ is finite 
$\mmrank_{V,\Delta}(W) = \rank_{V,\Delta}(W)$.)
\end{definition}

\begin{theorem}
\label{thm:rank:vector:ordered}
Let $\R$ be a real closed field, $\C=\R[i]$, and $V$ a finite-dimensional $\R$-vector space. Let $\Delta\subset V_\C=V\otimes_\R \C$ be a real algebraic set with irreducible components $\Delta_1,\dots,\Delta_m$, $\dim\Delta=p$, $\deg\Delta=D$, such that
\[
\dim_x\Delta_i(\R)=\dim\Delta_i\qquad\text{for all $i$ and all }x\in\Delta_i(\R).
\]
Let $\mathcal{P}\subset\R[V]_{\le\delta}$, $\delta \geq 2$, 
with $\card(\mathcal{P})=s$, and let $W\subset V$ be a $\mathcal{P}$-semi-algebraic set. Then, with rank taken relative to $\Delta(\R)=\Delta\cap V$,
\begin{equation}
\label{eqn:thm:rank:vector:ordered:1}
\mmrank_{V,\Delta}(W)\ = \Omega\left( \frac{\log b_0(W)}{(p+1)\bigl(\log s+\log\delta+\log D+\log\log b_0(W)\bigr)}\right).
\end{equation}

In particular, if $W$ is finite,
\begin{equation}
\label{eqn:thm:rank:vector:ordered:2}
\rank_{V,\Delta}(W)\ = \Omega\left(\frac{\log\card(W)}{(p+1)\bigl(\log s+\log\delta+\log D+\log\log\card(W)\bigr)}\right).
\end{equation}
\end{theorem}
 
\begin{remark}
\label{rem:thm:rank:vector:ordered:1}
The hypothesis of Theorem~\ref{thm:rank:vector:ordered} is weaker than the requirement that $\Delta$ be a union of non-singular real varieties, since non-singular real varieties are equidimensional at their real points. 
Also notice the factor $p+1$ multiplying $\log D$ 
which makes the bound in \eqref{eqn:thm:rank:vector:ordered:2} slightly worse than the one in \eqref{eqn:thm:rank:vector:algebraic}. This is 
intrinsic to our method.
\end{remark}

\begin{remark}
\label{rem:thm:rank:vector:ordered:2}
Note that we cannot use 
Theorem~\ref{thm:rank:vector:algebraic} with $k = C = \R[i]$
to deduce Theorem~\ref{thm:rank:vector:ordered} as a special case.
A finite $\mathcal{P}$-semi-algebraic set need not be realizable
as a finite $\mathcal{Q}$-constructible set over $\C$, where
the cardinality and the degrees of the polynomials in $\mathcal{Q}$ are 
proportional to those of $\mathcal{P}$.
\end{remark}

\subsection{Relative rank in algebras}
Let $k$ be a field and $A$ a finite-dimensional $k$-algebra. The algebra $A$ is allowed to be non-commutative or non-associative and we do not require $A$ to be unital.  For example, $A$ could be a finite-dimensional Lie algebra, the algebra of endomorphisms of a finite dimensional $k$-vector space, the group algebra
$k[G]$ for some finite group $G$ (or any sub-algebra of $k[G]$). In each of these cases, it is meaningful to ask
given some subset $\Delta$ of $A$ and an element $f\in A$, how many elements of $\Delta$ does one need to express $f$ as a product of elements of $\Delta$.

\begin{definition}[Relative rank in algebras]
\label{def:rank:algebra}
Let $A$ be a $k$-algebra and $\Delta \subset A$. 
For $\alpha \in A$, we define
\[
\rank_{A,\Delta}(\alpha)
\]
as the smallest number $t$, if it exists, such that there exists $\alpha_1,\ldots,\alpha_t \in \Delta$ such that $\alpha$  
is the product of $\alpha_1, \ldots, \alpha_t$ under some bracketing, and $\infty$ otherwise.
We will call
$\rank_{A,\Delta}(\alpha)$ the rank of $\alpha$ relative to $\Delta$.

For any subset $W \subset A$, we will denote 
\[
\rank_{A,\Delta}(W) = \max_{\alpha\in W} \rank_{A,\Delta}(\alpha).
\]
\end{definition}

\begin{remark}
\label{rem:rank:algebra:monotone}
    Note that if $\Delta \subset \Delta'$, then $\rank_{A,\Delta}(\alpha) \geq \rank_{A,\Delta'}(\alpha)$ for all $\alpha \in A$, and 
    $\rank_{A,\Delta}(W) \geq \rank_{A,\Delta'}(W)$ for all subsets 
    $W \subset A$.
\end{remark}

The following definition will play a role in application of the notion of rank of algebras to proving lower bounds on quantum circuit complexity lower discussed in the following section.
Let $\R$ be a real closed field, $A$ a finite dimensional $\R$-algebra and 
$\Delta \subset A$.
The following definition generalizes the notion of rank in Definition~\ref{def:rank:algebra}.

\begin{definition}[Relative max-min rank in real algebras]
\label{def:mmrank:algebra}
Let
$W$ be a semi-algebraic subset of $A$ (not necessarily finite). 
We define the max-min-rank, 
\[
\mmrank_{A,\Delta}(W),
\]of 
$W$ relative to $\Delta$ by
\[
\mmrank_{A,\Delta}(W) = \max_{C \in \Cc(W)} 
\min_{\alpha \in C} \rank_{A,\Delta}(\alpha).
\]
(Note that if $W$ is finite 
$\mmrank_{A,\Delta}(W) = \rank_{A,\Delta}(W)$.)
\end{definition}

Notice that a finite dimensional $k$-algebra $A$ has the structure of a vector space (by forgetting the 
multiplication) and it makes sense to consider algebraic subsets  of $A$. We denote the ring of polynomials on 
$A$ by $k[A]$ and the subspace of polynomials in $k[A]$ of degree at most $\delta$ by $k[A]_{\leq \delta}$.

\begin{theorem}
\label{thm:rank:algebra:algebraic}
Let $k$ be an algebraically closed field and $A$ a finite dimensional $k$-algebra.
Let $\Delta \subset A$ be an algebraic subset with $\dim(\Delta) = p$ and $\deg(\Delta)  = D$, and 
$\mathcal{P} \subset k[A]_{\leq \delta}$ with $\card(\mathcal{P}) = s$.
Then, for any finite $\mathcal{P}$-constructible subset $W \subset A$ 
\begin{equation}
\label{eqn:thm:rank:algebra:algebraic}
\rank_{A,\Delta} (W) =  \Omega \left(\frac{\log(\card(W))}{p(\log s + \log \delta + \log \log \card(W)) + \log D }\right)
\end{equation}
(where the constant in the $\Omega(\cdot)$ is an absolute constant).
\end{theorem}

We have the following theorem in the real closed case.

\begin{theorem}
\label{thm:rank:algebra:ordered}
Let $\R$ be a real closed field, $\C=\R[i]$,
$A$ a finite-dimensional $\R$-algebra, and
let $N=\dim_\R A$. Let $\Delta\subset A_\C=A\otimes_\R \C$ be a real algebraic set with irreducible components $\Delta_1,\dots,\Delta_m$, 
with $\dim\Delta=p$, $\deg\Delta=D$, such that
\[
\dim_x\Delta_i(\R)=\dim\Delta_i\qquad\text{for all $i$ and all }x\in\Delta_i(\R).
\]
Let $\mathcal{P} \subset \R[A]_{\leq \delta}$, $\delta \geq 2$, with $s=\card(\mathcal{P})$, and let $W\subset A$ be a $\mathcal{P}$-semi-algebraic set. 
Then, with rank taken relative to $\Delta(\R)=\Delta\cap A$,
\begin{equation}
\label{eqn:thm:rank:algebra:ordered}
\mmrank_{A,\Delta}(W)\ =  \Omega\left(\frac{\log b_0(W)}{(p+1)\,\bigl(\log s+\log\delta+\log D+\log\log b_0(W)\bigr)}\right).
\end{equation}
\end{theorem}

\begin{remark}
Remarks similar to 
Remarks~\ref{rem:thm:rank:vector:ordered:1} and \ref{rem:thm:rank:vector:ordered:2} made after Theorem~\ref{thm:rank:vector:ordered} apply to Theorem~\ref{thm:rank:algebra:ordered} as well.
\end{remark}

\subsection{Quantum circuit lower bounds}
\label{subsec:quantum}
\subsubsection{Shannon's lower bound for classical circuits}
Boolean circuit complexity measures complexity of Boolean functions $f: \{0,1\}^n \rightarrow \{0,1\}$ by the size of the smallest circuit computing  $f$.
We assume that the gates of a circuit come from some fixed finite universal set of gates.
We refer the reader to \cite[Page 7, Definition 3.1]{Wegener} for the precise definition of a circuit.
Notice that in this case the set of circuits of bounded size is finite. 

Shannon \cite{Shannon} showed using a counting argument (using the fact that as noted before that the set of circuits of bounded size is finite)  very early that almost all Boolean functions need circuits of size $\Omega({2^n}/{n})$ (see \cite[Page 90, Theorem 2.1]{Wegener}.
Here almost all refers to the fact that the number of Boolean functions that require circuits of size $\Omega({2^n}/{n})$ is bounded from below by $2^{2^n}(1 - o(1))$. It is also known (see for instance \cite[Page 92, Theorem 2.2]{Wegener}) that every Boolean
function $f: \{0,1\}^n \rightarrow \{0,1\}$ can be computed by a circuit of size $O({2^n}/{n})$.

\begin{remark}
\label{rem:arity-classical}
In the theorems cited above the set of gates used is assumed to be all possible gates 
of arity two (i.e. the finite set of $2^{2^2} = 16$ gates computing all possible Boolean functions 
$g:\{0,1\}^2 \rightarrow \{0,1\}$). The arity two is not important. The same asymptotic results will
hold even if we allow all gates of some bounded arity $q \geq 2$ 
(the constants will depend on the arity) .
\end{remark}

\subsection{Analog of Shannon's lower bound for quantum circuits}
We begin with a brief review of quantum complexity since it motivates our definition.
Recall that in quantum complexity theory a qubit refers to a two-dimensional complex Hermitian 
vector space $V$ with a fixed orthonormal basis $\mathcal{B} = \{|0\ra, |1\ra\}$. A $n$-qubit register refers to
the $n$-fold tensor power, $V^{\otimes n}$ of $V$, with an induced Hermitian structure in which the vectors
$\mathcal{B}^{\otimes n} = \{|j\ra | 0 \leq j \leq N-1\}$ form an orthonormal basis
(here $N = 2^n$, and we identify a $0$-$1$ string $x_{n-1}\cdots x_{0}$ of length $n$, with the integer
\[
x = \sum_{k=0}^{n-1} x_k \cdot 2^k,
\]
and denote $|x_{n-1}\ra\otimes \cdots \otimes |x_0\ra$ simply by $|\mathbf{x}\ra$).

A $q$-ary quantum gate $g$ is an element $U_g \in \mathbf{U}(V^{\otimes q})$. 
In a quantum circuit $C$,  with input $n$-qubits, such a gate operates on a subset of $q$ of
the $n$-qubits (say labelled by the tuple $\mathbf{i} = (i_1,\ldots,i_q), 1 \leq i_1 <\cdots <i_q \leq n $),
which gives rise to an element $U_{g,\mathbf{i}}$ of  $\mathbf{U}(V^{\otimes n})$ which can be identified with  
$U_g \otimes \mathrm{Id}_{V^{\otimes (n-q)}}$.
(after rearranging the order of the copies of $V$ in the tensor power $V^{\otimes n}$
bringing the ones indexed by $(i_1,\ldots,i_q)$ to the the positions $(1,\ldots,q)$).

A quantum circuit $C$ with $n$ input qubits labelled by $1,\ldots,n$ having $r$  
$q$-ary gates is determined by the following data:
\begin{enumerate}[1.]
    \item an ordering of the $r$ gates 
(lets suppose the ordered tuple of gates is $g_1,\ldots,g_r$,  and
\item 
for each $i, 1 \leq i \leq r$,  an ordered  choice, $\mathbf{i} = (i_1,\ldots,i_q)$ of $q$ elements
from amongst $1,\ldots,n$.
\end{enumerate}
We denote 
\[
U(C) = \prod_{1 \leq i \leq r} U_{g_i,\mathbf{i}} \in \mathbf{U}(V^{ \otimes n})
\]
the unitary transformation implemented by $C$ (where $N = 2^{n}$), and we denote $\size(C) = r$.

For $q \geq 1$, 
we will denote by $\mathcal{C}_q$, the set of quantum circuits 
using $q$-ary quantum gates.

\begin{definition}[Quantum circuit complexity of a unitary transformation]
\label{def:quantum-complexity}
    The $q$-ary quantum circuit complexity of a unitary transformation $U \in \mathbf{U}(V^{\otimes n})$ is defined 
    as 
    \[
        \min_{C \in \mathcal{C}_q, U(C) = U} \size(C).
    \]
\end{definition}
In other words, if we denote  
\[
\Delta = \Delta_{n,q} = \{ U_{g,\mathbf{i}} \mid U_g \in \mathbf{U}(V^{\otimes q}), 
\mathbf{i} = (i_1,\ldots,i_q), 1 \leq i_1 < \cdots < i_q \leq n\},
\]
then the $q$-ary quantum circuit complexity  of a unitary transformation $U \in \mathbf{U}(V^{\otimes n})$ 
is precisely (using Definition~\ref{def:rank:algebra})
\[
\rank_{\End(V^{\otimes n}),\Delta}(U)
\]
considering $U$ as an element of $\End(V^{\otimes n})$.
This motivates the following definition:
\begin{definition}[Relative quantum circuit complexity]
 \label{def:quantum-complexity-relative} 
For arbitrary $\Delta \subset \End(V^{\otimes n})$ we will call 
\[
\rank_{\End(V^{\otimes n}),\Delta}(U)
\]
the \emph{quantum complexity of $U$ relative to $\Delta$}.
\end{definition}

\begin{remark}
    \label{rem:oracle} 
    $\Delta$ should be thought of  
    as an algebraic family of allowed gates
    -- with unit cost for each application of an element of $\Delta$ in a quantum circuit. Notice though we do not assume that $\Delta$ is contained in the unitary group
    $\mathbf{U}(V^{\otimes n})$.
\end{remark}

\begin{remark}[Relation with nonunitary quantum computation]
Aaronson~\cite{Aaronson} considered the complexity class
\(\mathbf{BQP}_{\mathrm{nu}}\), obtained by allowing quantum circuits to use
arbitrary invertible linear transformations of bounded arity, rather than
only unitary transformations, with the resulting state normalized before
the final measurement. He proved that

$$
\mathbf{BQP}_{\mathrm{nu}}=\mathbf{PP}.
$$

The proof proceeds by showing that nonunitary evolution can simulate
postselection, together with the equality
\(\mathbf{PostBQP}=\mathbf{PP}\).

At the level of allowed gates, Aaronson's enlargement is closely related
to replacing the family

$$
\Delta_{n,q}
=
\left\{
U_{g,\mathbf i}
\;\middle|\;
U_g\in \mathrm U(V^{\otimes q}),\ 
\mathbf i=(i_1,\ldots,i_q)
\right\}
$$

by

$$
\Delta'_{n,q}
=
\left\{
L_{g,\mathbf i}
\;\middle|\;
L_g\in\End(V^{\otimes q}),\ 
\mathbf i=(i_1,\ldots,i_q)
\right\}.
$$

The latter is an algebraic family and is the Zariski closure of the
corresponding family of invertible local transformations.

There is, however, an important distinction between Aaronson's model and
the relative-rank framework used here.  Aaronson studies the languages
decidable by uniform polynomial-size circuits, with normalization and
measurement included in the computational model.  By contrast,

$$
\rank_{\End(V^{\otimes n}),\Delta'}(E)
$$

is the minimum length of an exact factorization of a prescribed
endomorphism \(E\) into elements of \(\Delta'\).  In particular, it does
not by itself incorporate normalization, postselection, or a measurement
rule.  Thus, Aaronson's result should be viewed as motivation for studying
complexity relative to nonunitary algebraic families of gates, rather
than as an identification of the relative-rank model with the complexity
class \(\mathbf{PP}\).

The main theorems in this section (Theorems~\ref{thm:quantum1} and \ref{thm:quantum2}) point in a complementary direction: they show that
Shannon-type lower bounds continue to hold relative to algebraic families
of allowed transformations of controlled dimension and degree, even when
these families contain nonunitary or nonlocal transformations.
\end{remark}

We now explain (following \cite{Lipton-Regan}) what is meant by a quantum circuit computing a Boolean function. 
For ease of understanding,  we start with a provisional (very stringent) definition 
and then give a more general (much less stringent)  definition afterwards. 

Let $f: \{0,1\}^n \rightarrow \{0,1\}$ be a Boolean function. We associate a 
unitary transformation $U_f \in \mathbf{U}(2^{n+1})$ to $f$ which takes for
$x_0,\ldots,x_{n-1},y \in \{0,1\}$, the separable state 
$|\mathbf{x} \ra \otimes |y\ra$ to the state 
$|\mathbf{x} \ra \otimes |f(x) \oplus y\ra$, where $\oplus$  denotes ``exclusive-or''.
Note that using the measurement postulate of quantum mechanics, if we 
set the input qubits to $|\mathbf{x}\ra \otimes |0\ra$ and measure
the ancillary qubit in the output, we will be left in the state $|1\ra$ if $f(x) = 1$ and 
in the state $|0\ra$ if $f(x) = 0$ \emph{with probability $1$}.

\begin{definition}
\label{def:quantum-complexity-Boolean}
We will call the quantum circuit complexity of $U_f$ the stringent quantum circuit complexity of $f$, and the 
 quantum complexity of $U_f$ relative to $\Delta$ as
 stringent quantum complexity of $f$ relative to $\Delta$.
\end{definition}

We prove the following theorem.

\begin{theorem}
    \label{thm:quantum1}
    There exists $C > 0$, such that for 
    all $n > 0$, $\Delta \subset \End(V^{\otimes (n+1)})$ algebraic subset  with $\dim \Delta = p$ and $\deg \Delta = D$,
there exists a Boolean function $f:\{0,1\}^n \rightarrow \{0,1\}$, such that the stringent quantum complexity of 
$f$ relative to $\Delta$, is at least 
\[
C \cdot \frac{2^n}{p\cdot n + \log D}.
\]
\end{theorem}

We can deduce as a corollary (setting
$p = 2^q \times 2^q)$ and $D = \binom{n+1}{q}$ in Theorem~\ref{thm:quantum1}) the following lower bound 
which is the quantum analog of Shannon's lower bound for classical circuits.

\begin{corollary}
\label{cor:quantum1}
For each $q \geq 1$, there exists $C=C_q > 0$ (depending only on $q$), such that for each $n >0$ 
there exists a Boolean function $f:\{0,1\}^n \rightarrow \{0,1\}$, such that the stringent quantum circuit complexity of
$f$ is at least $C \cdot \frac{2^n}{n}$. 
\end{corollary}
\begin{remark}
The same counting argument also shows
that the fraction of Boolean functions that need quantum circuits in $\mathcal{C}_q$ of size greater than $C \cdot \frac{2^n}{n}$ is $1- 2^{-2^{n-1}} = 1 - o(1)$.  
\end{remark}

\begin{remark}
Notice that the set $\mathbf{U}(V^{\otimes q})$ of $q$-ary quantum gates 
is infinite (uncountably so), and hence the number of 
distinct quantum circuits of any bounded size is also uncountably infinite. So a priori there is no reason for such circuits not being able to compute all Boolean functions on $n$ variables.
\end{remark}

As mentioned before we will work with a more general notion of what it means for a 
quantum circuit to compute a Boolean function. We relax our prior definition in two
ways. First, instead of ending up in the right state depending on the value of the function $f$ 
with probability $1$, we are satisfied if we end at the right state with probability $> 1/2$.
We will also allow more than one ancillary qubit. The precise  definition is as follows.

\begin{definition}[Unitary transformations computing a Boolean function]
\label{def:quantum-boolean}
Let
$$
\mathcal H_{n,t}
=
V^{\otimes n}\otimes V\otimes V^{\otimes t},
$$
where the first \(n\) qubits form the input register, the next qubit is the
output qubit, and the remaining \(t\) qubits are work qubits.

For \(E\in \End(\mathcal H_{n,t})\), \(x\in\{0,1\}^{n}\), and
\(b\in\{0,1\}\), define

$$
p_b(E;x)
=
\sum_{\substack{x'\in\{0,1\}^{n}\\ z\in\{0,1\}^{t}}}
\left|
\left\langle x'\,b\,z
\middle|
E
\middle|
x\,0\,0^{t}
\right\rangle
\right|^{2},
$$

and put

$$
Q_x(E)=p_1(E;x)-p_0(E;x).
$$

Each \(Q_x\) is a real polynomial of degree at most \(2\) on
\(\End(\mathcal H_{n,t})\), viewed as a real vector space.

Given a Boolean function

$$
f:\{0,1\}^{n}\longrightarrow\{0,1\},
$$

let

$$
\mathcal U^{+}_{f,t}
=
\left\{
U\in \mathrm U(\mathcal H_{n,t})
\;\middle|\;
(2f(x)-1)Q_x(U)>0
\text{ for every }x\in\{0,1\}^{n}
\right\}.
$$

Equivalently, \(U\in\mathcal U^{+}_{f,t}\) if, for every input \(x\),
measurement of the output qubit after applying \(U\) to
\(\lvert x\,0\,0^{t}\rangle\) produces \(f(x)\) with probability strictly
greater than \(1/2\).

We also introduce the auxiliary semi-algebraic set

$$
\mathcal E^{+}_{f,t}
=
\left\{
E\in \End(\mathcal H_{n,t})
\;\middle|\;
(2f(x)-1)Q_x(E)>0
\text{ for every }x\in\{0,1\}^{n}
\right\}.
$$

Then

$$
\mathcal U^{+}_{f,t}\subset \mathcal E^{+}_{f,t}.
$$

\end{definition}

\begin{remark}
For a unitary transformation \(U\),

$$
p_0(U;x)+p_1(U;x)=1,
$$

so the condition in Definition~\ref{def:quantum-boolean} is precisely
that the output qubit has the correct value with probability greater than
\(1/2\). No condition is imposed on the final state of the input and work
registers.

The definition is an unbounded-error definition: no uniform positive gap
from \(1/2\) is required. Moreover,

$$
U_f\otimes \Id_{V^{\otimes t}}\in\mathcal U^{+}_{f,t}
$$

for every \(t\geq 0\); in particular, \(U_f\in\mathcal U^{+}_{f,0}\).
\end{remark}

\begin{definition}
\label{def:quantum-complexity-Boolean-relaxed}
We will call the 
minimum of the 
quantum circuit complexities of $U \in U_{f,t}^+$
as the quantum circuit complexity of $f$ with $t+1$ ancillary qubits, 
and the 
minimum of 
quantum complexities of $U \in U_{f,t}^+$ relative to $\Delta$ as
the quantum complexity of $f$ with $t+1$ ancillary qubits relative to $\Delta$.
\end{definition}

In the following theorem we consider $A = \End(V^{\otimes (n+t+1)})$ as
a real algebra of dimension $2\cdot 4^{n+t+1}$.

\begin{theorem}
\label{thm:quantum2}
\label{cor:15}
Let $n>0, t \geq 0$, and  $A=\End(V^{\otimes(n+t+1)})$, viewed as a real algebra, and let $\Delta\subset A_\C$ be a real algebraic set whose real points are equidimensional on each irreducible component, with $\dim\Delta=p$ and $\deg\Delta=D$. Then there is an absolute constant $C>0$ and a Boolean function $f:\{0,1\}^n\to\{0,1\}$ whose quantum complexity with $t+1$ ancillary qubits relative to $\Delta$ is at least
\[
C\cdot\frac{2^n}{(p+1)(n+\log D)}.
\]
\end{theorem}

In the absence of the algebraic gate family \(\Delta\), a stronger Shannon-type lower bound follows by a direct adaptation of the Warren-bound argument of \cite[Theorem 7.1]{Roychowdhury-et-al}. We record the following slightly weaker consequence of Theorem~\ref{thm:quantum2} to illustrate how the classical quantum-circuit setting is recovered from our general relative-rank framework.

\begin{corollary}(cf. \cite[Theorem 7.1]{Roychowdhury-et-al})
\label{cor:quantum2}
For each $q \geq 1$, there exists $C = C_q > 0$ (depending only on $q$), such that for each $n >0$  and $t \geq 0$,
there exists a Boolean function $f:\{0,1\}^n \rightarrow \{0,1\}$, such that the quantum circuit complexity of $f$ with $t+1$ ancillary qubits is at least  $C \cdot \frac{2^n}{n+\log(n+t)}$
\end{corollary}

\begin{remark}
A more traditional approach to proving quantum circuit lower bounds is to replace arbitrary local unitary gates by circuits over a fixed finite universal gate set, such as Clifford+\(T\), using standard unitary approximation results (see, for example, \cite{Knill95,Maslov-et-al,Selinger}), and then apply a counting argument analogous to Shannon’s classical argument.

To make this approach quantitative, one must strengthen Definition~\ref{def:quantum-boolean} by imposing a prescribed gap. More precisely, for a fixed \(\delta\in(0,1/2)\), suppose that a circuit \(C\) is required to satisfy

\begin{equation}
    \label{eqn:def:quantum-boolean:2}
    p_{f(x)}(U(C);x)\ge \frac12+\delta
\qquad\text{for every }x\in\{0,1\}^n.
\end{equation}

Since \(U(C)\) is unitary, this is equivalently expressed as

$$
(2f(x)-1)Q_x(U(C))\ge 2\delta
\qquad\text{for every }x\in\{0,1\}^n.
$$

This condition is stable under sufficiently small perturbations in operator norm. Indeed, for any two unitary transformations \(U,U'\),

$$
\bigl|p_b(U;x)-p_b(U';x)\bigr|
   \le 2\lVert U-U'\rVert_{\mathrm{op}}.
$$

Consequently, if \(U'\) satisfies

$$
\lVert U-U'\rVert_{\mathrm{op}}<\frac{\delta}{4},
$$

then every circuit satisfying \eqref{eqn:def:quantum-boolean:2} with gap \(\delta\) is replaced by one satisfying the same condition with gap at least \(\delta/2\). If \(U(C)=U_r\cdots U_1\), it is enough to approximate each gate \(U_j\) within \(O(\delta/r)\) in operator norm, since

$$
\bigl\|U_r\cdots U_1-U'_r\cdots U'_1\bigr\|_{\mathrm{op}}
   \le \sum_{j=1}^r\lVert U_j-U'_j\rVert_{\mathrm{op}}.
$$

One may therefore approximate a circuit satisfying  \eqref{eqn:def:quantum-boolean:2} by a circuit over Clifford+\(T\) and then count circuits over this finite gate set. Using the standard synthesis bounds, this gives that almost all Boolean functions on \(n\) variables require quantum circuits of size

$$
\Omega\!\left(
\min\left\{
\frac{2^n}{n\log n},
\frac{2^n}{\log(1/\delta)\log n}
\right\}
\right)
$$

under the fixed-gap condition \eqref{eqn:def:quantum-boolean:2}.

This argument does not yield a uniform lower bound as \(\delta\to0\). Definition~\ref{def:quantum-boolean} requires only that the success probability be strictly greater than \(1/2\), with no gap uniform over functions or circuit sizes. Thus a finite-gate approximation may require arbitrarily high precision. This is analogous to the distinction between bounded-error probabilistic computation, as in \(\mathbf{BPP}\), and unbounded-error computation, as in \(\mathbf{PP}\). The relative-rank argument used in Theorem~\ref{thm:quantum2} avoids this approximation issue and applies directly to the strict inequalities in Definition~\ref{def:quantum-boolean}. On the other hand, it is presently nonconstructive and does not produce explicit Boolean functions having the asserted circuit complexity; see \cite{Yifan-et-al} for recent progress concerning explicit functions.
\end{remark}

\subsection{Proofs of Theorems~\ref{thm:rank:vector:algebraic} and \ref{thm:rank:vector:ordered}}
We use the following convention concerning pullbacks of finite families of polynomials. If \(\mathcal P\) is a finite family and \(\psi\) is a polynomial map, then

$$
\psi^*(\mathcal P)=(P\circ\psi)_{P\in\mathcal P}
$$

denotes the corresponding indexed family; thus, equal pullback polynomials arising from different elements of \(\mathcal P\) retain their original labels. A zero-nonzero pattern or sign condition on \(\psi^*(\mathcal P)\) is understood as an assignment to these labels. If two equal polynomials are assigned incompatible values, the corresponding realization is empty.

The bounds used below continue to hold for such indexed families. Indeed, after equal polynomials are identified, the compatible patterns are in bijection with the patterns on the resulting set of distinct polynomials, while all incompatible patterns have empty realization. Moreover, the number of distinct polynomials is at most \(\operatorname{card}(\mathcal P)\).

\begin{lemma}
\label{lem:fiber}
Let $X \subset k^M$ be an algebraic subset, $\psi\colon k^M \to k^N$ a
polynomial map, and $\mathcal{P} \subset k[X_1,\dots,X_N]$ a finite family with
$\card(\mathcal{P}) = s$.

Let $F \subset k^N$ be a finite
$\mathcal{P}$-constructible set. Then
\[
\card\bigl(\{\,w \in F \mid \psi^{-1}(w) \cap X \neq \emptyset\,\}\bigr)
\;\le\;
\sum_{\widetilde\pi \in \{0,1\}^{\psi^*(\mathcal{P})}}
\deg\Bigl(\overline{\RR(\widetilde\pi, X)}\Bigr).
\]
\end{lemma}

\begin{proof}
Membership in $F$ is determined by the zero--nonzero pattern of $\mathcal{P}$, since
$F$ is defined by a Boolean formula with atoms $P = 0$, $P \in \mathcal{P}$. Hence
\[
F \;=\; \bigsqcup_{\pi \in \Pi} \RR(\pi, k^N),
\qquad
\Pi = \{\pi \mid \emptyset \neq \RR(\pi,k^N) \subseteq F\},
\]
and each $\RR(\pi, k^N)$, $\pi \in \Pi$, is finite. 

For \(\pi\in\Pi\), define the zero-nonzero pattern
$$
\widetilde\pi\in\{0,1\}^{\psi^*(\mathcal P)}
$$
by assigning to the copy of \(P\circ\psi\) indexed by \(P\in\mathcal P\) the value
$$
\widetilde\pi_P=\pi(P).
$$

Since the indices are retained, the assignment

$$
\pi\longmapsto\widetilde\pi
$$
is injective.

Then
\[
\RR(\widetilde\pi, X)
\;=\;
\psi^{-1}\bigl(\RR(\pi, k^N)\bigr) \cap X
\;=\;
\bigsqcup_{w \in \RR(\pi,k^N)} \bigl(\psi^{-1}(w) \cap X\bigr),
\]
is a finite disjoint union of Zariski closed subsets of $X$. In particular
$\RR(\widetilde\pi, X)$ is Zariski closed, equal to its own closure, and each
nonempty term $\psi^{-1}(w) \cap X$ is open and closed in
$\RR(\widetilde\pi,X)$, hence a union of irreducible components of
$\RR(\widetilde\pi,X)$. Since every irreducible component has degree at least
one,
\[
\deg \RR(\widetilde\pi, X) \;\ge\;
\card\bigl(\{\,w \in \RR(\pi,k^N) \mid \psi^{-1}(w) \cap X \neq
\emptyset\,\}\bigr).
\]
Summing over $\pi \in \Pi$ proves the lemma. 
\end{proof}

\begin{proof}[Proof of Theorem~\ref{thm:rank:vector:algebraic}]
If $\rank_{V,\Delta}(W) = \infty $, then there is nothing to prove. Otherwise, let
$\rank_{V,\Delta}(W) = t_0$.
Every $w \in W$ satisfies $1 \le \mathrm{rank}_{V,\Delta}(w) \le t_0$, so
$W = \bigsqcup_{t=1}^{t_0} W_t$ with
$W_t = \{w \in W \;\mid\; \mathrm{rank}_{V,\Delta}(w) = t\}$.

Fix $t$ and let $\varphi_t\colon V^t \to V$,
$\varphi_t(v_1,\dots,v_t) = v_1 + \cdots + v_t$, a linear map. If $w \in
W_t$ then $w = v_1 + \cdots + v_t$ with $v_1,\dots,v_t \in \Delta$, i.e.\
$\varphi_t^{-1}(w) \cap \Delta^t \neq \emptyset$. Apply
Lemma~\ref{lem:fiber} with $X = \Delta^t$, $\psi = \varphi_t$, $F = W$, and
then Theorem~\ref{thm:main:algebraic} to the algebraic set $\Delta^t$
which has dimension $\le tp$ and degree $\le D^t$ --- and the
family $\varphi_t^*(\mathcal{P})$ of $s$ polynomials of degree $\le \delta$ (as
$\varphi_t$ is linear):
\[
\card(W_t)
\;\le\;
\sum_{\widetilde\pi}
\deg\Bigl(\overline{\RR(\widetilde\pi, \Delta^t)}\Bigr)
\;\le\;
D^t \sum_{i=0}^{tp} \binom{s}{i}\,\delta^i.
\]
Summing over $t$ and using $\binom{s}{i}\delta^i \le (s\delta)^{tp}$ for
$i \le tp$, together with $D^t \le D^{t_0}$ and
$(s\delta)^{tp} \le (s\delta)^{t_0 p}$, gives 
\[
\card(W)
\;\le\;
\sum_{t=1}^{t_0} D^t \sum_{i=0}^{tp} \binom{s}{i}\,\delta^i
\;\le\;
t_0\,(t_0 p + 1)\, D^{t_0} (s\delta)^{t_0 p}.
\]
The theorem follows after taking logarithm of both sides.
\end{proof}

In order to prove Theorem~\ref{thm:rank:vector:ordered} we need the following proposition which is a generalization of Theorem~\ref{thm:main:ordered:a}. Since the proof of this generalization repeats the salient points of the proof of Theorem~\ref{thm:main:ordered:a} we relegate it to an appendix (Section~\ref{sec:appendix}).

\begin{proposition}
\label{prop:A}
Let $\Delta$ be as in Theorem~\ref{thm:rank:vector:ordered}, $k\ge 1$, and $\mathcal{Q} \subset\R[X_1,\dots,X_{Nk}]_{\le d}$ with $\card(\mathcal{Q}) =s$, $d \geq 2$. Then
\[
\sum_{\sigma\in\{0,\pm1\}^{\mathcal{Q}}} b_0\bigl(\RR(\sigma,\Delta(\R)^k)\bigr)\ \le\ (O(s\,k\,d\,D))^{3k(p+1)}.
\]
\end{proposition}

\begin{proof}
    See Section~\ref{sec:appendix}.
\end{proof}
\begin{proof}[Proof of Theorem~\ref{thm:rank:vector:ordered}]
Let $t_0=\mmrank_{V,\Delta}(W)$; if $t_0=\infty$ there is nothing to prove, so assume $1\le t_0<\infty$. For $k\ge1$ let $\phi_k:V^k\to V$, $\phi_k(v_1,\dots,v_k)=v_1+\dots+v_k$,
and denote by $\phi_k^{*}: \R[V] \rightarrow \R[V^k]$ the corresponding pull-back
homomorphism.

\medskip
Let \[
\mathcal{Q}_k=\{\phi_k^{*}(P) \; \mid \; P\in\mathcal{P}\}.
\]
Note that $\card(\mathcal{Q}_k) = s$ and 
$\mathcal{Q}_k \subset\R[V^k]_{\le\delta}$. 

For $\sigma\in\{0,1,-1\}^\mathcal{P}$, we denote by 
$\phi_k^*(\sigma) \in \{0,1,-1\}^{\mathcal{Q}_k}$ defined by 
\[
\phi_k^*(\sigma)(\phi_k^*(P)) = \sigma(P), P \in \mathcal{P}.
\]

\smallskip
We claim
\begin{equation}\label{eq:red}
b_0(W)\ \le\ \sum_{k\le t_0}\ \sum_{\sigma\in\{0,1,-1\}^\mathcal{P}} b_0
\bigl(\RR(\phi_k^*(\sigma),\Delta(\R)^k)\bigr).
\end{equation}

For each $C\in\Cc(W)$, choose $\alpha_C\in C$ with $\mathrm{rank}_{V,\Delta}(\alpha_C) \leq  t_0$. Hence there exists $a_C \in \Delta(\R)^{k}$, such that $\alpha_C=\phi_k(a_C)$ with $k\leq t_0$.  
Let $\sigma_C \in \{0,1,-1\}^{\mathcal{P}}$
such that $\sign(P(\alpha_C)) = \sigma_C(P), P \in \mathcal{P}$,
and let $E_C$ the semi-algebraically connected component of 
$\RR(\phi_k^*(\sigma_C),\Delta(\R)^k)$ containing $a_C$. We say that
the triple $(k,\sigma_C,E_C)$ is associated to $C$.

Suppose that $(k,\sigma,E)$ is associated to both $C,C' \in \Cc(W)$.
Then $\phi_k(E)$ is a semi-algebraically connected subset of $\RR(\sigma,V)$; as $W$ is a $\mathcal{P}$-semi-algebraic set and meets $\RR(\sigma,V)$, $\RR(\sigma,V)\subset W$; so $\phi_k(E)$ is a semi-algebraically connected subset of $W$ meeting $C$ and $C'$, whence $C=C'$. Thus $C\mapsto(k,\sigma_C,E_C)$ is injective, proving \eqref{eq:red}.
 
\smallskip
By Proposition~\ref{prop:A} 
applied with $d=\delta$,
\[
\sum_{\sigma}b_0\bigl(\RR(\phi_k^*(\sigma), \Delta(\R)^k)\bigr)\ \le\ (O(s\,k\,\delta\,D))^{3k(p+1)}.
\]
 
\smallskip 
Combining, we get
\[
b_0(W)\ \le\ \sum_{k=1}^{t_0}(O(\,s\,k\,\delta D))^{3k(p+1)}\ \leq\ t_0 (O(s\,t_0\,\delta D))^{3t_0(p+1)}\ \le\ (O(s\,t_0\,\delta D))^{4t_0(p+1)},
\]
from which inequality \eqref{eqn:thm:rank:vector:ordered:1} follows.

Inequality \eqref{eqn:thm:rank:vector:ordered:2} follows from
inequality \eqref{eqn:thm:rank:vector:ordered:1}, 
since for any finite %%$\mathcal{P}$-$
semi-algebraic set $W$,
$\mmrank_{V,\Delta}(W) =\rank_{V,\Delta}(W)$ and $b_0(W)=\card(W)$.
\end{proof}

\subsection{Proofs of Theorems~\ref{thm:rank:algebra:algebraic} and 
\ref{thm:rank:algebra:ordered}}

\begin{proof}[Proof of Theorem~\ref{thm:rank:algebra:algebraic}]
If $\rank_{A,\Delta}(W) = \infty $, then there is nothing to prove. Otherwise, let
$\rank_{A,\Delta}(W) = t_0$.

Then $W = \bigsqcup_{t=1}^{t_0} W_t$ with
$W_t = \{\alpha \in W \mid \mathrm{rank}_{A,\Delta}(\alpha) = t\}$. Fix $t$. If
$\alpha \in W_t$, there exist a bracketing $\beta$ and
$a \in \Delta^t$ with $\psi_{t,\beta}(a) = \alpha$,
where $\psi_{t,\beta}: A^t \rightarrow A$ is the map that sends
$(\alpha_1,\ldots,\alpha_t)$ to the product of $\alpha_1,\ldots,\alpha_t$ under the bracketing $\beta$
; hence
\[
W_t \;\subseteq\; \bigcup_{\beta}
\{\,\alpha \in W \mid \psi_{t,\beta}^{-1}(\alpha) \cap \Delta^t \neq
\emptyset\,\}.
\]
For each fixed $\beta$, Lemma~\ref{lem:fiber} (with $X = \Delta^t$,
$\psi = \psi_{t,\beta}$, $F = W$) followed by Theorem~\ref{thm:main:algebraic} applied to
$\Delta^t$ (dimension $\le tp$, degree $\le D^t$) and the family $\psi_{t,\beta}^*(\mathcal{P})$ of $s$
polynomials of degree $\le t\delta$  bounds the
corresponding set by $D^t \sum_{i=0}^{tp}\binom{s}{i}(t\delta)^i$. Summing
over the $\leq 4^{t-1} \leq 4^{t_0}$ bracketings and then over $t$, gives
\[
\card(W)
\;\le\;
\sum_{t=1}^{t_0} 4^{t_0}\, D^t
\sum_{i=0}^{tp} \binom{s}{i}\,(t\delta)^i
\;\le\;
t_0\,(t_0 p + 1)\, 4^{t_0} D^{t_0} (s\, t_0\, \delta)^{t_0 p}.
\]
The theorem follows after taking logarithm of both sides.
\end{proof}

\begin{proof}[Proof of Theorem~\ref{thm:rank:algebra:ordered}]
Let $t_0=\mmrank_{A,\Delta}(W)$. If $t_0=\infty$ there is nothing to prove, so assume $1\le t_0<\infty$.

For $k\geq 1$ and a bracketing $\beta$ of $k$ factors (a full binary tree with $k$ leaves; there are $\mathrm{Cat}(k-1)\le 4^k$ of them, and a single one suffices if $A$ is associative), let $\phi_{k,\beta}:A^k\to A$ be the iterated product. It is multilinear, hence a polynomial map of degree $k$. 
We denote by $\phi_{k,\beta}^*: \R[A] \rightarrow \R[A^k]$ the corresponding pull-back homomorphism. 

We denote by

$$
\mathcal Q_{k,\beta}
=
\phi_{k,\beta}^*(\mathcal P)
=
\bigl(\phi_{k,\beta}^*(P)\bigr)_{P\in\mathcal P}
$$

the corresponding indexed family. Thus,

$$
\operatorname{card}(\mathcal Q_{k,\beta})=s,
\qquad
\mathcal Q_{k,\beta}\subset R[A^k]_{\le k\delta},
$$

where the cardinality is the indexed cardinality. For
\(\sigma\in\{0,1,-1\}^{\mathcal P}\), define

$$
\phi_{k,\beta}^*(\sigma)
   \in\{0,1,-1\}^{\mathcal Q_{k,\beta}}
$$

by assigning to the copy of \(\phi_{k,\beta}^*(P)\) indexed by
\(P\in\mathcal P\) the value

$$
\bigl(\phi_{k,\beta}^*(\sigma)\bigr)_P=\sigma(P).
$$

We claim:
\[
b_0(W)\ \le\ \sum_{k\le t_0}\ \sum_{\beta}\ \sum_{\sigma\in\{0,1, -1\}^{\mathcal{P}}} b_0
\bigl(\RR(\phi_{k,\beta}^*(\sigma),\Delta(\R)^k)\bigr).
\]

To prove the claim,
for each $C\in\Cc(W)$ choose $\alpha_C\in C$ with $\rank_{A,\Delta}(\alpha_C)\leq t_0$. Then, $\alpha_C=\phi_{k,\beta}(a_C)$ for some $k\leq t_0$, some bracketing $\beta$ and some $a_C\in\Delta(\R)^k$. 

Let $\sigma_C\in\{0,1,-1\}^{\mathcal{P}}$ be 
defined by $\sigma_C(P) = \sign(P(\alpha_C)), P \in \mathcal{P}$,
and let $E_C$ be the semi-algebraically connected component of 
$\RR(\phi_{k,\beta}^*(\sigma_C), \Delta(\R)^k)$ 
containing $a_C$. 

We say that the tuple $(k,\beta,\sigma_C,E_C)$ is associated to $C$.
We now prove that if a tuple $(k,\beta,\sigma,E)$ is associated to both $C,C' \in \Cc(W)$, then $C = C'$, which suffices to prove the claim. Suppose 
$(k,\beta,\sigma,E)$ is associated to both $C,C' \in \Cc(W)$.
 Then $\phi_{k,\beta}(E)$ is a semi-algebraically connected subset of $\RR(\sigma,A)$. Since $W$ is a $\mathcal{P}$-semi-algebraic set and $\alpha_C\in W\cap\RR(\sigma,A)$, we have $\RR(\sigma,A)\subset W$. Hence $\phi_{k,\beta}(E)$ is a semi-algebraically connected subset of $W$ containing $\alpha_C$ and $\alpha_{C'}$, and so $C=C'$.

Using Proposition~\ref{prop:A} we have
\[
b_0(W)\ \le\ \sum_{k=1}^{t_0}\left(O(s\,k^2\delta D)\right)^{3k(p+1)}
\ \leq\ t_0\,\left(O(s\,t_0^2\delta D)\right)^{3t_0(p+1)}
\ \le\ \left(O(s\,t_0^2\delta D)\right)^{4t_0(p+1)}.
\]
The theorem now follows from the last inequality. 
\end{proof}

\subsection{Proofs of Theorems~\ref{thm:quantum1}
and \ref{thm:quantum2}}

\begin{proof}[Proof of Theorem~\ref{thm:quantum1}]
    Recall that for any Boolean function $f \in \{0,1\}^n \rightarrow \{0,1\}$, the stringent quantum circuit complexity of $f$ relative to $\Delta$ is by definition 
    \[
    \rank_{\End(V^{\otimes (n+1)}),\Delta}(U_f),
    \]
    where $V$ is a two-dimensional complex inner product space.
    Let $W = \{U_f \;\mid\; f: \{0,1\}^n \rightarrow \{0,1\}\} \subset \End(V^{\otimes (n+1)})$.
    It is easy to see that $\card(W) = 2^{2^n}$, and $W$ is $\mathcal{P}$-constructible where $\mathcal{P}$ is a set of $2^{2(n+1)}$ polynomials
    of degrees bounded by 
    $2$. Now apply Theorem~\ref{thm:rank:algebra:algebraic}.
\end{proof}

\begin{proof}[Proof of Corollary~\ref{cor:quantum1}]
For $\mathbf{i} = (i_1,\ldots,i_q), 1 \leq i_1 <\cdots <i_q \leq n+1$,
let 
$A_{\mathbf{i}} \subset  \End(V^{\otimes (n+1)})$ 
be the sub-algebra obtained by tensoring $\End(V)$ in the $i_j$-th slots
$1 \leq i_1 < \cdots < i_{q} \leq n+1$ with $1_V$ in the remaining. 
Thus, 
\[
A_{(1,2,\ldots,q)} = \underbrace{\End(V) \otimes \cdots \otimes \End(V)}_{q} \otimes \overbrace{\mathrm{Id}_V \otimes \cdots \otimes \mathrm{Id}_V}^{n+1-q}.
\]

Let 
\[
\Delta = \bigcup_{\mathbf{i} = (i_1,\ldots,i_q), 1 \leq i_1 <\cdots <i_q \leq n+1} A_{\mathbf{i}}.
\]

Then, $\Delta$ is a union of $\binom{n+1}{q}$ subvarieties of 
dimension $p = 2^q \times 2^q$, and $D = \deg(\Delta) = \binom{n+1}{q}$.

Moreover, the subset of $\End(V^{\otimes (n+1)})$ consisting of unitary transformations of the form $U_{g,\mathbf{i}}$, where 
$g$ is a $q$-ary quantum gate and $\mathbf{i} = (i_1,\ldots,i_q)$ is a $q$-tuple of distinct indices with
$1 \leq i_j \leq n+1$ is contained in $\Delta$ defined above,
and using Remark~\ref{rem:rank:algebra:monotone}, it suffices to bound from below 
\[
\rank_{\End (V^{\otimes(n+1)}),\Delta}(W),
\]
where $W = \{U_f \;\mid\; f: \{0,1\}^n \rightarrow \{0,1\}\} \subset \End(V^{\otimes (n+1)})$.
The corollary now follows immediately from Theorem~\ref{thm:quantum1}.
\end{proof}

\begin{proof}[Proof of Theorem~\ref{thm:quantum2}]
Put

$$
A=\End(\mathcal H_{n,t}),
$$

viewed as a real algebra. For each \(x\in\{0,1\}^{n}\), decompose

$$
E\lvert x\,0\,0^{t}\rangle
=
u_{x,0}(E)\oplus u_{x,1}(E)
$$

according to the value \(0\) or \(1\) of the output qubit. Thus

$$
u_{x,b}(E)\in
V^{\otimes n}\otimes \lvert b\rangle\otimes V^{\otimes t},
\qquad b\in\{0,1\},
$$

and

$$
Q_x(E)
=
\left\|u_{x,1}(E)\right\|^{2}
-
\left\|u_{x,0}(E)\right\|^{2}.
$$

We first observe that each \(\mathcal E^{+}_{f,t}\) is semi-algebraically
connected. Indeed, for a fixed \(x\), the condition \(Q_x(E)>0\) has the
form

$$
\|u_1\|^{2}>\|u_0\|^{2}.
$$

The homotopy

$$
(u_0,u_1)\longmapsto ((1-s)u_0,u_1),
\qquad 0\leq s\leq1,
$$

preserves this inequality and retracts the set onto

$$
\{0\}\times
\left(
V^{\otimes n}\otimes\lvert1\rangle\otimes V^{\otimes t}
\setminus\{0\}
\right),
$$

which is semi-algebraically connected. The same argument, with the two
factors interchanged, applies to \(Q_x(E)<0\).

The conditions corresponding to distinct \(x\)'s involve distinct columns
of the matrix of \(E\), while all other columns are unrestricted.
Consequently, \(\mathcal E^{+}_{f,t}\) is a product of semi-algebraically
connected sets and is therefore semi-algebraically connected.

Now let

$$
W^{+}
=
\bigcup_{f:\{0,1\}^{n}\to\{0,1\}}
\mathcal E^{+}_{f,t}.
$$

Equivalently,

$$
W^{+}
=
\left\{
E\in A
\;\middle|\;
Q_x(E)\neq0
\text{ for every }x\in\{0,1\}^{n}
\right\}.
$$

If \(f\neq g\), then for some \(x\) the sets
\(\mathcal E^{+}_{f,t}\) and \(\mathcal E^{+}_{g,t}\) impose opposite
strict signs on \(Q_x\), and hence are disjoint. Since each
\(\mathcal E^{+}_{f,t}\) is nonempty and semi-algebraically connected,
these sets are precisely the semi-algebraically connected components of
\(W^{+}\). Therefore

$$
b_0(W^{+})=2^{2^{n}}.
$$

Let

$$
\mathcal P=\{Q_x\mid x\in\{0,1\}^{n}\}.
$$

Then \(W^{+}\) is a \(\mathcal P\)-semi-algebraic set,

$$
\card(\mathcal P)=2^{n},
\qquad
\deg Q_x\leq2.
$$

Applying Theorem~13 gives

$$
\mmrank_{A,\Delta}(W^{+})
\geq
C\,
\frac{2^{n}}
     {(p+1)(n+\log D)}
$$
for an absolute constant \(C>0\).
Since the semi-algebraically connected components of \(W^{+}\) are exactly
the sets \(\mathcal E^{+}_{f,t}\), there exists a Boolean function \(f\)
such that

$$
\min_{E\in\mathcal E^{+}_{f,t}}
\rank_{A,\Delta}(E)
\geq
C\,
\frac{2^{n}}
     {(p+1)(n+\log D)}.
$$

Finally,

$$
\mathcal U^{+}_{f,t}\subset\mathcal E^{+}_{f,t},
$$

and hence

$$
\min_{U\in\mathcal U^{+}_{f,t}}
\rank_{A,\Delta}(U)
\geq
\min_{E\in\mathcal E^{+}_{f,t}}
\rank_{A,\Delta}(E).
$$

The left-hand side is the quantum complexity of \(f\) with \(t+1\)
ancillary qubits relative to \(\Delta\), which proves the theorem.
\end{proof}

\begin{proof}[Proof of Corollary~\ref{cor:quantum2}]
    The subset $\Delta \subset A$ defined in the proof of
    Corollary~\ref{cor:quantum1} is a union of real algebraic subsets of 
    $A$
    (in fact affine subspaces) with $\dim \Delta =   2(2^q \times 2^q)$,
    $\deg \Delta = \binom{n+t+1}{q}$.

    Now apply Theorem~\ref{thm:quantum2}.
\end{proof}

\section{Open problems and future work}
\label{sec:conclusion}
We end with some open problems and suggestions for future work.
\begin{enumerate}[1.]
    \item Close the gap between the upper and lower bounds in Theorem~\ref{thm:main:ordered:a}.
    
    \item Is it possible to eliminate the restriction of local constancy of the real dimension in the hypothesis of Theorem~\ref{thm:main:ordered:a} ? 
    \item Is it possible to extend Theorem~\ref{thm:main:algebraic} to a bound on the sum of the Betti numbers of the realizations of all realizable zero-nonzero patterns ? 
    For example, the following result is proved in \cite{BPR-tight}.

\begin{thmx}
\label{thmx:BPRtight}
Let ${\mathcal P} \subset \C[X_1,\ldots,X_N]$ be a family of polynomials
with $\#{\mathcal P} = s$ and $\deg(P) \leq d$ for all $P \in {\mathcal P}$.
Then
\begin{eqnarray*}
\sum_{\pi \in \{0,1\}^{\mathcal{P}}}
b^{BM}(\RR(\pi,\C^N))  & \leq &  
\sum_{0 \leq \ell \leq N}
\binom{s}{N-\ell}\binom{s}{\ell}d^N 
 + \\
 &&\sum_{1 \leq i \leq \ell} \binom{s}{N - \ell}\binom{s}{\ell - i}d^{O(N^2)}
,
\end{eqnarray*}
where $b^{BM}(\cdot)$ denotes the sum of the Borel-Moore Betti numbers.
\end{thmx}
Is it possible to prove a similar bound for all realizable zero-nonzero patterns restricted to an algebraic set but where the bound is independent of the ambient dimension $N$ and depends only on the degree and the dimension of the algebraic set ?
\end{enumerate}

\section{Appendix}
\label{sec:appendix}
The proof of Proposition~\ref{prop:A} follows that of Theorem~\ref{thm:main:ordered:a} of the paper, but the projection step (Lemmas~\ref{lem:main:ordered:2} and \ref{lem:main:ordered:3} there) is replaced by one adapted to products, so that the degree $D^k$ of $\Delta^k$ enters only through $k$ separate hypersurface equations of degree $D$ rather than a single hypersurface of degree $D^k$.

\begin{lemma}[Birational projection with rational inverse]
\label{lem:K}
Let $\Delta\subset\C^N$ be an irreducible real variety with $\dim\Delta=p$ and $\deg\Delta=D$. There exist a real linear map $\pi:\C^N\to\C^{p+1}$ and polynomials $F,w,v_1,\dots,v_N\in\R[Y_1,\dots,Y_{p+1}]$ with $\deg F=D$, $w=\partial F/\partial Y_{p+1}$, $\deg v_i\le D$, such that
\begin{enumerate}[(a)]
\item
\label{itemlabel:lem:K:a}
$\pi|_\Delta$ is finite and $H:=\pi(\Delta)=Z(F)$ is an irreducible hypersurface;
\item 
\label{itemlabel:lem:K:b}
$w$ does not vanish identically on $H$;
\item 
\label{itemlabel:lem:K:c}
$w(\pi(x))\,x_i=v_i(\pi(x))$ for all $x\in\Delta$ and all $1\le i\le N$.
\end{enumerate}
Consequently, for $y\in H$ with $w(y)\ne0$ the fiber $\pi^{-1}(y)\cap\Delta$ is the single point $\rho(y)=\bigl(v_1(y)/w(y),\dots,v_N(y)/w(y)\bigr)$, which is in $\R^{N}$  if $y \in \R^{p+1}$. 
\end{lemma}

\begin{proof}
Since $\R^N$ is Zariski dense in $\C^N$, all generic conditions below can be satisfied by real linear forms.

Choose real linear forms $Y_1,\dots,Y_p$ giving a Noether normalization $\C[Y']\hookrightarrow\C[\Delta]$, $Y'=(Y_1,\dots,Y_p)$, finite with generic fiber of cardinality $D$, and a real linear form $u=Y_{p+1}$ separating the $D$ points of a generic fiber. Then $L=\C(\Delta)$ has degree $D$ over $K=\C(Y')$ and $u$ is a primitive element. Its minimal polynomial $F\in K[T]$ is monic of degree $D$ with coefficients in $\C[Y']$ (they are symmetric functions of the conjugates of the integral element $u$, and $\C[Y']$ is integrally closed). By Gauss's lemma $F$ is irreducible in $\C[Y',T]$. The map $\pi=(Y',u)$ is finite on $\Delta$, so $H=\pi(\Delta)$ is closed and irreducible of dimension $p$, and $F$ vanishes on $H$; hence $F$ is (up to a scalar) the reduced equation of $H$. 
This proves Part~\eqref{itemlabel:lem:K:a}.
Since $\pi^{-1}(\Lambda)\cap\Delta$ is finite for 
a generic
linear subspace $\Lambda\subset\C^{p+1}$ of codimension $p$, B\'ezout's inequality gives $\deg H\le D$. 
Since \(F\) is monic of degree \(D\) in \(T\), \(\deg F\ge D\); together with \(\deg H\le D\), this gives \(\deg F=D\).

As $H$ is stable under complex conjugation and $F$ is monic in $T$, $F$ has real coefficients. Finally $w=\partial_TF\ne0$ has $T$-degree $D-1<D$, so $F\nmid w$ 
which proves Part~\eqref{itemlabel:lem:K:b}.

In order to prove Part~\eqref{itemlabel:lem:K:c}, fix $i$ and for $\lambda\in\C$ put $u_\lambda=u+\lambda X_i$ and
\[
F_\lambda(Y',T)=\det\bigl(T\cdot I-M_u-\lambda M_{X_i}\bigr)\in K[\lambda,T],
\]
the characteristic polynomial of multiplication by $u_\lambda$ on $L\cong K^D$. For each fixed $\lambda_0\in\C$ this is the characteristic polynomial of an integral element, hence lies in $\C[Y'][T]$; by Vandermonde interpolation in $\lambda$, $F_\lambda\in\C[Y'][\lambda,T]$. For all but finitely many $\lambda_0$, $u_{\lambda_0}$ still separates the points of a fixed generic fiber, hence is primitive, and the argument of the previous paragraph shows $\deg_{(Y',T)}F_{\lambda_0}\le D$. A polynomial in $\lambda$ whose values have $(Y',T)$-degree $\le D$ for infinitely many $\lambda$ has all its coefficients of degree $\le D$; hence $F_\lambda\in\C[Y',T]_{\le D}[\lambda]$. Set
\[
v_i=-\left.\frac{\partial F_\lambda}{\partial\lambda}\right|_{\lambda=0}\in\C[Y',T]_{\le D}.
\]
Over a generic $y'\in\C^p$ the fiber $\Delta_{y'}$ is reduced with $D$ points and $F_\lambda(y',T)=\prod_{x\in\Delta_{y'}}(T-u(x)-\lambda x_i)$, so
\[
v_i(y',T)=\sum_{x\in\Delta_{y'}}x_i\prod_{x'\ne x}(T-u(x')),
\]
and at $T=u(x_0)$, $x_0\in\Delta_{y'}$, this gives $v_i(\pi(x_0))=x_{0,i}\,\partial_TF(\pi(x_0))=x_{0,i}\,w(\pi(x_0))$. Thus the polynomial $w(\pi(X))X_i-v_i(\pi(X))$ vanishes on the preimage in $\Delta$ of a nonempty Zariski open subset of $\C^p$; this preimage is Zariski dense in the irreducible $\Delta$, so the polynomial vanishes on $\Delta$. For real $\lambda$ the fiber over a real generic $y'$ is conjugation-stable and $u_\lambda$ is real, so $F_\lambda$ has real coefficients; hence so do its $\lambda$-coefficients, and $v_i$ is real.
This proves Part~\eqref{itemlabel:lem:K:c}.

Uniqueness of the fiber when $w(y)\ne0$ follows from Part~\eqref{itemlabel:lem:K:c}, and existence from the fact that $\pi(\Delta)=H$. 
\end{proof}

\begin{lemma}[Descent to a product of hypersurfaces]
\label{lem:C}
Let $\Delta_1,\dots,\Delta_k\subset\C^N$ be irreducible real varieties with $\dim\Delta_j=p_j\le p$, $\deg\Delta_j\le D$, and $\dim_x\Delta_j(\R)=p_j$ for all $x\in\Delta_j(\R)$. Let $V=\Delta_1\times\cdots\times\Delta_k\subset\C^{Nk}$ and $p_V=\sum_jp_j$. Let $\R'\supset\R$ be real closed, $G\in\R'[X_1,\dots,X_{Nk}]$ of degree $\le e$ with $e\ge1$, $\zeta$ a positive infinitesimal over $\R'$, and
\[
X=\ZZ\bigl(G-\zeta, V(\Rz)\bigr).
\]
Then 
\[
b_0(X) \leq (4keD)^{k(p+1)+1}.
\]
\end{lemma}

\begin{proof}
Apply Lemma~\ref{lem:K} to each $\Delta_j$, obtaining $\pi_j$, $F_j$, $w_j$, $v^{(j)}_i$ in variables $Y^{(j)}=(Y^{(j)}_1,\dots,Y^{(j)}_{p_j+1})$. Put $\pi=\pi_1\times\cdots\times\pi_k:\C^{Nk}\to\C^M$ with $M=\sum_j(p_j+1)\le k(p+1)$, and $\bW=\prod_j w_j(Y^{(j)})$, so $\deg\bW\le k(D-1)$. Then $\pi(V)=\prod_jH_j=Z(F_1,\dots,F_k)$, and over each $y\in\prod_jH_j$ with $\bW(y)\ne0$ the fiber $\pi^{-1}(y)\cap V$ is the single point $\rho(y)$ with coordinates $x^{(j)}_i=v^{(j)}_i(y^{(j)})/w_j(y^{(j)})$, real if $y$ is real. Let
\[
\widehat G=\bW^{e}\cdot(G\circ\rho)\in\R'[Y];
\]
since each monomial of $G$ has at most $e$ factors, $\widehat G$ is a polynomial, of degree $\leq ke(D-1)+e\le keD$. Define
\[
Y=\bigl\{\,y\in\Rz^{M}\;\mid\; F_j(y^{(j)})=0, 1\leq j\leq k, \bW(y)\neq0, \widehat G(y)-\zeta\bW(y)^{e}=0\bigr\}.
\]

By Hardt's triviality theorem applied to $G:V(\R')\to\R'$ there is a finite set $E\subset\R'$ such that $G$ is semi-algebraically trivial over each interval of $\R'\setminus E$. Since $V(\R')=\prod_j\Delta_j(\R')$ has local dimension $p_V$ at every point (the local dimension of a product is the sum of the local dimensions), triviality gives $\dim_x\bigl(V(\R')\cap G^{-1}(c)\bigr)=p_V-1$ for every $c\notin E$ and every $x$ in that fiber. This is a first-order statement in $c$ over $\R'$, so it transfers to $\Rz$ with $c=\zeta\notin E$: $\dim_xX=p_V-1$ for all $x\in X$. In particular,
\begin{equation}
\label{eqn:proof:lem:C:1}
\dim C=p_V-1, \text{ for all $C\in\Cc(X)$.}
\end{equation}
\medskip
Let $T=V\cap\pi^{-1}(\ZZ(\bW))$. By Lemma~\ref{lem:K} Part~\eqref{itemlabel:lem:K:b}, $\bW\not\equiv 0$ on $\prod_jH_j$, so $T$ is a proper closed subset of the irreducible variety $V$; hence $\dim T\le p_V-1$ and $\dim T(\R')\le p_V-1$. By the fiber-dimension formula for semi-algebraic maps, the set $E'$ of $c\in\R'$ with $\dim\bigl(T(\R')\cap G^{-1}(c)\bigr)=\dim T(\R')$ is finite; transferring 
this statement to \(\R'\langle\zeta\rangle\) as above, $\dim\bigl(X\cap T(\Rz)\bigr)\le p_V-2$. 
Using \eqref{eqn:proof:lem:C:1}, no semi-algebraically connected component of $X$ is contained in $T(\Rz)$. Setting $X^\circ=X\setminus T(\Rz)=X\cap\pi^{-1}\{\bW\ne0\}$, every semi-algebraically connected component of $X$ contains a semi-algebraically connected component of $X^\circ$, so $b_0(X)\le b_0(X^\circ)$.

If $x\in X^\circ$ and $y=\pi(x)$, then $y$ is a real point of $\prod_jH_j$ with $\bW(y)\ne0$, $x=\rho(y)$, and $\widehat G(y)/\bW(y)^e=G(x)=\zeta$, so $y\in Y$. Conversely, for $y\in Y$, $\rho(y)$ is a real point of $V$ (Lemma~\ref{lem:K}, extended) with $G(\rho(y))=\zeta$, so $\rho(y)\in X^\circ$. Thus $\pi:X^\circ\to Y$ is a bijection, continuous with continuous inverse $\rho$, i.e.\ a semi-algebraic homeomorphism.

$Y$ is semi-algebraically homeomorphic to the real algebraic set
\[
Y'=\bigl\{(y,\tau)\in\Rz^{M+1}\; \mid \; F_j(y^{(j)})=0, 1\leq j\leq k, \widehat G(y)-\zeta\bW(y)^e=0,\tau\bW(y)-1=0\bigr\},
\]
defined by polynomials of degree $\le keD$. 

By the Ole\u{\i}nik--Petrovski\u{\i} bound 
 (see for example \cite[Proposition 7.28]{BPRbook2}),
 applied to the sum of squares of the defining polynomials, of degree $\le 2keD$),
\[
b_0(Y')\ \le\ 2keD\,(4keD-1)^{M}\ \le\ (4keD)^{k(p+1)+1}.
\]
Combining, $b_0(X)\le b_0(X^\circ)=b_0(Y)=b_0(Y')$, as claimed.
\end{proof}

\begin{proof}[Proof of Proposition~\ref{prop:A}]
First observe that
the components with empty real locus may be discarded from 
$\Delta(\R)$
and any component with 
$\dim_x(\Delta_i(\R)) = \dim \Delta_i$ at some real point $x$ is automatically defined over $\R$.
This is because the full-dimensional real locus is Zariski dense in \(\Delta_i\); therefore complex conjugation fixes \(\Delta_i\).

Write $\Delta=\Delta_1\cup\dots\cup\Delta_m$ with $m\le D$. Then $\Delta^k=\bigcup_{\bi\in[m]^k}V_\bi$ with $V_\bi=\Delta_{i_1}\times\cdots\times\Delta_{i_k}$, and every connected component of $\RR(\sigma,\Delta(\R)^k)$ contains a connected component of some $\RR(\sigma,V_\bi(\R))$. It therefore suffices to bound $\sum_\sigma b_0(\RR(\sigma,V_\bi(\R)))$ for a fixed $\bi$ and multiply by $m^k\le D^k$.

Each $V_\bi$ is an irreducible real variety of dimension $p_V\le kp$ whose real points have constant local dimension $p_V$. 
Proposition~\ref{prop:master}
(the master bound), applied with $V:=V_\bi\subset\C^{Nk}$ and $\mathcal{P}:=\mathcal{Q}$, gives
\[
\sum_{\sigma\in\{0,\pm1\}^\mathcal{Q}} b_0\bigl(\RR(\sigma,V_\bi(\R))\bigr)\ \leq A+B+C,
\]
where $C=b_0(V_\bi(\R))$, and $A,B$ are sums of at most
\[
\sum_{j\le kp}4^j\binom{s}{j}\ \le\ \tfrac43(4s)^{kp}
\]
quantities of the form $b_0\bigl(\ZZ(G-\zeta,V_\bi(\Rt\langle\zeta\rangle))\bigr)$, where $G=Q_{I,\alpha}$ or $Q_{I,\alpha,0}$ is a polynomial of degree $\leq 2d$ (for $d \geq 2$) over the real closed field $\Rt=\R\langle\gamma_0,\dots,\gamma_s,\varepsilon,\delta\rangle$ and $\zeta$ is infinitesimal over $\Rt$.

\emph{The term $C$.} Since $V_\bi(\R)=\prod_j\Delta_{i_j}(\R)$, we have $b_0(V_\bi(\R))=\prod_jb_0(\Delta_{i_j}(\R))$. Proposition~\ref{prop:master:C} of the paper applied to each factor gives
\[
C\ \le\ \prod_{j=1}^k 14\,D_{i_j}\,(4D_{i_j}-1)^{p_{i_j}}\ \le\ (16D)^{k(p+1)}.
\]

\emph{The terms $A,B$.} By Lemma~\ref{lem:C} with $\R'=\Rt$ and $e=2d$, each summand is at most $(8kdD)^{k(p+1)+1}$.

Hence
\begin{align*}
\sum_\sigma b_0\bigl(\RR(\sigma,\Delta(\R)^k)\bigr) \leq& \; D^k\Bigl((16D)^{k(p+1)}+3\,(4s)^{kp}\,(8kdD)^{k(p+1)+1}\Bigr) \\
\leq& \; (16\,s\,k\,d\,D)^{3k(p+1)}
\end{align*}

the last inequality being a routine comparison of exponents (each factor $D^k$, $(16D)^{k(p+1)}$, $3(4s)^{kp}$, $(8kdD)^{k(p+1)+1}$ is at most $(16skdD)^{k(p+1)}$, and there are at most three of them in each term).
\end{proof}

\bibliographystyle{amsplain}
\bibliography{master}
\end{document}